\documentclass[10pt]{amsart}
\usepackage[margin=1.in]{geometry}
\usepackage{colortbl}
\newtheorem{theorem}{Theorem}[section]
\newtheorem{lemma}[theorem]{Lemma}
\newtheorem{proposition}{Proposition}[section]

\theoremstyle{definition}
\newtheorem{definition}[theorem]{Definition}

\theoremstyle{remark}
\newtheorem{remark}[theorem]{Remark}

\numberwithin{equation}{section}

\usepackage{graphicx}
\usepackage{subcaption}
\usepackage{multirow}
\usepackage{bbm}
\usepackage[bottom]{footmisc}

\pdfoutput=1
\usepackage[pagewise]{lineno}
\usepackage{enumitem}
\theoremstyle{definition}

\usepackage{xcolor}
\usepackage[export]{adjustbox}
\usepackage{caption}
\usepackage{hyperref}
\usepackage{epstopdf}
\usepackage{epsfig}

\begin{document}

\title[Convergence estimates of a SL scheme for ES-BGK model for polyatomic molecules]{Convergence estimates of a semi-Lagrangian scheme for the ellipsoidal BGK model for polyatomic molecules}

\author[S. Boscarino]{Sebastiano Boscarino}
\address{Sebastiano Boscarino\\
Department of Mathematics and Computer Science  \\
University of Catania\\
95125 Catania, Italy} \email{boscarino@dmi.unict.it}

\author[S.-Y. Cho]{Seung-Yeon Cho}
\address{Seung-Yeon Cho\\
Department of Mathematics and Computer Science  \\
University of Catania\\
95125 Catania, Italy}
\email{chosy89@skku.edu}

\author[G. Russo]{Giovanni Russo}
\address{Giovanni Russo\\
Department of Mathematics and Computer Science  \\
University of Catania\\
95125 Catania, Italy} \email{russo@dmi.unict.it}

\author[S.-B. Yun]{Seok-Bae Yun}
\address{Seok-Bae Yun\\
Department of Mathematics\\ 
Sungkyunkwan University\\
 Suwon 440-746, Republic of Korea}
\email{sbyun01@skku.edu}

\begin{abstract}
%

In this paper, we propose a new semi-Lagrangian scheme for the  
polyatomic ellipsoidal BGK model. In order to avoid time step 
restrictions coming from convection term and small Knudsen number, we combine a semi-Lagrangian approach for the convection term with an implicit treatment for the relaxation term.
We show how to explicitly solve the implicit step, thus obtaining an efficient and stable scheme for any Knudsen number. 
We also derive an explicit error estimate on the convergence of the proposed scheme for every fixed value of the Knudsen number.

\end{abstract}
\maketitle

\footnotetext{}

\section{Introduction}
\subsection{Polyatomic ES-BGK model}

The BGK model \cite{BGK} has been popularly employed for various flow problems of rarefied gas dynamics  in place of the Boltzmann equation since it reproduces the dynamics of the Boltzmann equation in a reliable manner at much lower computational cost. The importance of developing polyatomic versions of the BGK model 
has been recognized soon after the inception of the model - which is very natural since most of the gas molecules consists of several atoms - and the several attempts to derive polyatomic version of the BGK model have been proposed in the literature. The polyatomic generalization of the BGK model can be realized in various manners such as the introduction of new variables describing the internal energy due to the inner configuration of the molecules \cite{ALPP,BIP}, vibrational excitation \cite{BDMMMM}
, and reformulation into the gas mixture framework \cite{KPP,P}.
In this paper, we are interested in the polyatomic BGK model obtained from the so called ellipsoidal BGK model \cite{ALPP,BS,H} (Polyatomic ES-BGK model):
\begin{align}\label{A-1}
\begin{split}
\frac{\partial{f}}{\partial{t}} + v \cdot \nabla_x{f} &= \frac{A_{\nu,\theta}}{\kappa}\left(\mathcal{M}_{\nu,\theta}(f)-f\right),\cr
f(x,v,0,I) &= f_0(x,v,I).
\end{split}
\end{align}
The velocity-energy distribution function $f(x,v,t,I)$ represents the number density of particles in the phase space.
For simplicity, we assumed periodic boundary condition in $d$-dimensional space. Without loss of generality, the length of the domain is assumed to be one. The parameter $I \in \mathbb{R}_+$ is related to internal energy $\varepsilon$ due to rotation and vibration $\varepsilon(I)=I^{\frac{2}{\delta}}$, where $\delta>0$ represents the number of degrees of freedom for the internal motion of the molecules such as the rotation and vibration. Our independent variables $x$ and $v$  belong to phase space $(x,v)\in\mathbb{T}^d\times \mathbb{R}^3$, with $\mathbb{T}^d\equiv \mathbb{R}^d/\mathbb{Z}^d$, and $t\geq0$ denotes the time.
The Knudsen number $\kappa>0$ is the ratio between the mean free path of the gas molecules and the macroscopic length scale of the problem. 
We consider a collision frequency $A_{\nu,\theta}:= 1/ (1 - \nu + \nu\theta)$, for  $0<\theta \leq 1$ and $-\frac{1}{2}< \nu < 1$. The two parameters can be chosen to fit Prandtl number and transport coefficients computed by Chapmann-Enskog expansion of the equation Boltzmann equation. 
The polyatomic Gaussian $\mathcal{M}_{\nu,\theta}(f)$ is given by

\begin{align}\label{conti eg}
	\mathcal{M}_{\nu,\theta}(f):=\frac{\rho\Lambda_\delta}{\sqrt{\det\left(2 \pi \mathcal{T}_{\nu,\theta} \right)}(T_\theta)^\frac{\delta}{2}}\exp \left({-\frac{(v-U(x,t))^{\top}\mathcal{T}_{\nu,\theta}^{-1}(v-U(x,t))}{2} -\frac{I^{\frac{2}{\delta}}}{T_\theta}}\right),
\end{align}
where $\Lambda_{\delta}$ is a normalizing constant defined by
\[
\Lambda_{\delta}^{-1}:=\int_{\mathbb{R}_+} e^{-I^{\frac{2}{\delta}}}dI.
\]
The macroscopic local density $\rho(x,t)$, bulk velocity $U(x,t)$, stress tensor $\Theta(x, t)$ and internal
energy $E_\delta(x, t)$ are defined as follows:
\begin{align*}
\begin{split}
\rho(x,t)&:= \int_{\mathbb{R}^3\times\mathbb{R}_+}f(x,v,t,I)dvdI,\cr
\rho(x,t) U(x,t) &:= \int_{\mathbb{R}^3\times\mathbb{R}_+}vf(x,v,t,I)dvdI,\cr
\rho(x,t) \Theta(x,t) &:= \int_{\mathbb{R}^3\times\mathbb{R}_+}\left(v-U(x,t)\right)\otimes\left(v-U(x,t)\right) f(x,v,t,I)dvdI,\cr
E_\delta(x,t) &:= \int_{\mathbb{R}^3\times\mathbb{R}_+}\left(\frac{1}{2}|v-U(x,t)|^2 + I^{\frac{2}{\delta}}\right)f(x,v,t,I)dvdI.
\end{split}
\end{align*}
The internal energy $E_{\delta}$ consists of the translational energy $E_{tr}$ and the non-translational energy $E_{I,\delta}$:
\begin{align*}
E_{tr} &:= \int_{\mathbb{R}^3\times\mathbb{R}_+}\frac{1}{2}|v-U(x,t)|^2f(x,v,t,I) dvdI,\cr
E_{I,\delta} &:= \int_{\mathbb{R}^3\times\mathbb{R}_+} I^{\frac{2}{\delta}}f(x,v,t,I)dvdI.
\end{align*}
The corresponding temperatures
$T_\delta$, $T_{tr}$ and $T_{I,\delta}$ are defined by
\[
E_\delta =: \frac{3+\delta}{2}\rho T_\delta, \quad E_{tr} =: \frac{3}{2}\rho T_{tr}, \quad E_{I,\delta} =: \frac{\delta}{2}\rho T_{I,\delta}.
\]
Note that $T_\delta$ is the convex combination of $T_{tr}$ and $T_{I,\delta}$:
\begin{align*}
T_\delta = \frac{3}{3+\delta} T_{tr} + \frac{\delta}{3+\delta}T_{I,\delta}.
\end{align*}
We also define the relaxation temperature $T_\theta$ and the temperature tensor $\mathcal{T}_{\nu,\theta}$ as follows:
\begin{align*}
T_\theta &= \theta T_\delta + (1-\theta)T_{I,\delta},\cr
\mathcal{T}_{\nu,\theta}&=\theta T_\delta Id + (1-\theta)\big\{(1-\nu)T_{tr}Id + \nu \Theta \big\}.
\end{align*}
where $Id$ is a $3\times 3$ identity matrix. The polyatomic relaxation operator has five-dimensional collision invariants:
\begin{align*} 
	 \int_{\mathbb{T}^d\times\mathbb{R}^3\times\mathbb{R}_+} \big(\mathcal{M}_{\nu,\theta}(f)- f\big) \begin{pmatrix}
	 	1\\v\\\frac{1}{2}|v|^2+I^{\frac{2}{\delta}}
	 \end{pmatrix} dxdvdI=0,
\end{align*}
so that the conservation laws hold for mass, momentum and energy:
\[
\frac{d}{dt}  \int_{\mathbb{T}^d\times\mathbb{R}^3\times\mathbb{R}_+} f \phi(v,I) dxdvdI=0.
\]
The celebrated H-theorem was first verified in \cite{ALPP} (See also \cite{BS2000,BS,PY1,Y})
\[
\frac{d}{dt}\int_{\mathbb{R}^3\times\mathbb{R}_+}f\ln{f}dvdI =  \int_{\mathbb{R}^3\times\mathbb{R}_+} \left(\mathcal{M}_{\nu,\theta}(f)-f\right) \ln{f} dvdI   \leq 0.
\]

We note that this model reduces to the monatomic ES-BGK model \cite{H} when $\theta = 0$. On the other hand, if we take $\nu = \theta = 0$ and integrate both sides of \eqref{A-1} against $I$, the original BGK model is recovered \cite{BGK}.
It is also interesting that there is a dichotomy in the time asymptotic state of $f$ depending on $\theta$ (see \cite{PY1}). For $0<\theta\leq 1$, $f$ converges to $\mathcal{M}_{0,1}(f)$:
\begin{align*}
	\mathcal{M}_{0,1}(f):=\frac{\rho\Lambda_\delta}{\left(2 \pi T_{\delta} \right)^\frac{3}{2}(T_\delta)^\frac{\delta}{2}}\exp \left({-\frac{|v-U(x,t)|^2}{2T_{\delta}} -\frac{I^{\frac{2}{\delta}}}{T_{\delta}}}\right),
\end{align*}
while if $\theta=0$, its time asymptotic limit is the isothermal equalibrium $\mathcal{M}_{0,0}(f)$:
\begin{align*}
	\mathcal{M}_{0,0}(f):=\frac{\rho\Lambda_\delta}{\left(2 \pi T_{tr} \right)^\frac{3}{2}(T_{I,\delta})^\frac{\delta}{2}}\exp \left({-\frac{|v-U(x,t)|^2}{2T_{tr}} -\frac{I^{\frac{2}{\delta}}}{T_{I,\delta}}}\right).
\end{align*}



\subsection{Implicit semi-Lagrangian scheme}
Several methods have been adopted for numerical solutions of \eqref{A-1}. In \cite{KA,KKA}, the authors adopted iterative schemes to find the steady state solutions. When dealing with time-dependent problems, explicit schemes can be adopted if the Knudsen number is not too small \cite{ABLP,KAG}.  
On the other hand, if one is interested in small value of $\kappa$, then an implicit treatment of collision term is necessary in order to avoid excessive restrictions on the time step. Splitting schemes can be used in which an explicit convection step is followed by an implicit relaxation step \cite{CL}. Because during the relaxation step mass momentum and energy are constant, the solution of the implicit step is relatively easy. However, splitting schemes have the drawback that for small Knudsen number they are restricted to the first order accuracy in time \cite{CJR,J}. Accuracy can be improved for small Knudsen number using implicit explicit Runge-Kutta schemes \cite{BIP}. In this paper, the authors use an Eulerian framework in which convection terms are treated explicitly and collision term is treated implicitly. The drawback of Eulerian schemes is the CFL-type time step restriction $\left|v\frac{\Delta t}{\Delta x}\right|<1$ imposed by the convection term.
To overcome these difficulties, we propose a semi-Lagrangian method with an implicit treatment of the relaxation term of the following form: 
\begin{align}\label{implicit form}
	\frac{f_{i,j,k}^{n+1}-\tilde{f}_{i,j,k}^n}{\Delta t} = \frac{A_{\nu,\theta}}{\kappa} \left(\mathcal{M}_{\nu,\theta}(f_{i,j,k}^{n+1})-f_{i,j,k}^{n+1}\right),
\end{align}	
where $f_{i,j,k}^{n+1}$ is the discrete solution of the scheme, $\tilde{f}_{i,j,k}^{n}$ is the approximation of the discrete solution on the foot of characteristic,
and $\mathcal{M}_{\nu,\theta}(f_{i,j,k}^{n+1})$ denotes the numerical polyatomic ellipsoidal Gaussian (See Section 2 for precise definitions.)
However, this implicit scheme requires to solve 
non-linear systems.

To overcome this difficulty, we observe that the polyatomic ellipsoidal Gaussian constructed from  $f_{i,j,k}^{n+1}$ in \eqref{implicit form}
can be replaced by the polyatomic ellipsoidal Gaussian constructed form $\tilde{f}_{i,j,k}^{n}$ up to small error, which making the equation solvable as
\begin{align*}
	f_{i,j,k}^{n+1} =\frac{\kappa\tilde{f}_{i,j,k}^n +   A_{\nu,\theta} \Delta t \mathcal{M}_{\nu,\theta}(\tilde{f}_{i,j,k}^{n})}{\kappa + A_{\nu,\theta} \Delta t}.
\end{align*}
Note that the proposed scheme for the polyatomic ES-BGK model reduces to the semi-Lagrangian scheme for monatomic BGK model in \cite{GRS,RS,SP} and semi-Lagrangian scheme for monatomic ES-BGK model \cite{RY} by taking appropriate values of $\nu$ and $\theta$ and integrating it over $I$ variable. 


The main result of this paper is the derivation of the error estimate based on $L_q^\infty$-norm (see notation in section \ref{notation}), which is stated in Theorem \ref{T.3.2} as follows:
\begin{align*}
\|f^{N_t} - f(T^f)\|_{L_q^\infty} &\leq  C \bigg( \frac{(\Delta x)^2}{\Delta t} + (\Delta x)^2 + \Delta v + \Delta I + \Delta t\bigg),
\end{align*}	 
where $C$ is a constant depending on $T^f,q,\delta,\kappa,\theta,\nu,\Delta t$, but can be uniformly bounded regardless of $\Delta t>0$. 
The main ingredient of the convergence proof is the establishment of the following uniform stability estimate of the discrete solution (see section \ref{stability sec}): 
\begin{align*}
	C_0^1e^{-\frac{A_{\nu,\theta}}{\kappa} T^f} e^{-C_0^2(|v_j|^a +I_k^b)}\leq\tilde{f}_{i,j,k}^n   \leq e^{\frac{ C_{\mathcal{M}}A_{\nu,\theta}  }{\kappa + A_{\nu,\theta} \Delta t}T^f} \|f_0\|_{L_q^\infty}(1+|v_j|^2)^{-q/2}.
\end{align*}
We note that, unlike most of numerical stability estimates, the uniform lower bound is important since it is crucially used to prove that the polyatomic temperature never vanishes
(see Lemma \ref{L.5.14}):
\begin{align*}
	(\tilde{T}_\delta)_i^n \geq  \bigg( \frac{1}{2}\frac{\bar{C}_{a,b} C_0^1}{C_\delta \|f_0\|_{L_{q}^{\infty}}}e^{-\big(\frac{1}{\kappa}  +\frac{ (C_{\mathcal{M}}-1) }{\kappa + A_{\nu,\theta} \Delta t}\big)T^f}\bigg)^{\frac{2}{3+\delta}},
\end{align*}
so that the discrete polyatomic ellipsoidal Gaussian never degenerates into Dirac delta.




We close this subsection with a brief review on implicit semi-Lagrangian schemes for BGK models. 
In \cite{SP}, high order semi-Lagrangian methods were constructed using diagonally implicit Runge-Kutta schemes \cite{PR} and high order non-oscillatory spatial reconstruction \cite{CRG}. Owing to the L-stability property of time discretization, the resulting schemes enable one to use a large time step even in the fluid regime. In \cite{GRS}, multi-step time discretization such as BDF methods were adopted in the semi-Lagrangian framework. The performance of such methods was verified through boundary value problems in \cite{GRS2,RF}. In \cite{BCRY}, such semi-Lagrangian schemes were employed as a predictor scheme corrected it by a conservative procedure to obtain an exactly conservative scheme at the discrete level. We also refer to \cite{GRS3} for semi-Lagrangian methods applied to gas mixtures and reactive flows.

The convergence estimate for the original monatomic BGK model was investigated  in \cite{RSY}. The argument has been simplified and applied to the more complicate case of the ES-BGK model \cite{RY}, which is the main motivation of the current work. These two results seem to be the only available convergence estimates for fully discrete schemes for spatially inhomogeneous collisional kinetic equations.


The semi-Lagrangian methods have been widely used also for the numerical solutions of Vlasov-type equations \cite{CMS,FSB,QC,QS,SRB,TRQ}. We refer to \cite{DP} for a nice survey on numerical schemes for kinetic equations.
\newline

\subsection{Notation}\label{notation}
Throughout this paper, we use the following notations :
\begin{itemize}
	\item $C$ denotes a constant which can be explicitly computable.
	\item $C_{a,b,\dots}$, $\bar{C}_{a,b,\dots}$ denote a constant that depend on $a, b,\dots$.
	\item We use lower indices $i,j,k$ for space, velocity, internal energy variables and an upper index $n$ for time variable, respectively.	
	\item We write the velocity vector $v$ as $v\equiv(v^1,v^2,v^3)$.
	\item $T^f$ denotes the final time of the numerical experiment.
	\item 
		The relation $A\leq B$ for $3\times 3$ matrices $A$ and $B$ means that $B-A$ is positive definite, i.e., $k^{\top} (B-A) k \geq 0$  for all $k \equiv (k^1,k^2,k^3)^{\top}\in \mathbb{R}^3$. 
	\item For $N,q \in \mathbb{N}$, the weighted $L^\infty$-Sobolev norm for continuous solution is defined by
	\begin{align*}
		\|f(t)\|_{L_q^\infty} &:= \sup_{x,v,I}|f(x,v,t,I)(1 + |v|^2 + I^{\frac{2}{\delta}})^{\frac{q}{2}}|, \cr
		\|f(t)\|_{{N,q}}^{\infty} &:= \sum_{|\alpha|+|\beta|+\gamma\leq
			N}\sup_{x,v,I}|\partial(\alpha,\beta,\gamma)f(x,v,t,I)(1+|v|+I^{\frac{1}{\delta}})^q|,
	\end{align*}
	where $\alpha, \beta, \gamma \in \mathbb{Z}_+ \times \mathbb{Z}_+^3 \times \mathbb{Z}_+$, and the differential operator $\partial(\alpha,\beta,\gamma)$ stands for $\partial_x^\alpha \partial_v^\beta \partial_I^\gamma$. Indeed, $\|f(t)\|_{L_q^\infty}=\|f(t)\|_{{0,q}}^{\infty}$.
	\item The weighted $L_q^\infty$-Sobolev norm for discrete solution is defined by
	\begin{align*}
		\|f^n\|_{L_q^\infty} &:= \sup_{i,j,k}|f_{i,j,k}^n(1 + |v_j|^2 + I_k^{\frac{2}{\delta}})^{\frac{q}{2}}|,
	\end{align*}
	where $f_{i,j,k}^n$ is a numerical solution of $f(x_i, v_j, t^n, I_k)$.
	\item 
	To measure the distance between discrete and continuous solutions, we use the following supremum on grid points:
	\begin{align*}
		\|f^{n} - f(t^{n})\|_{L_q^\infty}= \sup_{i,j,k}|f_{i,j,k}^n-f(x_i,v_j,t^n,I_k)|(1 + |v_j|^2 + I_k^{\frac{2}{\delta}})^{\frac{q}{2}}.
	\end{align*}
\end{itemize}

	This paper is organized as follows: In Section 2, we derive a first order semi-Lagrangian scheme for the polyatomic ES-BGK model. Section 3 is devoted to the statement of the main result of this paper. In the following Section 4 and 5, we present several technical estimates on the discrete solution and its macroscopic variables. In Section 6, we rewrite the polyatomic ES-BGK model \eqref{A-1} for the easy comparison of continuous and discrete solution. Then, in Section 7, the difference between the continuous and discrete Gaussians is estimated. Finally, in Section 8, we prove our main theorem.




\begin{section}{Description of the numerical scheme}
	
	\subsection{Discretization}
	For velocity variables, we take same mesh spacing $\Delta v$ in all directions, while, for the internal energy variable, we use a uniform mesh of size $\Delta I$. For space, one-dimensional periodic unit interval is considered with a uniform mesh $\Delta x$. We assume a fixed time step $\Delta t$. Then,
	\begin{align*}
		t^n&=n\Delta t, \quad n=0,1,\dots,N_t,\cr
		x_i&=i \Delta x, \quad i=0,\pm1,\dots,\pm N_x, \pm (N_x+1), \dots
	\end{align*}
	where $N_t \Delta t= T^f$, $N_x \Delta x = 1$. 
	Note that we consider space discretization on the whole spatial domain and then impose periodicity for technical simplicity in the convergence proof.

	For velocity and internal energy variables, we use
\begin{align*}
	v_j\equiv(v_j^1,v_j^2,v_j^3)&\equiv(j_1\Delta v,j_2\Delta v,j_3\Delta v),\quad  (j_1,j_2,j_3) \in  \mathbb{Z}^3,\cr
	I_k&=k\Delta I, \quad k=0, 1, 2, \dots.
\end{align*}

To be more concise, we introduce the following notations:
	\begin{definition}\label{B-2}
		\begin{enumerate}
			\item Let $x(i,j):=x_i-v_j^1 \Delta t$ and $s\equiv s(i,j)$ be the index such that 
			\begin{align*}
				x(i,j) \in [x_s, x_{s + 1}).
			\end{align*}
			\item Let $\tilde{f}_{i,j,k}^n$ be the linear interpolation of $f_{s,j,k}^n$ and $f_{s+1,j,k}^n$ on $x(i,j)$ at time $t^n$:
			\begin{align*}
				\tilde{f}_{i,j,k}^n:= a_j f_{s,j,k}^n + (1-a_j) f_{s+1,j,k}^n,
			\end{align*}		
		where $a_j:=(x_{s+1}-x(i,j))/\Delta x$. Note that there is only $j$ dependence on $a_j$ due to the use of uniform grid in space variable.

		\end{enumerate}
	\end{definition}
\subsection{Implicit semi-Lagrangian scheme}
Our scheme reads
\begin{align}\label{main scheme}
	f_{i,j,k}^{n+1} =\frac{\kappa\tilde{f}_{i,j,k}^n +   A_{\nu,\theta} \Delta t \mathcal{M}_{\nu,\theta}(\tilde{f}_{i,j,k}^{n})}{\kappa + A_{\nu,\theta} \Delta t}.
\end{align}
where $A_{\nu,\theta}:= 1/ (1 - \nu + \nu\theta)$ for  $0<\theta \leq 1$ and $-\frac{1}{2}< \nu < 1$.
\begin{align*}
	\tilde{f}_{i,j,k}^n:= a_j f_{s,j,k}^n + (1-a_j) f_{s+1,j,k}^n.
\end{align*}
Note that $a_j$ and $s$ are defined in \eqref{B-2}. The discrete ellipsoidal Gaussian based on $\{\tilde{f}_{i,j,k}^n\}$ is given by
\begin{align}\label{approx M}
\begin{split}
		\mathcal{M}_{\nu,\theta}(\tilde{f}_{i,j,k}^{n}) :&= \frac{\tilde{\rho}_i^n\Lambda_\delta}{\sqrt{\det\left(2 \pi (\tilde{\mathcal{T}}_{\nu,\theta})_i^n \right)}\big((\tilde{T}_\theta)_i^n\big)^\frac{\delta}{2}} \exp \left({-\frac{(v_j-\tilde{U}_i^n)^{\top} ((\tilde{\mathcal{T}}_{\nu,\theta})_i^n )^{-1}(v_j-\tilde{U}_i^n)}{2} -\frac{I_k^{\frac{2}{\delta}}}{(\tilde{T}_\theta)_i^n}}\right).
\end{split}
\end{align}
where $\Lambda_{\delta}$ is a normalizing constant defined by
\[
\Lambda_{\delta}^{-1}:=\int_{\mathbb{R}_+} e^{-I^{\frac{2}{\delta}}}dI.
\]
The macroscopic variables computed from $\{\tilde{f}_{i,j,k}^n\}$ are defined as follows:

\begin{align}\label{tilde TTT}
	\begin{array}{lll}
		\text{$\bullet$ Mass:} &\qquad\qquad\qquad  &\qquad\qquad \\
			&\tilde{\rho}_i^n:= \sum_{j,k} \tilde{f}_{i,j,k}^n (\Delta v)^3 \Delta I.&  \\
		\text{$\bullet$ Momentum:} &  & \\
			&\tilde{\rho}_i^n\tilde{U}_i^n:= \sum_{j,k} \tilde{f}_{i,j,k}^n v_j(\Delta v)^3 \Delta I.&\\
		\text{$\bullet$ Stress tensor:} &  & \\
			&\tilde{\rho}_i^n\tilde{\Theta}_i^n = \sum_{j,k} \tilde{f}_{i,j,k}^{n}(v_j-\tilde{U}_i^{n}) \otimes (v_j-\tilde{U}_i^{n})(\Delta v)^3 \Delta I.&\\
		\text{$\bullet$ Polyatomic temperature:} &  & \\
			&(\tilde{T}_\delta)_i^n = \frac{3}{3+\delta} (\tilde{T}_{tr})_i^n + \frac{\delta}{3+\delta}(\tilde{T}_{I,\delta})_i^n,&\\
		\text{where} &  & \\
		&(\tilde{T}_{tr})_i^n:= \frac{2}{3}\frac{1}{\tilde{\rho}_i^n}\sum_{j,k} \tilde{f}_{i,j,k}^{n} \frac{|v_j-\tilde{U}_i^n|^2}{2}(\Delta v)^3 \Delta I,&\\
		&(\tilde{T}_{I,\delta})_i^n:=  \frac{2}{\delta}\frac{1}{\tilde{\rho}_i^n}\sum_{j,k} \tilde{f}_{i,j,k}^{n} I_k^{\frac{2}{\delta}}  (\Delta v)^3 \Delta I.&\\
		\text{$\bullet$ Relaxation temperature:} &  & \\
		&(\tilde{T}_\theta)_i^n:= \theta (\tilde{T}_\delta)_i^n + (1-\theta)(\tilde{T}_{I,\delta})_i^n.&\\
		\text{$\bullet$ Polyatomic temperature tensor:} &  & \\
		&(\tilde{\mathcal{T}}_{\nu,\theta})_i^n:=\lambda\theta (\tilde{T}_\delta)_i^n Id + \lambda(1-\theta)(1-\nu)(\tilde{T}_{tr})_i^nId  +(1-\theta)\bar{\nu}  \tilde{\Theta}_i^n,&\\
	\end{array}
\end{align}
For notational simplicity, we also introduce 
\begin{align*}
	\lambda \equiv \lambda(\nu,\theta,\kappa,\Delta t):= \frac{\kappa + A_{\nu,\theta}\Delta t}{\Delta t + \kappa}, \quad \bar{\nu} \equiv \bar{\nu}(\nu,\kappa,\Delta t):= \frac{\kappa \nu }{\Delta t + \kappa}.
\end{align*}
Since the initial step can be taken to be arbitrarily correct, we assume for technical simplicity that
the initial step is approximated as follows to guarantee that  no error arises in the initial approximation of the initial data:

\begin{align}\label{no initial err}
	\begin{array}{lll}
		\text{$\bullet$ initial distribution:} &\qquad\qquad\qquad  &\qquad\qquad \\
		&f_{i,j,k}^0=f_0(x_i,v_j,I_k), \quad
		\tilde{f}_{i,j,k}^0=f_0(x_i-v_j^1\Delta t,v_j,I_k).&  \\
		\text{$\bullet$ Mass:} &  & \\
		&\tilde{\rho}_i^0= \int_{\mathbb{R}^3\times\mathbb{R}_+}f_0(x_i-v^1\Delta t,v,I)dvdI, &\\
		\text{$\bullet$ Momentum:} &\qquad\qquad\qquad  &\qquad\qquad \\
		&\tilde{\rho}_i^n\tilde{U}_i^n:= \tilde{\rho}_i^0\tilde{U}_i^0 = \int_{\mathbb{R}^3\times\mathbb{R}_+}vf_0(x_i-v^1\Delta t,v,I)dvdI.&\\
		\text{$\bullet$ Stress tensor:} &  & \\
		&\tilde{\rho}_i^0\tilde{\Theta}_i^0 = \int_{\mathbb{R}^3\times\mathbb{R}_+}\left(v-\tilde{U}_i^0\right)\otimes\left(v-\tilde{U}_i^0\right) f_0(x_i-v^1\Delta t,v,I)dvdI.&\\
		\text{$\bullet$ Polyatomic temperature:} &  & \\
		&(\tilde{T}_\delta)_i^0 = \frac{3}{3+\delta} (\tilde{T}_{tr})_i^0 + \frac{\delta}{3+\delta}(\tilde{T}_{I,\delta})_i^0,&\\
		\text{where} &  & \\
		&(\tilde{T}_{tr})_i^0:= \frac{2}{3}\frac{1}{\tilde{\rho}_i^0}\int_{\mathbb{R}^3\times\mathbb{R}_+} \frac{|v-\tilde{U}_i^0|^2}{2} f_0(x_i-v^1\Delta t,v,I) dvdI,&\\
		&(\tilde{T}_{I,\delta})_i^0:=  \frac{2}{\delta}\frac{1}{\tilde{\rho}_i^n}\int_{\mathbb{R}^3\times\mathbb{R}_+} I^{\frac{2}{\delta}}  f_0(x_i-v^1\Delta t,v,I) dvdI.&\\
		\text{$\bullet$ Relaxation temperature:} &  & \\
		&(\tilde{T}_\theta)_i^0:= \theta (\tilde{T}_\delta)_i^0 + (1-\theta)(\tilde{T}_{I,\delta})_i^0.& \\
		&&\\
		\text{$\bullet$ Polyatomic temperature tensor:} &  & \\
		&(\tilde{\mathcal{T}}_{\nu,\theta})_i^0:=\lambda\theta (\tilde{T}_\delta)_i^0 Id + \lambda(1-\theta)(1-\nu)(\tilde{T}_{tr})_i^0Id  +(1-\theta)\bar{\nu}  \tilde{\Theta}_i^0.&\\
		\text{where} &  & \\
		&\lambda \equiv \lambda(\nu,\theta,\kappa,\Delta t):= \frac{\kappa + A_{\nu,\theta}\Delta t}{\Delta t + \kappa}, \quad \bar{\nu} \equiv \bar{\nu}(\nu,\kappa,\Delta t):= \frac{\kappa \nu }{\Delta t + \kappa}.&	
	\end{array}
\end{align}

	\subsection{Derivation of the first order scheme} 
	Now we consider how the scheme \eqref{main scheme} is derived. Throughout this paper, we focus on one-dimensional spatial domain $(d=1)$. We start from the backward characteristic of (\ref{A-1}):
	\begin{align}\label{B-3}
		\begin{split}
		\frac{df}{ds}&=\frac{A_{\nu,\theta}}{\kappa}\left(\mathcal{M}_{\nu,\theta}(f)-f\right),\cr
		\frac{dx}{ds}&=v_j^1,\cr 
		\frac{dv}{ds}&=\frac{dI}{ds}=0,
		\end{split}
	\end{align}
Here, one can easily have
\[\displaystyle x(s)\equiv x_i-v_j^1(t^{n+1}-s), \quad v(s)\equiv v_j, \quad I(s)\equiv I_k.\]
To solve (\ref{B-3}), considering the stiffness coming from $\kappa$, we apply the implicit Euler method:
\begin{align}\label{B-4}
	\frac{f_{i,j,k}^{n+1}-\tilde{f}_{i,j,k}^n}{\Delta t} = \frac{A_{\nu,\theta}}{\kappa} \left(\mathcal{M}_{\nu,\theta}(f_{i,j,k}^{n+1})-f_{i,j,k}^{n+1}\right).
\end{align}	
where the discrete ellipsoidal Gaussian is given by
\begin{align*}
	\mathcal{M}_{\nu,\theta}(f_{i,j,k}^n)&=\frac{\rho_i^n\Lambda_\delta}{\sqrt{\det\left(2 \pi (\mathcal{T}_{\nu,\theta})_i^n \right)}((T_\theta)_i^n)^\frac{\delta}{2}}\exp \left({-\frac{(v_j-U_i^n)^{\top}((\mathcal{T}_{\nu,\theta})_i^n)^{-1}(v_j-U_i^n)}{2} -\frac{I_k^{\frac{2}{\delta}}}{(T_\theta)_i^n}}\right)
\end{align*}
with the discrete normalizing factor: 
\begin{equation*}
\begin{split}
\Lambda_{\delta}^{-1}&=\sum_k e^{-I_k^{\frac{2}{\delta}}}\Delta I,\quad 
\end{split}
\end{equation*}

and the macroscopic fields are defined by
\begin{align*} 
	\begin{array}{lll}
		\text{$\bullet$ Mass:} &\qquad\qquad\qquad  &\qquad\qquad \\
		&\rho_i^n:= \sum_{j,k} f_{i,j,k}^n (\Delta v)^3 \Delta I.&  \\
		\text{$\bullet$ Momentum:} &  & \\
		&\rho_i^n U_i^n:= \sum_{j,k} f_{i,j,k}^n v_j(\Delta v)^3 \Delta I.&\\
		\text{$\bullet$ Stress tensor:} &  & \\
		&\rho_i^n\Theta_i^n = \sum_{j,k} f_{i,j,k}^{n}(v_j-U_i^{n}) \otimes (v_j-U_i^{n})(\Delta v)^3 \Delta I.&\\
		\text{$\bullet$ Polyatomic temperature:} &  & \\
		&(T_\delta)_i^n = \frac{3}{3+\delta} (T_{tr})_i^n + \frac{\delta}{3+\delta}(T_{I,\delta})_i^n,&\\
		\text{where} &  & \\
		&(T_{tr})_i^n:= \frac{2}{3}\frac{1}{\rho_i^n}\sum_{j,k} f_{i,j,k}^{n} \frac{|v_j-U_i^n|^2}{2}(\Delta v)^3 \Delta I,&\\
		&(T_{I,\delta})_i^n:=  \frac{2}{\delta}\frac{1}{\rho_i^n}\sum_{j,k} f_{i,j,k}^{n} I_k^{\frac{2}{\delta}}  (\Delta v)^3 \Delta I.&\\
		\text{$\bullet$ Relaxation temperature:} &  & \\
		&(T_\theta)_i^n:= \theta (T_\delta)_i^n + (1-\theta)(T_{I,\delta})_i^n.&\\
		\text{$\bullet$ Polyatomic temperature tensor:} &  & \\
		&(\mathcal{T}_{\nu,\theta})_i^n:=\theta (T_\delta)_i^n Id + (1-\theta)(1-\nu)(T_{tr})_i^nId  +(1-\theta)\nu  \Theta_i^n,&\\
	\end{array}
\end{align*}


We note that \eqref{B-4} involves high computational cost since it is implicit form. To transform this implicit scheme into an explicitly computable scheme with beneficial stability properties preserved, we adopt the argument developed in \cite{BCRY,GRS,PP,RY,SP} to our polyatomic setting.
	
We start with conservative quantities. We multiply both sides of \eqref{B-4} by collision invariants:
$$\phi_{j,k}:=\left(1,v_j,\frac{1}{2}|v_j|^2+I_k^{\frac{2}{\delta}}\right)$$ 
	and take a summation over $j,k$ to derive
	\begin{align*}
		\sum_{j,k}\left(f_{i,j,k}^{n+1}-\tilde{f}_{i,j,k}^n\right) \phi_{j,k} (\Delta v)^3 \Delta I = \sum_{j,k}\frac{A_{\nu,\theta}\Delta t}{\kappa} \left(\mathcal{M}_{\nu,\theta}(f_{i,j,k}^{n+1})-f_{i,j,k}^{n+1}\right) \phi_{j,k} (\Delta v)^3 \Delta I.
	\end{align*}	
	Since the right hand side vanishes for enough $v$ and $I$ nodes, we have
    \begin{align}\label{rhouE}
\begin{split}
	    	\rho_i^{n+1} &= \tilde{\rho}_i^n:= \sum_{j,k} \tilde{f}_{i,j,k}^n (\Delta v)^3 \Delta I, \cr
	U_i^{n+1} &= \tilde{U}_i^n:= \frac{1}{\tilde{\rho}_i^n}\sum_{j,k} \tilde{f}_{i,j,k}^n v_j(\Delta v)^3 \Delta I,\cr
	(E_{\delta})_i^{n+1} &= (\tilde{E}_{\delta})_i^n:=\sum_{j,k} \tilde{f}_{i,j,k}^{n} \left(\frac{|v_j-\tilde{U}_i^n|^2}{2} + I_k^{\frac{2}{\delta}}\right) (\Delta v)^3 \Delta I.
\end{split}
    \end{align}
Using this, we approximate $(T_\delta)_i^{n+1}$, $(T_{tr})_i^{n+1}$ and $(T_{I,\delta})_i^{n+1}$ as follows:

\begin{align}\label{TTT}
	\begin{split}
	(T_\delta)_i^{n+1}&= (\tilde{T}_{\delta})_i^n := \frac{2}{3+\delta}\frac{1}{\tilde{\rho}_i^n}\sum_{j,k} \tilde{f}_{i,j,k}^{n} \left(\frac{|v_j-\tilde{U}_i^n|^2}{2} + I_k^{\frac{2}{\delta}}\right) (\Delta v)^3 \Delta I\cr
	(T_{tr})_i^{n+1}&=\frac{2}{3}\frac{1}{\tilde{\rho}_i^n}\sum_{j,k} f_{i,j,k}^{n+1} \frac{|v_j-\tilde{U}_i^n|^2}{2}(\Delta v)^3 \Delta I\cr
	&\approx  \frac{2}{3}\frac{1}{\tilde{\rho}_i^n}\sum_{j,k} \tilde{f}_{i,j,k}^{n} \frac{|v_j-\tilde{U}_i^n|^2}{2}(\Delta v)^3 \Delta I,\cr
	&=:(\tilde{T}_{tr})_i^n,\cr
	(T_{I,\delta})_i^{n+1}&=\frac{2}{\delta}\frac{1}{\tilde{\rho}_i^n}\sum_{j,k} f_{i,j,k}^{n+1} I_k^{\frac{2}{\delta}}  (\Delta v)^3 \Delta I\cr
	&\approx  \frac{2}{\delta}\frac{1}{\tilde{\rho}_i^n}\sum_{j,k} \tilde{f}_{i,j,k}^{n} I_k^{\frac{2}{\delta}}  (\Delta v)^3 \Delta I\cr
	&=:(\tilde{T}_{I,\delta})_i^n .
	\end{split}
\end{align}
Note that the approximations for $(T_{tr})_i^{n+1}$ and $(T_{I,\delta})_i^{n+1}$ can be justified because we are considering a first order scheme.
Now, we turn to the approximation of the stress tensor $\Theta_i^{n+1}$. Although it is a non-conservative quantity, we can approximate it in a legitimate way as in \cite{RY}. For this, we introduce 
$$\xi_{ij}^{n+1}:=(v_j-U_i^{n+1}) \otimes (v_j-U_i^{n+1})$$ 
and multiply this to (\ref{B-4}) to derive 	
\begin{align}\label{2.5}
	\sum_{j,k} (f_{i,j,k}^{n+1}-\tilde{f}_{i,j,k}^{n})\xi_{ij}^{n+1} (\Delta v)^3 \Delta I=\sum_{j,k}\frac{A_{\nu,\theta} \Delta t}{\kappa} \left(\mathcal{M}_{\nu,\theta}(f_{i,j,k}^{n+1})-f_{i,j,k}^{n+1}\right)\xi_{ij}^{n+1}
	(\Delta v)^3 \Delta I.
\end{align}
Recalling the relation $U_i^{n+1} = \tilde{U}_i^n$, we obtain
\[
\xi_{ij}^{n+1}=(v_j-U_i^{n+1}) \otimes (v_j-U_i^{n+1}) = (v_j-\tilde{U}_i^n) \otimes (v_j-\tilde{U}_i^n) =: \tilde{\xi}_{ij}^{n}.
\]
This implies that the second term on the left in \eqref{2.5} becomes
\begin{align}\label{app 1}
	\sum_{j,k} \tilde{f}_{i,j,k}^{n}\xi_{ij}^{n+1}(\Delta v)^3 \Delta I = \sum_{j,k}\tilde{f}_{i,j,k}^{n}\tilde{\xi}_{ij}^{n}(\Delta v)^3 \Delta I= \tilde{\rho}_i^n \tilde{\Theta}_i^n,
\end{align}	
where $\tilde{\Theta}_i^n$ is defined by
\[
\tilde{\rho}_i^n\tilde{\Theta}_i^n = \sum_{j,k} \tilde{f}_{i,j,k}^{n}(v_j-\tilde{U}_i^{n}) \otimes (v_j-\tilde{U}_i^{n})(\Delta v)^3 \Delta I.
\]
On the other hand, the right hand side in \eqref{2.5} can be rewritten by
\begin{align}\label{app 2}
\begin{split}
	\sum_{j,k}\frac{A_{\nu,\theta} \Delta t}{\kappa}&\left(\mathcal{M}_{\nu,\theta}(f_{i,j,k}^{n+1})-f_{i,j,k}^{n+1}\right)\xi_{ij}^{n+1}
(\Delta v)^3 \Delta I \cr
&=\frac{A_{\nu,\theta} \Delta t}{\kappa}\left(\rho_i^{n+1}(\mathcal{T}_{\nu,\theta})_i^{n+1}-\rho_i^{n+1}\Theta_i^{n+1}\right)\cr
&=\frac{A_{\nu,\theta} \Delta t}{\kappa}\left(\rho_i^{n+1}\bigg[ \theta(T_\delta)_i^{n+1} Id + (1-\theta)\bigg\{(1-\nu)(T_{tr})_i^{n+1}Id + \nu \Theta_i^{n+1} \bigg\}\bigg]-\rho_i^{n+1}\Theta_i^{n+1}\right)\cr
&=\frac{A_{\nu,\theta} \Delta t}{\kappa}\bigg(\rho_i^{n+1}\bigg[ \theta(T_\delta)_i^{n+1}  + (1-\theta)(1-\nu)(T_{tr})_i^{n+1}\bigg]Id - \big\{1 -\nu + \nu\theta\big\}\rho_i^{n+1}\Theta_i^{n+1}\bigg)\cr
&=\frac{\Delta t}{\kappa}\rho_i^{n+1}\left[ A_{\nu,\theta}\theta(T_\delta)_i^{n+1} + (1- A_{\nu,\theta} \theta) (T_{tr})_i^{n+1} \right]Id  - \frac{\Delta t}{\kappa}\rho_i^{n+1}\Theta_i^{n+1}.
\end{split}
\end{align}
Then, we insert \eqref{app 1} and \eqref{app 2} into \eqref{2.5} to compute $\Theta_i^{n+1}$ as follows:
\begin{align*}
\Theta_i^{n+1} &= \frac{\Delta t \left[ A_{\nu,\theta}\theta(\tilde{T}_\delta)_i^n + (1- A_{\nu,\theta} \theta) (\tilde{T}_{tr})_i^n \right]Id  +\kappa \tilde{\Theta}_i^n}{\Delta t + \kappa}.
\end{align*}
Now, we use this and \eqref{TTT} to approximate the polyatomic stress tensor:
\begin{align*}
	(\mathcal{T}_{\nu,\theta})_i^{n+1}&=\theta (T_\delta)_i^{n+1} Id + (1-\theta)\bigg\{(1-\nu)(T_{tr})_i^{n+1}Id + \nu \Theta_i^{n+1} \bigg\}\cr
	&\approx \theta (\tilde{T}_\delta)_i^n Id + (1-\theta)(1-\nu)(\tilde{T}_{tr})_i^nId  \cr
	&\quad + (1-\theta)\nu \bigg[\frac{\Delta t \left[ A_{\nu,\theta}\theta(\tilde{T}_\delta)_i^n + (1- A_{\nu,\theta} \theta) (\tilde{T}_{tr})_i^n \right]Id  + \kappa\tilde{\Theta}_i^n}{\Delta t + \kappa}\bigg]\cr
	&= \bigg(\theta + \frac{(1-\theta)\nu \Delta tA_{\nu,\theta}\theta}{\Delta t + \kappa}\bigg) (\tilde{T}_\delta)_i^n Id\cr
	&\quad + \bigg((1-\theta)(1-\nu) + \frac{(1-\theta)\nu\Delta t(1- A_{\nu,\theta} \theta)}{\Delta t + \kappa}\bigg)(\tilde{T}_{tr})_i^nId  \cr
	&\quad + (1-\theta)\nu \frac{ \kappa\tilde{\Theta}_i^n}{\Delta t + \kappa}.
\end{align*}
We write it in a more compact manner
\begin{align}\label{B-5}
	(\mathcal{T}_{\nu,\theta})_i^{n+1}
	&\approx (\tilde{\mathcal{T}}_{\nu,\theta})_i^n :=\lambda\theta (\tilde{T}_\delta)_i^n Id + \lambda(1-\theta)(1-\nu)(\tilde{T}_{tr})_i^nId  +(1-\theta)\bar{\nu}  \tilde{\Theta}_i^n,
\end{align}
using
\begin{align*} 
	\lambda \equiv \lambda(\nu,\theta,\kappa,\Delta t):= \frac{\kappa + A_{\nu,\theta}\Delta t}{\Delta t + \kappa}, \quad \bar{\nu} \equiv \bar{\nu}(\nu,\kappa,\Delta t):= \frac{\kappa \nu }{\Delta t + \kappa}.
\end{align*}
Similarly, we approximate $(T_\theta)_i^{n+1}$ as follows:
\begin{align}\label{Ttheta}
	(T_\theta)_i^{n+1} \approx (\tilde{T}_\theta)_i^n:= \theta (\tilde{T}_\delta)_i^n + (1-\theta)(\tilde{T}_{I,\delta})_i^n.
\end{align}
In view of \eqref{rhouE}, \eqref{TTT}, \eqref{B-5}, \eqref{Ttheta}, we find that   $\mathcal{M}_{\nu,\theta}(f_{i,j,k}^{n+1})$ is legitimately replaced by $\mathcal{M}_{\nu,\theta}(\tilde{f}_{i,j,k}^{n})$: 
\begin{align*}
	\mathcal{M}_{\nu,\theta}(\tilde{f}_{i,j,k}^{n}):&= \frac{\tilde{\rho}_i^n\Lambda_\delta}{\sqrt{\det\left(2 \pi (\tilde{\mathcal{T}}_{\nu,\theta})_i^n \right)}\big((\tilde{T}_\theta)_i^n\big)^\frac{\delta}{2}} \exp \left({-\frac{(v_j-\tilde{U}_i^n)^{\top} ((\tilde{\mathcal{T}}_{\nu,\theta})_i^n )^{-1}(v_j-\tilde{U}_i^n)}{2} -\frac{I_k^{\frac{2}{\delta}}}{(\tilde{T}_\theta)_i^n}}\right).
\end{align*}

Finally, we substitue this into \eqref{B-4}, and solve for $f_{i,j,k}^{n+1}$ to get our scheme:
\begin{align}\label{B-6}
f_{i,j,k}^{n+1} =\frac{\kappa\tilde{f}_{i,j,k}^n +   A_{\nu,\theta} \Delta t \mathcal{M}_{\nu,\theta}(\tilde{f}_{i,j,k}^{n})}{\kappa + A_{\nu,\theta} \Delta t}.
\end{align}

\begin{remark}
	For $\theta=0$, after taking summation over $k$ in \eqref{B-6}, this scheme becomes the first order SL scheme for the monatomic ES-BGK model in \cite{RY}. For $\theta=\nu=0$, the scheme further reduces to the first order SL scheme for the BGK model in \cite{GRS,RS,SP}. In this paper, we only consider the case $0<\theta \leq 1$ and $-\frac{1}{2}< \nu < 1$.
\end{remark}

\end{section}

\section{Main result}
In this section, we present the explicit error estimate of our scheme measured in weighted $\|\cdot\|_{L_q^\infty}$-norm. We state a theorem for the existence of classical solutions in \cite{Park sa jun PHD thesis}, which is necessary for error estimates in following sections.
In the following theorem, we take a final time $T^f>0$.


\begin{theorem}\label{T.3.1} \cite{Park sa jun PHD thesis,PY}
Let $-1/2<\nu< 1$, $0<\theta \leq 1$, $\delta>0$, $q>5+\delta$.	Suppose that the initial function $f_0$ satisfies the following two conditions: 
		\begin{align}\label{T.3.1 cond}
			\begin{split}
			&(1)~\|f_0\|_{L_{2,q}^{\infty}} < \infty,\cr
			&(2)~f_0(x-vt,v,I) > C_0^1e^{-C_0^2(|v|^a +I^b)},\quad \text{for some constants $a, b,C_0^1, C_0^2 >0$}.
			\end{split}
		\end{align}
		Then, there exists a unique solution for (\ref{A-1}) 
		that satisfies
		\begin{itemize}
			\item ($\mathcal{A}1$): $f$ is uniformly bounded:
			\[
			\|f(t)\|_{L_{2,q}^{\infty}}\leq C_{2,1}e^{C_{2,2} t}\{ \|f_0\|_{L_{2,q}^\infty} + 1 \}
			\]
			for some positive constants $C_{2,1}$ and $C_{2,2}$.
			\item ($\mathcal{A}2$): There exist positive constants $C_{T^f,f_0}$, $C_{T^f,f_0,\delta}$ and $C_{T^f,f_0,\delta,q}$ such that
			\begin{align*}
				\rho(x,t)&\geq C_{T^f,f_0},\cr
				T_\delta(x,t)&\geq C_{T^f,f_0,\delta},\cr\rho(x,t) +|U(x,t)| + T_\delta(x,t)&\leq C_{T^f,f_0,\delta,q}.
			\end{align*}
		\end{itemize}
		
\end{theorem}
\begin{remark}
Existence of classical solutions and its asymptotic equilibrization in near-equilibrium regime can be found in \cite{Y1}.
\end{remark}

Now, we state our main theorem.
\begin{theorem}\label{T.3.2}
	Let $-1/2<\nu< 1$, $0<\theta \leq 1$, \textcolor{magenta}{$0<\delta\leq 2$} and $q>5+\delta$. Let $f$ be the unique smooth solution of \eqref{A-1} corresponding
	to the initial data $f_0$ satisfying two initial conditions in Theorem \ref{T.3.1} and $\| f_0\|_{L_{1,q+1}^\infty}<\infty$.
	
	For a positive $r_{\Delta v, \Delta I}>0$ given in Theorem \ref{T.5.5}, assume that $\Delta v$ and $\Delta I$ satisfy
	$$\Delta v,\Delta I < r_{\Delta v, \Delta I}.$$ 
	Then, the discrete solution $f^{n}_{i,j,k}$ constructed from (\ref{main scheme}) satisfies
	the following explicit error estimate:
\begin{align*}
\|f^{N_t} - f(T^f)\|_{L_q^\infty} &\leq  C \bigg( \frac{(\Delta x)^2}{\Delta t} + (\Delta x)^2 + \Delta v + \Delta I + \Delta t\bigg)
\end{align*}	 
where $C$ is a constant depending on $T^f,q,\delta,\kappa,\theta,\nu,\Delta t$, but can be uniformly bounded regardless of $\Delta t>0$. 
\end{theorem}

\begin{remark}
	(1) The value of $r_{\Delta v, \Delta I}$ is given in Theorem \ref{T.5.5}.
	(2) The constant $C$ in the error bound blows up as $\kappa \rightarrow 0$.
\end{remark}

\section{Technical lemmas}
In this section, we present several technical lemmas.
\begin{lemma}\label{L.4.1}
	The discrete
	solution $f^n$ and $\tilde{f}^n$ satisfies
	\begin{align*}
		\|\tilde{f}^0\|_{L_q^\infty}&\leq \|f_0\|_{L_q^\infty}, \quad \text{for } n= 0,\cr		
		\|\tilde{f}^n\|_{L_q^\infty}&\leq \|f^n\|_{L_q^\infty}, \quad \text{for } n\geq 1.
	\end{align*}	
\end{lemma}
\begin{proof}
	For $n=0$, we recall that no initial errors are assumed. Then, 
	\begin{align*}
		\|\tilde{f}^0\|_{L_q^\infty}&=\sup_{i,j,k}\left|\tilde{f}_{i,j,k}^0\left(1+|v_j|^2+I_k^{\frac{2}{\delta}}\right)^{\frac{q}{2}}\right|\cr
		&=\sup_{i,j,k}\left|f_0(x_i-v_j^1\Delta t, v_j,I_k)\left(1+|v_j|^2+I_k^{\frac{2}{\delta}}\right)^{\frac{q}{2}}\right|\cr
		&\leq\sup_{x,v,I}\left|f_0(x-v^1\Delta t, v,I)\left(1+|v|^2+I^{\frac{2}{\delta}}\right)^{\frac{q}{2}}\right|\cr
		&= \|f_0\|_{L_q^\infty}.
	\end{align*}
	For $n \geq 1$, 
	we use \eqref{B-2} to obtain
	\begin{align*}
		\|\tilde{f}^n\|_{L_q^\infty}&=\sup_{i,j,k}\left|\tilde{f}_{i,j,k}^n\left(1+|v_j|^2+I_k^{\frac{2}{\delta}}\right)^{\frac{q}{2}}\right|\cr
		&= \sup_{i,j,k}\left|\big(a_jf_{s,j,k}^n + (1-a_j) f_{s+1,j,k}^n\big)\left(1+|v_j|^2+I_k^{\frac{2}{\delta}}\right)^{\frac{q}{2}}\right|\cr
		&\leq \sup_{i,j,k}\left|f_{i,j,k}^n \left(1+|v_j|^2+I^{\frac{2}{\delta}}\right)^{\frac{q}{2}}\right|,
	\end{align*}
	where the index $s$ is determined as in \eqref{B-2} for each $i,\,j$, and the last inequality follows from the inequalities $0<a_j\leq 1$. This completes the proof.	
\end{proof}

In the following lemma, we establish the equivalent relations for $(\tilde{\mathcal{T}}_{\nu,\theta})_i^n$ and
$(\tilde{T}_{\theta})_i^n$. 

\begin{lemma}\label{L.4.2}
	Let $\delta>0$, $-1/2<\nu<1$ and $0<\theta\leq 1$. Suppose  $\tilde{f}_{i,j,k}^n>0$ and $\tilde{\rho}_i^n>0$. Then, the discrete temperature tensor $(\tilde{\mathcal{T}}_{\nu,\theta})_i^n$ and
	relaxation temperature $(\tilde{T}_{\theta})_i^n$ satisfy the following  estimates:
	\begin{align*} 
	\begin{split}
		&(1)\ \lambda\theta (\tilde{T}_{\delta})_i^n Id \leq (\tilde{\mathcal{T}}_{\nu,\theta})_i^n \leq \frac{1}{3}\lambda C_{\nu}\big\{3+\delta(1-\theta)\big\} (\tilde{T}_{\delta})_i^nId, \cr
		&(2)\ \theta (\tilde{T}_{\delta})_i^n \leq (\tilde{T}_{\theta})_i^n \leq \frac{1}{\delta}\big\{\delta+3(1-\theta)\big\} (\tilde{T}_{\delta})_i^n,
	\end{split}
	\end{align*}
	where $C_{\nu}=\max\{1-\nu,1+2\nu\}$ and $\displaystyle \lambda \equiv \frac{\kappa + A_{\nu,\theta}\Delta t}{\Delta t + \kappa}$.
\end{lemma}
\begin{proof}
	\noindent {\bf (1) The estimate for $(\tilde{\mathcal{T}}_{\nu,\theta})_i^n$}: For $k \in \mathbb{R}^3$, recall the definition of $(\tilde{\mathcal{T}}_{\nu,\theta})_i^n$ in \eqref{B-5} to have
	\begin{align}\label{D-1}
		\begin{split}
			k^{\top}\big\{\tilde{\rho}_i^n (\tilde{\mathcal{T}}_{\nu,\theta})_i^n \big\}k
			&=k^{\top}\bigg\{\lambda\theta  \bigg[\frac{2}{3+\delta}\sum_{j,k} \tilde{f}_{i,j,k}^{n} \left(\frac{|v_j-\tilde{U}_i^n|^2}{2} + I^{\frac{2}{\delta}}\right) (\Delta v)^3 \Delta I\bigg]Id\bigg\}k\cr
			&\quad +k^{\top}\bigg\{\lambda(1-\theta)(1-\nu) \bigg[\frac{2}{3}\sum_{j,k} \tilde{f}_{i,j,k}^{n} \frac{|v_j-\tilde{U}_i^n|^2}{2}(\Delta v)^3 \Delta I\bigg]Id\bigg\}k\cr
			&\quad +k^{\top}\bigg\{(1-\theta)\bar{\nu}  \sum_{j,k} \tilde{f}_{i,j,k}^{n}(v_j-\tilde{U}_i^{n}) \otimes (v_j-\tilde{U}_i^{n})(\Delta v)^3 \Delta I\bigg\}k\cr
			&\equiv R_1 + R_2 + R_3.
		\end{split}
	\end{align}
	 Depending on the range of $\nu$, we respectively estimate the upper and lower bounds of $k^{\top}\big\{\tilde{\rho}_i^n (\tilde{\mathcal{T}}_{\nu,\theta})_i^n \big\}k$ in \eqref{D-1} as follows:  
	
	\noindent {\bf (1-1) Upper bound estimate of (\ref{D-1}):}
	
	\noindent {\bf (1-1-1) $0 < \nu <1$}: 
	We first simplify $R_3$ by using the following identity:
	\begin{align*} 
	k^{\top}(v_j-\tilde{U}_i^{n}) \otimes (v_j-\tilde{U}_i^{n})k= \big(k\cdot (v_j-\tilde{U}_i^{n})\big)^2,
	\end{align*}
	and use the Cauchy-Schwartz inequality as follows:
		\begin{align*}
			\sum_{j,k} \tilde{f}_{i,j,k}^{n} \big(k\cdot (v_j-\tilde{U}_i^{n})\big)^2 (\Delta v)^3 \Delta I &\leq \sum_{j,k} \tilde{f}_{i,j,k}^{n}|v_j-\tilde{U}_i^{n}|^2 (\Delta v)^3 \Delta I |k|^2 \leq 3\tilde{\rho}_i^n (\tilde{T}_{tr})_i^n|k|^2.
		\end{align*}
		Then, the upper bound of \eqref{D-1} is given by
		\begin{align}\label{1-a}
			\begin{split}
			k^{\top}\{\tilde{\rho}_i^n (\tilde{\mathcal{T}}_{\nu,\theta})_i^n\}k
			&\leq \lambda\theta\tilde{\rho}_i^n (\tilde{T}_{\delta})_i^n|k|^2 +\lambda(1-\theta)(1-\nu)\tilde{\rho}_i^n (\tilde{T}_{tr})_i^n|k|^2\cr
			&\quad + 3(1-\theta)\bar{\nu} \tilde{\rho}_i^n (\tilde{T}_{tr})_i^n|k|^2 \cr
			&\leq \lambda (1+2\nu)\tilde{\rho}_i^n\left\{\theta (\tilde{T}_{\delta})_i^n+ (1-\theta) (\tilde{T}_{tr})_i^n\right\}|k|^2.
			\end{split}
		\end{align}
		In the last line, we use $\displaystyle 0< \bar{\nu}=\frac{\kappa \nu }{\Delta t + \kappa} \leq \nu$ and $\lambda>1$.
		
	\noindent {\bf (1-1-2) $-1/2 < \nu \leq 0$}:
		In this case, we have $\bar{\nu}\leq 0$. Then, \eqref{D-1} becomes
		\begin{align}\label{1-b}
			\begin{split}
			k^{\top}\{\tilde{\rho}_i^n (\tilde{\mathcal{T}}_{\nu,\theta})_i^n\}k
			&\leq \lambda \left\{\theta\tilde{\rho}_i^n (\tilde{T}_{\delta})_i^n|k|^2+(1-\theta)(1-\nu)\tilde{\rho}_i^n (\tilde{T}_{tr})_i^n|k|^2\right\} \cr
			&\leq \lambda (1-\nu)\tilde{\rho}_i^n\left\{\theta (\tilde{T}_{\delta})_i^n+ (1-\theta) (\tilde{T}_{tr})_i^n\right\}|k|^2.
			\end{split}
		\end{align}
	Combine \eqref{1-a} and \eqref{1-b} and divide both sides of \eqref{D-1} by $\tilde{\rho}_i^n>0$ to derive
	\begin{align}\label{D-3}
		k^{\top} (\tilde{\mathcal{T}}_{\nu,\theta})_i^n k
		\leq  \max\{1-\nu,1+2\nu\} \lambda \left\{(1-\theta)(\tilde{T}_{tr})_i^n+\theta (\tilde{T}_{\delta})_i^n\right\}|k|^2.
	\end{align}
	Now, we recall the definition of $(\tilde{T}_{\delta})_i^n$ in \eqref{tilde TTT} to obtain
	\begin{equation*}
	(\tilde{T}_{\delta})_i^n= \frac{3}{3+\delta}(\tilde{T}_{tr})_i^n+\frac{\delta}{3+\delta}(\tilde{T}_{I,\delta})_i^n\geq \frac{3}{3+\delta}(\tilde{T}_{tr})_i^n,
	\end{equation*}
	which, together with (\ref{D-3}), leads to
	\[
	k^{\top} (\tilde{\mathcal{T}}_{\nu,\theta})_i^n k \leq  \frac{1}{3}\max\{1-\nu,1+2\nu\} \lambda \big\{3+\delta(1-\theta)\big\}(\tilde{T}_{\delta})_i^n|k|^2.
	\]
	
	\noindent {\bf (1-2) Lower bound estimate of (\ref{D-1}):}
		
	\noindent {\bf (1-2-1) $0<\nu<1$}: The summation $R_2 + R_3$ in (\ref{D-1}) satisfies
		\begin{align*}
			R_2 + R_3&=k^{\top}\bigg\{\lambda(1-\theta)(1-\nu) \bigg[\frac{2}{3}\sum_{j,k} \tilde{f}_{i,j,k}^{n} \frac{|v_j-\tilde{U}_i^n|^2}{2}(\Delta v)^3 \Delta I\bigg]Id\bigg\}k\cr
			&\quad +(1-\theta)\bar{\nu}  \sum_{j,k} \tilde{f}_{i,j,k}^{n} \big(k\cdot(v_j-\tilde{U}_i^{n})\big)^2 (\Delta v)^3 \Delta I\cr
			&\geq k^{\top}\bigg\{\lambda(1-\theta)(1-\nu) \tilde{\rho}_i^n (\tilde{T}_{tr})_i^nId\bigg\}k\cr
			&\geq \frac{\kappa}{\Delta t + \kappa} (1-\theta)(1-\nu) k^{\top}\bigg\{ \tilde{\rho}_i^n (\tilde{T}_{tr})_i^nId\bigg\}k.
		\end{align*}
		In the last line, we use $\displaystyle \lambda = \frac{\kappa + A_{\nu,\theta}\Delta t}{\Delta t + \kappa} \geq \frac{\kappa}{\Delta t + \kappa}$ with $A_{\nu,\theta} = 1/(1-\nu + \nu\theta)  >0$.
		
	\noindent {\bf (1-2-2) $-1/2<\nu\leq 0$}: In this range of $\nu$, we have $\lambda>0$. Then, 
		\begin{align*}
			R_2 + R_3&=k^{\top}\bigg\{\bigg(\lambda(1-\theta)(1-\nu)\bigg) \bigg[\frac{2}{3}\sum_{j,k} \tilde{f}_{i,j,k}^{n} \frac{|v_j-\tilde{U}_i^n|^2}{2}(\Delta v)^3 \Delta I\bigg]Id\bigg\}k\cr
			&\quad +k^{\top}\bigg\{(1-\theta)\bar{\nu}  \sum_{j,k} \tilde{f}_{i,j,k}^{n}(v_j-\tilde{U}_i^{n}) \cdot (v_j-\tilde{U}_i^{n})(\Delta v)^3 \Delta I\bigg\}k\cr
			&\geq k^{\top}\bigg\{\bigg(\frac{\kappa}{\Delta t + \kappa}(1-\theta)(1-\nu)\bigg) \tilde{\rho}_i^n (\tilde{T}_{tr})_i^nId\bigg\}k\cr
			&\quad +k^{\top}\bigg\{\frac{3\kappa}{\Delta t + \kappa}(1-\theta)\nu  \tilde{\rho}_i^n (\tilde{T}_{tr})_i^n Id \bigg\}k\cr
			&= \frac{\kappa}{\Delta t + \kappa}(1-\theta)(1+2\nu) \tilde{\rho}_i^n (\tilde{T}_{tr})_i^n |k|^2.
		\end{align*}
		Since $R_2 + R_3 \geq 0$ for $-1/2 < \nu < 1$, we can conclude that
		\begin{align*}
			k^{\top}\{\tilde{\rho}_i^n (\tilde{\mathcal{T}}_{\nu,\theta})_i^n\}k
			\geq \lambda \theta \tilde{\rho}_i^n (\tilde{T}_{\delta})_i^n|k|^2  + R_2 + R_3 \geq \lambda \theta \tilde{\rho}_i^n (\tilde{T}_{\delta})_i^n|k|^2.
		\end{align*}

	 \noindent {\bf (2) The estimate for $(\tilde{T}_{\theta})_i^n$}: Note that $(\tilde{T}_{tr})_i^n \geq 0$, which gives
	\begin{equation*}
		(\tilde{T}_{\delta})_i^n= \frac{3}{3+\delta}(\tilde{T}_{tr})_i^n+\frac{\delta}{3+\delta}(\tilde{T}_{I,\delta})_i^n\geq \frac{\delta}{3+\delta}(\tilde{T}_{I,\delta})_i^n.
	\end{equation*}
	Then,
	\begin{align*}
		(\tilde{T}_{\theta})_i^n&= (1-\theta)(\tilde{T}_{I,\delta})_i^n+\theta (\tilde{T}_\delta)_i^n\cr
		&\leq (1-\theta)\bigg(\frac{3+\delta}{\delta}(\tilde{T}_{\delta})_i^n\bigg)+\theta (\tilde{T}_\delta)_i^n=\frac{1}{\delta}\left\{\delta+3(1-\theta)\right\} (\tilde{T}_{\delta})_i^n.
	\end{align*}
	Also, from $(\tilde{T}_{I,\delta})_i^n>0$, we have
	\begin{equation*}
		(\tilde{T}_{\theta})_i^n= (1-\theta)(\tilde{T}_{I,\delta})_i^n+\theta (\tilde{T}_{\delta})_i^n\geq \theta (\tilde{T}_{\delta})_i^n.
	\end{equation*}
	This completes the proof.

\end{proof}
\section{Stability of the discrete distribution function}\label{stability sec}
The goal of this section is to show that the numerical solutions and its corresponding macroscopic quantities are uniformly bounded.
First, we define three constants which will be used throughout this section.
\begin{definition}\label{D.5.1}
	We define constants $\bar{C}_{a,b}, \bar{C}_{a,b,q,\delta}$ and $\bar{C}_{\delta,q-m}$ by
	\begin{align*}
		\bar{C}_{a,b} &:= \int_{\mathbb{R}^3\times\mathbb{R}_+} e^{-C_0^2(|v|^a +I^b)} dv dI,\cr
		\bar{C}_{a,b,q,\delta} &:= \sup_{v,I} e^{-C_0^2(|v|^a +I^b)} (1+|v|^2+I^{\frac{2}{\delta}})^{\frac{q}{2}},\cr
		\bar{C}_{\delta,q-m} &:= \int_{\mathbb{R}^3\times\mathbb{R}_+} \frac{1}{(1+|v|^2+I^{\frac{2}{\delta}})^{\frac{q-m}{2}}}dv dI, \quad q -m> \max(2,\delta).
	\end{align*}
	where $a,b,m,q$ are constants and $C_0^2$ is defined in \ref{T.3.1 cond}.
\end{definition}
In the following, we summarize the main stability estimates of this section as $E_1^n$ and $E_2^n$.
\begin{definition}\label{D.5.2}
	For $n \geq 1$, we say that 
	\begin{enumerate}
		\item  $f_{i,j,k}^n$ satisfies $E_1^n$, if $A^n$ and $B^n$ hold:
		\begin{align*}
			(A^n)& \quad \|\tilde{f}_{i,j,k}^n\|_{L_q^\infty}  \leq \bigg(\frac{\kappa  +   A_{\nu,\theta} \Delta t C_{\mathcal{M}}}{\kappa + A_{\nu,\theta} \Delta t}\bigg)^n \|f_0\|_{L_q^\infty} \leq e^{\frac{ C_{\mathcal{M}}A_{\nu,\theta}  }{\kappa + A_{\nu,\theta} \Delta t}T^f} \|f_0\|_{L_q^\infty},\\ 
			(B^n)& \quad \tilde{f}_{i,j,k}^n \geq \bigg(\frac{\kappa}{\kappa + A_{\nu,\theta} \Delta t}\bigg)^n C_0^1e^{-C_0^2(|v_j|^a +I_k^b)} \geq e^{-\frac{A_{\nu,\theta}}{\kappa} T^f} C_0^1e^{-C_0^2(|v_j|^a +I_k^b)}.
		\end{align*}
	\item  $f_{i,j,k}^n$ satisfies $E_2^n$, if $C^n$ and $D^n$ hold:
	\begin{align*}
	(C^n)& \quad \tilde{\rho}_i^n \geq \frac{1}{2}\bar{C}_{a,b} C_0^1e^{-\frac{A_{\nu,\theta}}{\kappa} T^f}=:\tilde{\rho}_{lower},\cr
	& \quad (\tilde{T}_\delta)_i^n \geq  \bigg( \frac{1}{2}\frac{\bar{C}_{a,b} C_0^1}{C_\delta \|f_0\|_{L_{q}^{\infty}}}e^{-\big(\frac{1}{\kappa}  +\frac{ C_{\mathcal{M}} }{\kappa + A_{\nu,\theta} \Delta t}\big)T^f}\bigg)^{\frac{2}{3+\delta}}=:(\tilde{T}_\delta)_{lower}.  \cr
	(D^n)& \quad \|\tilde{\rho}^n\|_{L_x^\infty} \leq 2 \bar{C}_{\delta,q} e^{\frac{ C_{\mathcal{M}}A_{\nu,\theta}  }{\kappa + A_{\nu,\theta} \Delta t}T^f} \|f_0\|_{L_{q}^{\infty}}=:\tilde{\rho}_{upper},\cr
	& \quad \|\tilde{U}^n\|_{L_x^\infty} \leq \frac{4\bar{C}_{\delta,q-1}}{\bar{C}_{a,b} C_0^1} e^{\big(\frac{1}{\kappa}  + \frac{ C_{\mathcal{M}}  }{\kappa + A_{\nu,\theta} \Delta t}\big)A_{\nu,\theta}T^f} \|f_0\|_{L_{q}^{\infty}}=:\tilde{U}_{upper},\cr
	& \quad \|(\tilde{T}_\delta)^n\|_{L_x^\infty} \leq  \frac{8}{3+\delta}\frac{\bar{C}_{\delta,q-2}}{\bar{C}_{a,b} C_0^1}e^{\big(\frac{1}{\kappa}  + \frac{ C_{\mathcal{M}}  }{\kappa + A_{\nu,\theta} \Delta t}\big)A_{\nu,\theta}T^f} \|f_0\|_{L_{q}^{\infty}}=:(\tilde{T}_\delta)_{upper}.
	\end{align*}
	\item 
	We define $E^n = E_1^n \wedge E_2^n$.
	\end{enumerate}
\end{definition}

\begin{remark} 
	The constants $C_0^1$ and $C_0^2$ are defined in \eqref{T.3.1 cond}. Also, the definition of $C_{\mathcal{M}}$ is given in Lemma \ref{L.5.9}. In $A^n$ and $B^n$. 
\end{remark}

To state the main result of this section, we need another technical definitions.
\begin{definition}\label{D.5.4}
	We define $a_1, a_2$ and $a_3$ by
	\begin{align*}
		a_1:&=\left(\frac{1}{2^{\frac{8+\delta}{2}}\pi^2(3+\delta)^{\frac{1+\delta}{2}}}\frac{\tilde{\rho}_{lower} (\tilde{T}_{\delta})_{lower}}{ \|f_0\|_{L_q^{\infty}}} e^{-\frac{ C_{\mathcal{M}}A_{\nu,\theta}  }{\kappa + A_{\nu,\theta} \Delta t}T^f} \right)^\frac{1}{5+\delta},\cr
		a_2&:=\left(\frac{q-\delta-5}{2^{-\frac{2+\delta}{2}}\pi^2(3+\delta)^{\frac{3+\delta}{2}}}\frac{\tilde{\rho}_{lower}}{ \|f_0\|_{L_{q}^{\infty}}} e^{-\frac{ C_{\mathcal{M}}A_{\nu,\theta}  }{\kappa + A_{\nu,\theta} \Delta t}T^f}\right)^\frac{1}{\delta+3-q},\cr
		a_3&:=\left(\frac{1}{2^{\frac{8+\delta-q}{2}}\pi^2(3+\delta)^{\frac{1+\delta-q}{2}}}\frac{\tilde{\rho}_{lower} (\tilde{T}_\delta)_{lower}^{q} }{ \|f_0\|_{L_{q}^{\infty}}}e^{-\frac{ C_{\mathcal{M}}A_{\nu,\theta}  }{\kappa + A_{\nu,\theta} \Delta t}T^f}\right)^\frac{1}{3+\delta+q}.
	\end{align*}

\end{definition}

The following stability estimate is the main result of this section.
\begin{theorem}\label{T.5.5}
	Choose $l> 0$ small enough so that $\Delta v ,\Delta I < l$ satisfies
	\begin{align}\label{T.5.5 C}
\begin{split}
			\frac{1}{2}\bar{C}_{a,b} &< \sum_{j,k}e^{-C_0^2(|v_j|^a +I_k^b)} (\Delta v)^3 \Delta I <2\bar{C}_{a,b}, \cr
	\frac{1}{2}\bar{C}_{a,b,q,\delta}&< \sup_{j,k} e^{-C_0^2(|v_j|^a +I_k^b)} \left(1+|v_j|^2+I_k^{\frac{2}{\delta}}\right)^{\frac{q}{2}} <2\bar{C}_{a,b,q,\delta}, \cr
	\frac{1}{2}\bar{C}_{\delta,q-m} &< \sum_{j,k} \frac{1}{\left(1+|v_j|^2+I_k^{\frac{2}{\delta}}\right)^{\frac{q-m}{2}}} (\Delta v)^3 \Delta I <2\bar{C}_{\delta,q-m},
\end{split}
	\end{align}
and
	\begin{align}\label{T.5.5 A}
\begin{split}
		\sum_{\mathcal{A}(v_j,\tilde{U}_i^n,I_k) \leq R+\Delta v + \Delta I}  (\Delta v)^3 \Delta I &\leq \int_{\mathcal{A}(v,\tilde{U}_i^n,I) \leq 2(R+\Delta v + \Delta I)} dv dI,\\
	\sum_{\mathcal{A}(v_j,0,I_k)>R+2\Delta v + 2\Delta I}\frac{1}{\left|\mathcal{A}(v_j,0,I_k)\right|^{q-2}}(\Delta v)^3 \Delta I
	&\leq \int_{\mathcal{A}(v,0,I)>R+\Delta v + \Delta I}\frac{1}{\left|\mathcal{A}(v,0,I)\right|^{q-2}}dv dI,
\end{split}
	\end{align}
	where  
	\[
	\mathcal{A}(a,b,c):=\left( \frac{1}{3+\delta}|a-b|^2 + \frac{2}{3+\delta}c^{\frac{2}{\delta}}  \right)^\frac{1}{2}.
	\]
	Also, assume that $\Delta v$ and $\Delta I$ satisfies
	\begin{align}\label{cond T.5.5}
		\Delta v +\Delta I<  \min\left(a_1,a_2,a_3,l,\frac{1}{2}\right)=: r_{\Delta v,\Delta I},
	\end{align}
	where $a_1,a_2,a_3$ are defined in Definition \ref{D.5.4}. 
	Then, $f_{i,j,k}^n$ satisfies $E^n$ for all $n \geq 0$.
\end{theorem}
Since several technical lemmas have to be established, we postpone the proof of this theorem to the end of this section.

\begin{lemma}\label{L.5.6} Assume $f_{i,j,k}^n$ satisfies $E^n$ and the condition \eqref{cond T.5.5} holds. Then, 
	\begin{align*}
	\tilde{\rho}_i^n \leq C_{\delta}\|f^n\|_{L_q^{\infty}}\big((\tilde{T}_{\delta})_i^n\big)^{\frac{3+\delta}{2}},
	\end{align*}
	where
	\[
	C_{\delta}=2^{\frac{13+2\delta}{2}}\pi^2(3+\delta)^{\frac{1+\delta}{2}}.
	\]
\end{lemma}
\begin{proof}
	We first divide the macroscopic density $\tilde{\rho}_i^n$ into two parts:
	\begin{align*}
		\begin{split}
		\tilde{\rho}_i^n&=\sum_{\mathcal{A}(v_j,\tilde{U}_i^n,I_k) > R+\Delta v + \Delta I} \tilde{f}_{i,j,k}^{n}  (\Delta v)^3 \Delta I+\sum_{\mathcal{A}(v_j,\tilde{U}_i^n,I_k) \leq R+\Delta v + \Delta I} \tilde{f}_{i,j,k}^{n}  (\Delta v)^3 \Delta I\cr
		&\equiv \mathcal{I}_{11} + \mathcal{I}_{12}.
		\end{split}
	\end{align*}
	The first term $\mathcal{I}_{11}$ is bounded by 
	\begin{align*}
	\mathcal{I}_{11} &=\sum_{\mathcal{A}(v_j,\tilde{U}_i^n,I_k) > R+\Delta v + \Delta I} \tilde{f}_{i,j,k}^{n} (\Delta v)^3 \Delta I\cr
	&\leq\sum_{\mathcal{A}(v_j,\tilde{U}_i^n,I_k) > R+\Delta v + \Delta I} \tilde{f}_{i,j,k}^{n} \frac{ \frac{1}{3+\delta}|v_j-\tilde{U}_i^n|^2 + \frac{2}{3+\delta}I_k^{\frac{2}{\delta}}}{(R+\Delta v + \Delta I)^2} (\Delta v)^3 \Delta I\cr
	&\leq \frac{1}{(R+\Delta v + \Delta I)^2}\tilde{\rho}_i^n (\tilde{T}_{\delta})_i^n.
	\end{align*}
	Since $\Delta v$ and $\Delta I$ satisfies \eqref{T.5.5 A}, we can bound $\mathcal{I}_{12}$ by
	\begin{align}\label{definite integral}
	\begin{split}
	\mathcal{I}_{12}&=\sum_{\frac{1}{3+\delta}|v_j-\tilde{U}_i^n|^2+\frac{2}{3+\delta}I_k^{\frac{2}{\delta}} \leq (R+\Delta v + \Delta I)^2} \tilde{f}_{i,j,k}^{n}  (\Delta v)^3 \Delta I\cr
	&\leq \left(\int_{\frac{1}{3+\delta}|v-\tilde{U}_i^n|^2+\frac{2}{3+\delta}I^{\frac{2}{\delta}} \leq 4(R+\Delta v + \Delta I)^2} dv dI \right)\|\tilde{f}^n\|_{L_{q}^{\infty}}.
	\end{split}
	\end{align}
	To calculate the definite integral in \eqref{definite integral}, we use a change of variable:
	\begin{align*}
	\left(\sqrt{\frac{1}{3+\delta}}(v-\tilde{U}_i^n), \sqrt{\frac{2}{3+\delta}}I^{\frac{1}{\delta}}\right)&=\big(r\sin\varphi\cos\theta\sin k,r\sin\varphi\sin\theta\sin k,r\cos\varphi\sin k, r\cos k\big),
	\end{align*}
	where 
	\[
	0\leq r\leq 2(R+\Delta v + \Delta I),\quad  0\leq\varphi\leq\pi,\quad  0\leq\theta\leq2\pi,\quad 0\leq k \leq \frac{\pi}{2}.
	\]
	Then, the Jacobian is given by 
	\begin{align*}
		\left|\frac{\partial(v^1-(\tilde{U}_i^n)^1,v^2-(\tilde{U}_i^n)^2,v^3-(\tilde{U}_i^n)^3,I)}{\partial(r,\varphi,\theta,k)}\right|
		&= 2^{-\frac{\delta}{2}}\left(3+\delta\right)^{\frac{3+\delta}{2}}\delta r^{\delta+2}|\sin\varphi\cos^{\delta-1}k\sin^2k|,
	\end{align*}
	and we have
	\begin{align*}
	\mathcal{I}_{12}&\leq\|\tilde{f}^n\|_{L_{q}^{\infty}} 2^{-\frac{\delta}{2}}(3+\delta)^{\frac{3+\delta}{2}}
	\int_{0}^{\frac{\pi}{2}}\int_{0}^{\pi}\int_{0}^{2\pi}\int_{0}^{2(R+\Delta v + \Delta I)} \delta r^{\delta+2}|\sin\varphi\cos^{\delta-1}k\sin^2k| drd\theta d\varphi dk.
	\end{align*}
	Using
	\begin{align*}
		&\int_{0}^{\frac{\pi}{2}} \delta |\cos^{\delta-1}k\sin^2k| dk \leq \int_{0}^{\frac{\pi}{2}} \delta \cos^{\delta-1}k\sin{k} dk = 1,\cr
		&\int_{0}^{\pi} |\sin\varphi| d\varphi \leq \pi, \quad \int_{0}^{2\pi} d\theta \leq 2\pi,\cr
		&\int_{0}^{2(R+\Delta v + \Delta I)}  r^{\delta+2}dr \leq \frac{1}{3+ \delta}\left(2(R+\Delta v + \Delta I)\right)^{3+\delta},
	\end{align*}
	we obtain
	\begin{align*}
	\mathcal{I}_{12}
	&\leq \|\tilde{f}^n\|_{L_{q}^{\infty}}
	\left\{2^{-\frac{\delta}{2}}(3+\delta)^{\frac{3+\delta}{2}}\frac{2\pi^2}{3+\delta}      \right\}\bigg( 2(R+\Delta v + \Delta I) \bigg)^{3+\delta} \cr
	&= \|\tilde{f}^n\|_{L_{q}^{\infty}}\left\{2^{\frac{8+\delta}{2}}\pi^2(3+\delta)^{\frac{1+\delta}{2}} \right\}\big( R+\Delta v + \Delta I \big)^{3+\delta}.
	\end{align*}
	Combining the estimates for $\mathcal{I}_{11}$ and $\mathcal{I}_{12}$, we derive
	\begin{align*}
	\tilde{\rho}_i^n \leq
	\frac{1}{(R+\Delta v + \Delta I)^2}\tilde{\rho}_i^n (\tilde{T}_{\delta})_i^n + \left\{2^{\frac{8+\delta}{2}}\pi^2(3+\delta)^{\frac{1+\delta}{2}} \right\}(R+\Delta v + \Delta I)^{3+\delta}\|\tilde{f}^n\|_{L_q^{\infty}}.
	\end{align*}
	Here, we equates two terms on the upper bound so that the bound can be minimized. That is, the number $R$ is taken by
	\[
	R+\Delta v + \Delta I=\left(\frac{\tilde{\rho}_i^n (\tilde{T}_{\delta})_i^n}{2^{\frac{8+\delta}{2}}\pi^2(3+\delta)^{\frac{1+\delta}{2}} \|\tilde{f}^n\|_{L_q^{\infty}}}\right)^\frac{1}{5+\delta} \geq a_1 > \Delta v + \Delta I,
	\]
	where $a_1$ is given in Definition \ref{D.5.4} and the last inequality holds due to \eqref{cond T.5.5}. With the choice of such $R>0$, we have
	\begin{align*}
	\tilde{\rho}_i^n
	&\leq
	2\left\{2^{\frac{8+\delta}{2}}\pi^2(3+\delta)^{\frac{1+\delta}{2}} \right\}^{\frac{2}{5+\delta}} \|\tilde{f}^n\|_{L_q^{\infty}}^{\frac{2}{5+\delta}} \left\{\tilde{\rho}_i^n (\tilde{T}_{\delta})_i^n\right\}^{\frac{3+\delta}{5+\delta}}.
	\end{align*}
	This, together with Lemma \ref{L.4.1}, gives
	\begin{align*}
	\tilde{\rho}_i^n
	\leq
	\left\{2^{\frac{13+2\delta}{2}}\pi^2(3+\delta)^{\frac{1+\delta}{2}} \right\}\|f^n\|_{L_q^{\infty}}\tilde{T}_{\delta}^{\frac{3+\delta}{2}},
	\end{align*}
	which completes the proof.
\end{proof}
\begin{lemma}\label{L.5.7} Let $q>5+\delta$. Suppose futher that $f_{i,j,k}^n$ satisfies $E^n$ and $\Delta v, \Delta I$ satisfy the condition \eqref{cond T.5.5}. Then,
	\begin{align*}
	\tilde{\rho}_i^n\big((\tilde{T}_\delta)_i^n+|\tilde{U}_i^n|^2\big)^{\frac{q-\delta-3}{2}} \leq C_{\delta,q,1}\|f^n\|_{L_{q}^{\infty}},
	\end{align*}
	where
	\[
	C_{\delta,q,1}=\left\{\frac{2^{\frac{q-2\delta-5}{2}}\pi^2(3+\delta)^{\frac{q}{2}}}{q-\delta-5}\right\}.
	\]
\end{lemma}
\begin{proof}
	We start by splitting the following quantity into two parts:
	\begin{align}\label{E-2}
	\begin{split}
	\tilde{\rho}_i^n\left((\tilde{T}_{\delta})_i^n+\frac{1}{3+\delta}|\tilde{U}_i^n|^2\right)
	&= \sum_{\mathcal{A}(v_j,0,I_k)>R+2\Delta v + 2\Delta I} \left(\frac{1}{3+\delta}|v_j|^2+\frac{2}{3+\delta}I_k^{\frac{2}{\delta}}\right)\tilde{f}_{i,j,k}^{n} (\Delta v)^3 \Delta I \cr
	&\quad + \sum_{\mathcal{A}(v_j,0,I_k)\leq R+2\Delta v + 2\Delta I} \left(\frac{1}{3+\delta}|v_j|^2+\frac{2}{3+\delta}I_k^{\frac{2}{\delta}}\right)\tilde{f}_{i,j,k}^{n} (\Delta v)^3 \Delta I\cr
	&=\mathcal{I}_{21}+\mathcal{I}_{22}.
	\end{split}
	\end{align} 
	The second term $\mathcal{I}_{22}$ is bounded by
	\begin{align}\label{com 1}
	\mathcal{I}_{22} \leq 4(R+\Delta v + \Delta I)^2 \tilde{\rho}_i^n.
	\end{align}
	For $\mathcal{I}_{21}$, we extract $\|f^n\|_{L^{\infty}_q}$ out of the summation:
	\begin{align*}
	\mathcal{I}_{21}
	&\leq \sum_{\frac{1}{3+\delta}|v_j|^2+\frac{2}{3+\delta}I_k^{\frac{2}{\delta}}>(R+2\Delta v + 2\Delta I)^2}
	\frac{\left(\frac{1}{3+\delta}|v_j|^2+\frac{2}{3+\delta}I_k^{\frac{2}{\delta}}\right)^{\frac{q}{2}}}{\left(\frac{1}{3+\delta}|v_j|^2+\frac{2}{3+\delta}I_k^{\frac{2}{\delta}}\right)^{\frac{q-2}{2}}} \tilde{f}_{i,j,k}^{n} (\Delta v)^3 \Delta I\cr
	&\leq \|\tilde{f}^n\|_{L_q^{\infty}}
	\sum_{\frac{1}{3+\delta}|v_j|^2+\frac{2}{3+\delta}I_k^{\frac{2}{\delta}}>(R+2\Delta v + 2\Delta I)^2}\frac{1}{\left(\frac{1}{3+\delta}|v_j|^2+\frac{2}{3+\delta}I_k^{\frac{2}{\delta}}\right)^{\frac{q-2}{2}}}(\Delta v)^3 \Delta I.
	\end{align*}
	As in Lemma {\ref{L.5.6}}, the condition \eqref{T.5.5 A} makes it possible to estimate the above discrete summation by a definite integral using a change of variable:
	\begin{align*}
		\left(\sqrt{\frac{1}{3+\delta}}v, \sqrt{\frac{2}{3+\delta}}I^{\frac{1}{\delta}}\right)&=\big(r\sin\varphi\cos\theta\sin k,r\sin\varphi\sin\theta\sin k,r\cos\varphi\sin k, r\cos k\big).
	\end{align*}
	Then, we get 
	\begin{align}\label{com 2}
	\begin{split}
		\mathcal{I}_{21}&\leq\|\tilde{f}^n\|_{L_q^{\infty}} \int_{0}^{\frac{\pi}{2}}\int_{0}^{\pi}\int_{0}^{2\pi}\int_{R+\Delta v + \Delta I}^{\infty}\frac{\delta\left(3+\delta\right)^{\frac{3}{2}}\left(\frac{3+\delta}{2}\right)^{\frac{\delta}{2}} r^{\delta+2}|\sin\varphi\cos^{\delta-1}k\sin^2k|}{r^{q-2}}dr d\theta d\varphi dk \cr
		&\leq \|\tilde{f}^n\|_{L_q^{\infty}} \left\{\frac{2\pi^2\left(3+\delta\right)^{\frac{3}{2}}\left(\frac{3+\delta}{2}\right)^{\frac{\delta}{2}}}{q-\delta-5}\right\} (R+\Delta v + \Delta I)^{\delta+5-q}\cr
		&=\|\tilde{f}^n\|_{L_q^{\infty}}
		\left\{\frac{2^{\frac{2-\delta}{2}}\pi^2(3+\delta)^{\frac{3+\delta}{2}}}{q-\delta-5}\right\}(R+\Delta v + \Delta I)^{\delta+5-q}.
	\end{split}
	\end{align}
	Combining \eqref{com 1} and \eqref{com 2}, we estimate (\ref{E-2}) by
	\begin{align*}
	&\tilde{\rho}_i^n\left((\tilde{T}_{\delta})_i^n+\frac{1}{3+\delta}|\tilde{U}_i^n|^2\right)\cr
	&\leq 4\tilde{\rho}_i^n (R+\Delta v + \Delta I)^2
	+\left\{\frac{2^{\frac{2-\delta}{2}}\pi^2(3+\delta)^{\frac{3+\delta}{2}}}{q-\delta-5}\right\}\|\tilde{f}^n\|_{L_q^{\infty}}(R+\Delta v + \Delta I)^{\delta+5-q}.
	\end{align*}
	To get an optimal bound, we equate two terms on the upper bound to derive
	\[
	R+\Delta v + \Delta I=\left(\frac{q-\delta-5}{2^{-\frac{2+\delta}{2}}\pi^2(3+\delta)^{\frac{3+\delta}{2}}}\frac{\tilde{\rho}_i^n}{\|\tilde{f}^n\|_{L_q^{\infty}}}\right)^\frac{1}{\delta+3-q}\geq a_2 > \Delta v + \Delta I,
	\]
	where such $R$ can be chosen due to the existence of $a_2$ given in Definition \ref{D.5.4}.
	Then,
	\begin{align*}
	\tilde{\rho}_i^n\left((\tilde{T}_{\delta})_i^n+\frac{1}{3+\delta}|\tilde{U}_i^n|^2\right) \leq 2\bigg\{\frac{2^{-\frac{2+\delta}{2}}\pi^2(3+\delta)^{\frac{3+\delta}{2}}}{q-\delta-5}\bigg\}^{\frac{2}{q-\delta-3}}(\tilde{\rho}_i^n)^{\frac{\delta+5-q}{\delta+3-q}}\|\tilde{f}^n\|_{L_q^{\infty}}^{\frac{2}{q-\delta-3}}.
	\end{align*}
	Consequently,
	\begin{align*}
	\tilde{\rho}_i^n\big((\tilde{T}_{\delta})_i^n+|\tilde{U}_i^n|^2\big)^{\frac{q-\delta-3}{2}}
	&\leq
	\{2(3+\delta)\}^{\frac{q-\delta-3}{2}} \left\{\frac{2^{-\frac{2+\delta}{2}}\pi^2(3+\delta)^{\frac{3+\delta}{2}}}{q-\delta-5}\right\}
	\|\tilde{f}^n\|_{L_q^{\infty}} \cr
	&=
	\left\{\frac{2^{\frac{q-2\delta-5}{2}}\pi^2(3+\delta)^{\frac{q}{2}}}{q-\delta-5}\right\}\|\tilde{f}^n\|_{L_q^{\infty}}.
	\end{align*}
	Combined with Lemma \ref{L.4.1}, this gives the desired estimate.
\end{proof}

\begin{lemma}\label{L.5.8}
	 Assume that $f_{i,j,k}^n$ satisfies $E^n$ and $\Delta v, \Delta I$ satisfy the condition \eqref{cond T.5.5}. Then,
	\begin{align*}
	\frac{\tilde{\rho}_i^n|\tilde{U}_i^n|^{3+\delta+q}}{\bigg(\big((\tilde{T}_{\delta})_i^n+|\tilde{U}_i^n|^2\big)(\tilde{T}_{\delta})_i^n\bigg)^{\frac{3+\delta}{2}}}
	\leq C_{\delta,q,2}\|f^n\|_{L_q^{\infty}},
	\end{align*}
	where $$C_{\delta,q,2}=2^{\frac{17+3\delta+2q}{2}}\pi^2(3+\delta)^{2+\delta}.$$
\end{lemma}
\begin{proof}
	We split the macroscopic momentum into two parts:
	\begin{align*}
	|\tilde{\rho}_i^n\tilde{U}_i^n|
	&\leq
	\sum_{\mathcal{A}(v_j,\tilde{U}_i^n,I_k)\leq R+\Delta v + \Delta I}\tilde{f}_{i,j}^n|v_j|(\Delta v)^3 \Delta I
	+\sum_{\mathcal{A}(v_j,\tilde{U}_i^n,I_k)>R+\Delta v + \Delta I}\tilde{f}_{i,j}^n|v_j|(\Delta v)^3 \Delta I \cr
	&\equiv \mathcal{I}_{31}+\mathcal{I}_{32}.
	\end{align*}
	We first use H\"{o}lder's inequality to obtain
	\begin{align*}
	\mathcal{I}_{31}
	&\leq
	\left|\sum_{\mathcal{A}(v_j,\tilde{U}_i^n,I_k)\leq R+\Delta v + \Delta I} \tilde{f}_{i,j}^n(\Delta v)^3 \Delta I \right|^{1-\frac{1}{q}}
	\left|\sum_{\mathcal{A}(v_j,\tilde{U}_i^n,I_k)\leq R+\Delta v + \Delta I}\tilde{f}_{i,j}^n|v_j|^q(\Delta v)^3 \Delta I\right|^{\frac{1}{q}} \cr
	&\leq
	(\tilde{\rho}_i^n)^{1-\frac{1}{q}}\|\tilde{f}^n\|_{L_q^{\infty}}^{\frac{1}{q}}	\left|\sum_{\mathcal{A}(v_j,\tilde{U}_i^n,I_k)\leq R+\Delta v + \Delta I}(\Delta v)^3 \Delta I\right|^{\frac{1}{q}}.
	\end{align*}
	Then, we use the condition \eqref{T.5.5 A} to get
	\begin{align*}
	\sum_{\mathcal{A}(v_j,\tilde{U}_i^n,I_k)\leq R+\Delta v + \Delta I}(\Delta v)^3 \Delta I&\leq \int_{\frac{3}{3+\delta}|v|^2+\frac{2}{3+\delta}I^{\frac{2}{\delta}}\leq 4(R+\Delta v + \Delta I)^2}dvdI\cr
	&\leq
	\left\{2^{\frac{2-\delta}{2}}\pi^2(3+\delta)^{\frac{1+\delta}{2}} \right\}2^{3+\delta}(R+\Delta v + \Delta I)^{3+\delta},
	\end{align*}
	which gives
	\begin{align*}
	\mathcal{I}_{31} \leq (\tilde{\rho}_i^n)^{1-\frac{1}{q}}\|\tilde{f}^n\|_{L_q^{\infty}}^{\frac{1}{q}}\left\{2^{\frac{2-\delta}{2}}\pi^2(3+\delta)^{\frac{1+\delta}{2}}\right\}^{\frac{1}{q}}2^{3+\delta}(R+\Delta v + \Delta I)^{3+\delta}.
	\end{align*}
	On the other hand, $\mathcal{I}_{32}$ satisfies
	\begin{align*}
	\mathcal{I}_{32}\leq \sum_{\mathcal{A}(v_j,\tilde{U}_i^n,I_k)>R+\Delta v + \Delta I}\tilde{f}_{i,j}^n|v_j|\frac{\left( \frac{1}{3+\delta}|v_j-\tilde{U}_i^n|^2 + \frac{2}{3+\delta}I_k^{\frac{2}{\delta}}  \right)^\frac{1}{2}}{R+\Delta v + \Delta I}(\Delta v)^3 \Delta I.
	\end{align*}
	Here, we use H\"{o}lder's inequality to obtain
	\begin{align*}
	\mathcal{I}_{32}
	&\leq \frac{\sqrt{2(3+\delta)}}{R+\Delta v + \Delta I}\left\{\sum_{j,k}\tilde{f}_{i,j}^n\left(\frac{1}{3+\delta}|v_j|^2+\frac{2}{3+\delta}I_k^{\frac{2}{\delta}}\right)(\Delta v)^3 \Delta I\right\}^{\frac{1}{2}} \cr
	&\quad \times
	\left\{\sum_{j,k}\tilde{f}_{i,j}^n\left(\frac{1}{3+\delta}|v_j-\tilde{U}_i^n|^2+\frac{2}{3+\delta}I_k^{\frac{2}{\delta}}\right)(\Delta v)^3 \Delta I\right\}^{\frac{1}{2}} \cr
	&=
	\frac{\sqrt{2(3+\delta)}}{R+\Delta v + \Delta I}\left\{\frac{1}{3+\delta}\tilde{\rho}_i^n|\tilde{U}_i^n|^2+\tilde{\rho}_i^n(\tilde{T}_{\delta})_i^n\right\}^{\frac{1}{2}}\{\tilde{\rho}_i^n(\tilde{T}_{\delta})_i^n\}^{\frac{1}{2}}.
	\end{align*}
	To sum up, we have
	\begin{align}\label{E-3}
	\begin{split}
	|\tilde{\rho}_i^n\tilde{U}_i^n|
	&\leq
	(\tilde{\rho}_i^n)^{1-\frac{1}{q}}\|\tilde{f}^n\|_{L_q^{\infty}}^{\frac{1}{q}}\left\{2^{\frac{2-\delta}{2}}\pi^2(3+\delta)^{\frac{1+\delta}{2}}\right\}^{\frac{1}{q}}2^{\frac{3+\delta}{q}}(R+\Delta v + \Delta I)^{\frac{3+\delta}{q}}\cr
	&\quad +\frac{\sqrt{2(3+\delta)}}{R+\Delta v + \Delta I}\left\{\tilde{\rho}_i^n|\tilde{U}_i^n|^2+\tilde{\rho}_i^n(\tilde{T}_{\delta})_i^n\right\}^{\frac{1}{2}}\{\tilde{\rho}_i^n(\tilde{T}_{\delta})_i^n\}^{\frac{1}{2}}.
	\end{split}
	\end{align}
	To optimize the upper bound in \eqref{E-3}, we take $R>0$ such that
	\begin{align*}
	R+\Delta v + \Delta I=\left(\frac{\big\{2(3+\delta)\big\}^{\frac{q}{2}}\tilde{\rho}_i^n\big\{\big(|\tilde{U}_i^n|^2+(\tilde{T}_{\delta})_i^n\big)(\tilde{T}_{\delta})_i^n\big\}^{\frac{q}{2}}}{\left\{2^{\frac{8+\delta}{2}}\pi^2(3+\delta)^{\frac{1+\delta}{2}}\right\}\|\tilde{f}^n\|_{L_q^{\infty}}}\right)^\frac{1}{3+\delta+q} \geq a_3>\Delta v + \Delta I.
	\end{align*}
	The number $a_3$ is given in Definition \ref{D.5.4}. 
	Then, the upper bound of (\ref{E-3}) is simplified to
	\begin{align*}
	2\left(\left\{2^{\frac{8+\delta}{2}}\pi^2(3+\delta)^{\frac{1+\delta}{2}}\right\}
	\left\{2(3+\delta)\right\}^{\frac{3+\delta}{2}}(\tilde{\rho}_i^n)^{2+\delta+q}\big\{\big(|\tilde{U}_i^n|^2+(\tilde{T}_{\delta})_i^n\big)(\tilde{T}_{\delta})_i^n\big\}^{\frac{3+\delta}{2}}\|\tilde{f}^n\|_{L_q^{\infty}}\right)^{\frac{1}{3+\delta+q}},
	\end{align*}
	from which we conclude that
	\begin{align*}
	\frac{\tilde{\rho}_i^n|\tilde{U}_i^n|^{3+\delta+q}}{[(|\tilde{U}_i^n|^2+(\tilde{T}_{\delta})_i^n)(\tilde{T}_{\delta})_i^n]^{\frac{3+\delta}{2}}}
	&\leq 2^{3+\delta+q}\{2(3+\delta)\}^{\frac{3+\delta}{2}}
	\left\{2^{\frac{8+\delta}{2}}\pi^2(3+\delta)^{\frac{1+\delta}{2}}\right\}\|\tilde{f}^n\|_{L_q^{\infty}} \cr
	&=
	2^{\frac{17+3\delta+2q}{2}}\pi^2(3+\delta)^{2+\delta}\|\tilde{f}^n\|_{L_q^{\infty}}.
	\end{align*}
	From Lemma \ref{L.4.1}, we finally obtain the desired estimate.
\end{proof}

\begin{lemma}\label{L.5.9}
	Let $q>5+\delta$. Suppose further that $f_{i,j,k}^n$ satisfies $E^n$ and $\Delta v, \Delta I$ satisfy the condition \eqref{cond T.5.5}. Then, 
	\begin{align*}
		\|\mathcal{M}_{\nu,\theta}(\tilde{f}^n)\|_{L_{q}^{\infty}} \leq C_{\mathcal{M}}\|f^n\|_{L_{q}^{\infty}},
	\end{align*}
	where $C_{\mathcal{M}}$ depending on $\nu, \delta,\theta$ and $q$.
\end{lemma}
\begin{remark}
	In the proof, it will be shown that $C_{\mathcal{M}}$ blows up as $\theta$ tends to 0 because $C_{\mathcal{M}} \propto 1/\theta^{\frac{3+\delta}{2}}$. 
\end{remark}
\begin{proof}
	We will show that $\mathcal{M}_{\nu,\theta}(\tilde{f}_{i,j,k}^n)$, $|v_j|^{q}\mathcal{M}_{\nu,\theta}(\tilde{f}_{i,j,k}^n)$ and $I_k^{\frac{q}{\delta}}\mathcal{M}_{\nu,\theta}(\tilde{f}_{i,j,k}^n) $  are controlled by $\|f^n\|_{L_q^{\infty}}$, respectively. \newline
	
	\noindent(a) {\bf The estimate for $\mathcal{M}_{\nu,\theta}(\tilde{f}_{i,j,k}^n)$:}
	We first use Lemma \ref{L.4.2} to get
	\begin{align}\label{E-4}
		\begin{split}
		\frac{1}{2}(v_j-\tilde{U}_i^n)^{\top}\big((\tilde{\mathcal{T}}_{\nu,\theta})_i^n\big)^{-1}(v_j-\tilde{U}_i^n)+\frac{I_k^{\frac{2}{\delta}}}{(\tilde{T}_{\theta})_i^n}\geq \frac{1}{2}
		\frac{3}{\lambda C_{\nu}\big\{3+\delta(1-\theta)\big\}}\frac{|v_j-\tilde{U}_i^n|^2}{(\tilde{T}_{\delta})_i^n}+\frac{I_k^{\frac{2}{\delta}}}{(\tilde{T}_{\theta})_i^n}\geq 0.
		\end{split}
	\end{align}
	Then,
	\begin{align}\label{E-5}
		\begin{split}
		\mathcal{M}_{\nu,\theta}(\tilde{f}_{i,j,k}^n) &\leq \frac{\tilde{\rho}_i^n\Lambda_\delta}{\sqrt{\det\left(2 \pi (\tilde{\mathcal{T}}_{\nu,\theta})_i^n \right)}\big((\tilde{T}_\theta)_i^n\big)^\frac{\delta}{2}}\cr
		&\leq
		\left(\frac{1}{\lambda} \right)^{\frac{3}{2}}\frac{1}{(2\pi)^{3/2}}\frac{\Lambda_\delta}{\theta^{\frac{3+\delta}{2}}}\frac{\tilde{\rho}_i^n}{\big((\tilde{T}_{\delta})_i^n\big)^{\frac{3+\delta}{2}}}\cr
		&\leq
		\left(\frac{1}{\lambda} \right)^{\frac{3}{2}}\frac{1}{(2\pi)^{3/2}}\frac{\Lambda_\delta}{\theta^{\frac{3+\delta}{2}}}\left\{2^{\frac{13+2\delta}{2}}\pi^2(3+\delta)^{\frac{1+\delta}{2}}\right\}\|f^n\|_{L_{q}^{\infty}}.
		\end{split}
	\end{align}
	
	\noindent$(b)$ {\bf The estimate for $\mathcal{M}_{\nu,\theta}(f_{i,j,k}^n)|v_j|^{q}$:} For this, we consider two estimates $|\tilde{U}_i^n|^{q}\mathcal{M}_{\nu,\theta}(f_{i,j,k}^n),$
	and $|v_j-\tilde{U}_i^n|^{q}\mathcal{M}_{\nu,\theta}(f_{i,j,k}^n)$, separately.\newline
	
	\noindent$(b_1)$ $|\tilde{U}_i^n|^{q}\mathcal{M}_{\nu,\theta}(f_{i,j,k}^n)$: From the second inequality in (\ref{E-5}), we obtain
	\begin{align*}
		|\tilde{U}_i^n|^{q}\mathcal{M}_{\nu,\theta}(f_{i,j,k}^n) \leq \left(\frac{1}{\lambda} \right)^{\frac{3}{2}}\frac{1}{(2\pi)^{3/2}}\frac{\Lambda_\delta}{\theta^{\frac{3+\delta}{2}}}|\tilde{U}_i^n|^{q}\frac{\tilde{\rho}_i^n}{\big((\tilde{T}_{\delta})_i^n\big)^{\frac{3+\delta}{2}}}.
	\end{align*}
	If $|\tilde{U}_i^n|<\big((\tilde{T}_{\delta})_i^n\big)^{\frac{1}{2}}$, we have from Lemma \ref{L.5.7} that
	\begin{align*}
		|\tilde{U}_i^n|^{q}\frac{\tilde{\rho}_i^n}{\big((\tilde{T}_{\delta})_i^n\big)^{\frac{3+\delta}{2}}} 
		\leq
		\tilde{\rho}_i^n \big((\tilde{T}_{\delta})_i^n+|\tilde{U}_i^n|^2\big)^{\frac{q-3-\delta}{2}}
		\leq \bigg\{\frac{2^{\frac{q-2\delta-5}{2}}\pi^2(3+\delta)^{\frac{q}{2}}}{q-\delta-5}\bigg\}\|f^n\|_{L_q^{\infty}}.
	\end{align*}
	On the other hand, in the case of $|\tilde{U}_i^n|\geq\big((\tilde{T}_{\delta})_i^n\big)^{\frac{1}{2}}$, we use Lemma \ref{L.5.8} to obtain
	\begin{align*}
		|\tilde{U}_i^n|^{q}\frac{\tilde{\rho}_i^n}{\big((\tilde{T}_{\delta})_i^n\big)^{\frac{3+\delta}{2}}} 
		&=\frac{\tilde{\rho}_i^n |\tilde{U}_i^n|^{q+3+\delta}}{|\tilde{U}_i^n|^{3+\delta}\big((\tilde{T}_{\delta})_i^n\big)^{\frac{3+\delta}{2}}}\cr 
		&\leq
		2^{\frac{3+\delta}{2}}\frac{\tilde{\rho}_i^n |\tilde{U}_i^n|^{q+3+\delta}}{\big\{\big((\tilde{T}_{\delta})_i^n+|\tilde{U}_i^n|^2\big)(\tilde{T}_{\delta})_i^n\big\}^{\frac{3+\delta}{2}}} \cr &\leq 2^{10+3\delta+q}\pi^2(3+\delta)^{2+\delta}\|f^n\|_{L_q^{\infty}}.
	\end{align*}
	Therefore,
	\begin{align*}
		|\tilde{U}_i^n|^{q}\mathcal{M}_{\nu,\theta}(f_{i,j,k}^n)\leq
		\frac{C_1}{\theta^{\frac{3+\delta}{2}}}\|f^n\|_{L_q^{\infty}},
	\end{align*}
	for
	\[
	C_1=\left(\frac{1}{\lambda} \right)^{\frac{3}{2}}\Lambda_\delta\bigg\{\frac{2^{\frac{q-2\delta-5}{2}}\sqrt{\pi}(3+\delta)^{\frac{q}{2}}}{q-\delta-5}\bigg\}
	+2^{\frac{17+6\delta+2q}{2}}\sqrt{\pi}(3+\delta)^{2+\delta}.
	\]
	$(b_2)$  {\bf The estimate for $|v_j-\tilde{U}_i^n|^q\mathcal{M}_{\nu,\theta}(\tilde{f}_{i,j,k}^n)$:}
	From (\ref{E-4}) and Lemma \ref{L.4.2}, we have
	\begin{align*}
		|v_j-&\tilde{U}_i^n|^{q} \mathcal{M}_{\nu,\theta}(\tilde{f}_{i,j,k}^n)\cr
		&\quad\leq
		\frac{1}{(2\pi)^{3/2}}\frac{1}{\theta^{\frac{3+\delta}{2}}}|v_j-\tilde{U}_i^n|^q\frac{\tilde{\rho}_i^n\Lambda_\delta}{\big((\tilde{T}_{\delta})_i^n\big)^{\frac{3+\delta}{2}}} 
		\left(\frac{1}{\lambda} \right)^{\frac{3}{2}}\exp\Bigg(-\frac{3}{2\lambda C_{\nu}\big\{3+\delta(1-\theta)\big\}}\frac{|v_j-\tilde{U}_i^n|^2}{(\tilde{T}_{\delta})_i^n}\Bigg) \cr
		&\quad=\frac{1}{(2\pi)^{3/2}}
		\frac{1}{\theta^{\frac{3+\delta}{2}}}\big((\tilde{T}_{\delta})_i^n\big)^{\frac{q}{2}}\frac{\tilde{\rho}_i^n\Lambda_\delta}{\big((\tilde{T}_{\delta})_i^n\big)^{\frac{3+\delta}{2}}}\left(\frac{|v_j-\tilde{U}_i^n|^2}{(\tilde{T}_{\delta})_i^n} \right)^{\frac{q}{2}}
		\left(\frac{1}{\lambda} \right)^{\frac{3}{2}}\cr
		&\quad\quad \times \exp\Bigg(-\frac{3}{2\lambda C_{\nu}\big\{3+\delta(1-\theta)\big\}}\frac{|v_j-\tilde{U}_i^n|^2}{(\tilde{T}_{\delta})_i^n}\Bigg) \cr
		&\quad\equiv \frac{C_2}{\theta^{\frac{3+\delta}{2}}}\tilde{\rho}_i^n \big((\tilde{T}_{\delta})_i^n\big)^{\frac{q-3-\delta}{2}},
	\end{align*}
	where
	\[
	C_2=\left(\frac{1}{\lambda} \right)^{\frac{3}{2}}\frac{\Lambda_\delta}{(2\pi)^{3/2}} \sup_{x\geq0}\big(x^{q/2}e^{-x}\big)\left\{\frac{2\lambda C_{\nu}\big(3+\delta(1-\theta)\big)}{3}\right\}^{\frac{q}{2}}.
	\]
	Then, we use Lemma \ref{L.5.7} to obtain
	\begin{align*}
		|v_j-\tilde{U}_i^n|^{q} \mathcal{M}_{\nu,\theta}(\tilde{f}_{i,j,k}^n) &\leq \frac{C_2}{\theta^{\frac{3+\delta}{2}}}\tilde{\rho}_i^n\big((\tilde{T}_\delta)_i^n+|\tilde{U}_i^n|^2\big)^{\frac{q-\delta-3}{2}} \cr
		&\leq
		\frac{C_2}{\theta^{\frac{3+\delta}{2}}}
		\left\{\frac{2^{\frac{q-2\delta-5}{2}}\pi^2(3+\delta)^{\frac{q}{2}}}{q-\delta-5}\right\}\|f^n\|_{L_q^{\infty}}\cr
		&\equiv
		\frac{C_3}{\theta^{\frac{3+\delta}{2}}}
		\|f^n\|_{L_q^{\infty}}.
	\end{align*}
	(c) {\bf The estimate for $I_k^{\frac{q}{\delta}}\mathcal{M}_{\nu,\theta}(\tilde{f}_{i,j,k}^n) $:} From (\ref{E-4}), we have
	\begin{align*}
		\frac{1}{2}(v_j-\tilde{U}_i^n)^{\top}\big((\tilde{\mathcal{T}}_{\nu,\theta})_i^n\big)^{-1}(v_j-\tilde{U}_i^n)+\frac{I_k^{\frac{2}{\delta}}}{(\tilde{T}_{\theta})_i^n}\geq
		\frac{\delta}{\delta+3(1-\theta)}\frac{I_k^{\frac{2}{\delta}}}{(\tilde{T}_{\theta})_i^n},
	\end{align*}
	and hence
	\begin{align*}
		I_k^{\frac{q}{\delta}}\mathcal{M}_{\nu,\theta}(\tilde{f}_{i,j,k}^n)
		&\leq
		\frac{\Lambda_{\delta}}{\sqrt{(2\pi)^3}}I_k^{\frac{q}{\delta}}\frac{1}{\theta^{\frac{3+\delta}{2}}}\frac{\tilde{\rho}_i^n}{\tilde{T}_{\delta}^{\frac{3+\delta}{2}}}\exp\bigg(-\frac{\delta}{\delta+3(1-\theta)}\frac{I_k^{\frac{2}{\delta}}}{(\tilde{T}_{\theta})_i^n}\bigg)
		\cr
		&=
		\frac{\Lambda_{\delta}}{\sqrt{(2\pi)^3}}\frac{1}{\theta^{\frac{3+\delta}{2}}}\big((\tilde{T}_{\delta})_i^n\big)^{\frac{q}{2}}\frac{\tilde{\rho}_i^n}{\big((\tilde{T}_{\delta})_i^n\big)^{\frac{3+\delta}{2}}}\left(\frac{I^{\frac{2}{\delta}}}{((\tilde{T}_{\delta})_i^n)} \right)^{\frac{q}{2}}
		\exp\bigg(-\frac{\delta}{\delta+3(1-\theta)}\frac{I_k^{\frac{2}{\delta}}}{(\tilde{T}_{\theta})_i^n}\bigg)
		\cr
		&\equiv \frac{C_4}{\theta^{\frac{3+\delta}{2}}}\tilde{\rho}_i^n \big((\tilde{T}_{\delta})_i^n\big)^{\frac{q-3-\delta}{2}},
	\end{align*}
	where
	\[
	C_4=\frac{\Lambda_{\delta}}{\sqrt{(2\pi)^3}}\sup_{x>0}|x^{q/2}e^{-x}|\left(\frac{\delta+3(1-\theta)}{\delta}\right)^{q/2}.
	\]
	Next, we use Lemma \ref{L.5.7} to derive
	\begin{align*}
		\begin{split}
		I_k^{\frac{q}{\delta}}\mathcal{M}_{\nu,\theta}(\tilde{f}_{i,j,k}^n)
		&\leq \frac{C_4}{\theta^{\frac{3+\delta}{2}}}\Lambda_{\delta}\tilde{\rho}_i^n\big((\tilde{T}_\delta)_i^n+|\tilde{U}_i^n|^2\big)^{\frac{q-\delta-3}{2}} \cr
		&\leq
		\frac{C_4}{\theta^{\frac{3+\delta}{2}}}\Lambda_{\delta}\left\{\frac{2^{\frac{q-2\delta-5}{2}}\pi^2(3+\delta)^{\frac{q}{2}}}{q-\delta-5}\right\}\|f^n\|_{L_q^{\infty}}\cr
		&\equiv
		\frac{C_5}{\theta^{\frac{3+\delta}{2}}}\|f^n\|_{L_q^{\infty}}.
		\end{split}
	\end{align*}
	Combining  $(a)$, $(b)$ and $(c)$, we finally obtain 
	\begin{align*}
	\sup_{i,j,k} |\mathcal{M}_{\nu,\theta}(\tilde{f}_{i,j,k}^n)(1+|v_j|^2+I_k^{\frac{2}{\delta}})^\frac{q}{2}|&\leq \sup_{i,j,k} |\mathcal{M}_{\nu,\theta}(\tilde{f}_{i,j,k}^n)(1+|v_j-\tilde{U}_i^n+\tilde{U}_i^n|^2+I_k^{\frac{2}{\delta}})^\frac{q}{2}|\cr
	&\leq 2\sup_{i,j,k} |\mathcal{M}_{\nu,\theta}(\tilde{f}_{i,j,k}^n)(1+|v_j-\tilde{U}_i^n|^2+|\tilde{U}_i^n|^2+I_k^{\frac{2}{\delta}})^\frac{q}{2}|\cr
	& \leq C_{\mathcal{M}}\|f^n\|_{L_{q}^{\infty}},
	\end{align*}
	where $C_{\mathcal{M}}$ is a constant depending on $\nu,\delta,\theta$ and $q$ and proportional to $\displaystyle 1/\theta^{\frac{3+\delta}{2}}$.
\end{proof}

\begin{lemma}\label{L.5.11}
	Assume that $f_0$ has no initial error \eqref{no initial err} and satisfies \ref{T.3.1 cond}. Then, $f_0$
	satisfies $E^0$.
\end{lemma}
\begin{proof}
	\noindent $\bullet$ $(A^0)$ From Lemma \ref{L.4.1}, we know
	\begin{align*}
	\|\tilde{f}^0\|_{L_q^\infty}
	\leq  \|f_0\|_{L_q^\infty}.
	\end{align*}
	\noindent $\bullet$ $(B^0)$ Using the lower bound assumption for $f_0$ in \eqref{T.3.1 cond}, we have
	\begin{align*}
	\tilde{f}_{i,j,k}^0 = f_0(x_i-v_j^1 \Delta t, v_j, I_k) \geq  C_0^1e^{-C_0^2(|v_j|^a +I_k^b)}\geq e^{-\frac{A_{\nu,\theta}}{\kappa} T^f} C_0^1e^{-C_0^2(|v_j|^a +I_k^b)}.
	\end{align*}
	\noindent $\bullet$ $(C^0)$ We also have from \eqref{T.3.1 cond} and \eqref{T.5.5 C} that
	\begin{align*}
	\tilde{\rho}_i^0 &= \int_{\mathbb{R}^3\times\mathbb{R}_+} f_0(x_i-v^1 \Delta t, v, I) dvdI\cr
	&\geq C_0^1\int_{\mathbb{R}^3\times\mathbb{R}_+} e^{-C_0^2(|v|^a +I^b)} dvdI \cr
	&= \bar{C}_{a,b}C_0^1\cr
	&\geq \frac{1}{2}\bar{C}_{a,b} C_0^1e^{-\frac{A_{\nu,\theta}}{\kappa} T^f}.
	\end{align*}
	This together with Lemma \ref{L.5.6} gives
	\begin{align*}
	(\tilde{T}_{\delta})_i^0 \geq \bigg(\frac{\tilde{\rho}_i^0}{C_{\delta}\|f^0\|_{L_q^{\infty}}}\bigg)^{\frac{2}{3+\delta}} \geq \bigg(\frac{\bar{C}_{a,b}C_0^1}{C_{\delta}\|f_0\|_{L_{q}^{\infty}}}\bigg)^{\frac{2}{3+\delta}} \geq \bigg( \frac{1}{2}\frac{\bar{C}_{a,b} C_0^1}{C_\delta \|f_0\|_{L_{q}^{\infty}}}e^{-\big(\frac{1}{\kappa}  +\frac{ C_{\mathcal{M}} }{\kappa + A_{\nu,\theta} \Delta t}\big)T^f}\bigg)^{\frac{2}{3+\delta}},
	\end{align*}
	where $C_{\delta}$ is a constant given in Lemma \ref{L.5.6}.
	
	\noindent $\bullet$ $(D^0)$ Using \eqref{T.5.5 C}, we obtain the upper bounds for $\tilde{\rho}_i^0,|\tilde{U}_i^0|$ and $(\tilde{T}_\delta)_i^0$ as follows:
	\begin{align*}
		\tilde{\rho}_i^0 
		&=  \int_{\mathbb{R}^3\times\mathbb{R}_+} f_0(x_i-v^1 \Delta t,v,I)\frac{(1+|v|^2 + I^\frac{2}{\delta})^\frac{q}{2}}{(1+|v|^2 + I^\frac{2}{\delta})^\frac{q}{2}} dvdI\cr
		&\leq \|f_0\|_{L_q^\infty}\int_{\mathbb{R}^3\times\mathbb{R}_+} \frac{1}{(1+|v|^2 + I^\frac{2}{\delta})^\frac{q}{2}} dvdI\cr
		&= \bar{C}_{\delta,q}\|f_0\|_{L_q^\infty},\cr
		|\tilde{U}_i^0| 
		&\leq \frac{1}{\tilde{\rho}_i^0}  \bigg|\int_{\mathbb{R}^3\times\mathbb{R}_+} f_0(x_i-v^1 \Delta t,v,I)\frac{(1+|v|^2 + I^\frac{2}{\delta})^\frac{q}{2}}{(1+|v|^2 + I^\frac{2}{\delta})^\frac{q}{2}} |v|dvdI\bigg|\cr
		&\leq \frac{\|f_0\|_{L_q^\infty}}{\tilde{\rho}_i^0}\int_{\mathbb{R}^3\times\mathbb{R}_+} \frac{1}{(1+|v|^2 + I^\frac{2}{\delta})^{\frac{q-1}{2}} } dvdI\cr
		&\leq  \frac{\bar{C}_{\delta,q-1}}{\bar{C}_{a,b}C_0^1}\|f_0\|_{L_q^\infty},
	\end{align*}
and
	\begin{align*}
	(\tilde{T}_\delta)_i^0 &= \frac{2}{3+\delta}\frac{1}{\tilde{\rho}_i^0} \int_{\mathbb{R}^3\times\mathbb{R}_+}\left(\frac{1}{2}|v-\tilde{U}_i^0|^2 + I^{\frac{2}{\delta}}\right)f_0(x_i-v^1 \Delta t,v,I)dvdI\cr
	&\leq\frac{2}{3+\delta}\bigg(\frac{1}{\tilde{\rho}_i^0}\int_{\mathbb{R}^3\times\mathbb{R}_+}  f_0(x_i-v^1 \Delta t,v,I) (|v|^2 + I^{\frac{2}{\delta}})dvdI - |\tilde{U}_i^0|^2\bigg)\cr
	&\leq\frac{2}{3+\delta}\frac{1}{\tilde{\rho}_i^0}\int_{\mathbb{R}^3\times\mathbb{R}_+}  f_0(x_i-v^1 \Delta t,v,I)  \frac{(1+|v|^2 + I^\frac{2}{\delta})^\frac{q}{2}}{(1+|v|^2 + I^\frac{2}{\delta})^\frac{q}{2}}(|v|^2 + I^{\frac{2}{\delta}}) dvdI  \cr
	&\leq \frac{2}{3+\delta}\frac{\|f_0\|_{L_q^\infty}}{\tilde{\rho}_i^0}\int_{\mathbb{R}^3\times\mathbb{R}_+} \frac{1}{(1+|v|^2 + I^\frac{2}{\delta})^{\frac{q-2}{2}} } dvdI\cr
	&\leq \frac{2}{3+\delta}\frac{\bar{C}_{\delta,q-2}}{\bar{C}_{a,b}C_0^1}\|f_0\|_{L_q^\infty}.
	\end{align*}
\end{proof}
\begin{lemma}\label{L.5.12}
	Assume $f_{i,j,k}^{n-1}$ satisfies $E^{n-1}$. Then, $f_{i,j,k}^n$ satisfies $A^n$:
	\begin{align*}
	(A^n)& \quad \|\tilde{f}^n\|_{L_q^\infty}  
	\leq \bigg(\frac{\kappa  +   A_{\nu,\theta} \Delta t C_{\mathcal{M}}}{\kappa + A_{\nu,\theta} \Delta t}\bigg)^n \|f_0\|_{L_{q}^{\infty}} 
	\leq e^{\frac{ C_{\mathcal{M}}A_{\nu,\theta}  }{\kappa + A_{\nu,\theta} \Delta t}T^f} \|f_0\|_{L_{q}^{\infty}}.
	\end{align*}
\end{lemma}
\begin{proof}
	Recall \eqref{B-6} and use Lemma \ref{L.4.1} and Lemma \ref{L.5.9} to obtain
	\begin{align*}
		\begin{split}
			\|f^n\|_{L_q^\infty} &\leq\frac{\kappa \|\tilde{f}^{n-1}\|_{L_q^\infty} +   A_{\nu,\theta} \Delta t \|\mathcal{M}_{\nu,\theta}(\tilde{f}^{n-1})\|_{L_q^\infty}}{\kappa + A_{\nu,\theta} \Delta t}\cr
			&\leq \frac{\kappa  +   A_{\nu,\theta} \Delta t C_{\mathcal{M}}}{\kappa + A_{\nu,\theta} \Delta t} \|f^{n-1}\|_{L_{q}^{\infty}}\cr
			&\leq \bigg(\frac{\kappa  +   A_{\nu,\theta} \Delta t C_{\mathcal{M}}}{\kappa + A_{\nu,\theta} \Delta t}\bigg)^{n} \|f_0\|_{L_{q}^{\infty}}.
		\end{split}
	\end{align*}
	Now, we make use of $(1+x)^n \leq e^{nx}$ to see
	\begin{align*}
		\bigg(\frac{\kappa  +   A_{\nu,\theta} \Delta t C_{\mathcal{M}}}{\kappa + A_{\nu,\theta} \Delta t}\bigg)^{n}= \bigg(1 + \frac{ (C_{\mathcal{M}}-1)A_{\nu,\theta} \Delta t }{\kappa + A_{\nu,\theta} \Delta t}\bigg)^{n} 
		\leq e^{\frac{ C_{\mathcal{M}}A_{\nu,\theta} }{\kappa + A_{\nu,\theta} \Delta t}T^f}.
	\end{align*}
	 Note that $C_{\mathcal{M}}>1$ and this estimate holds uniformly for $n\geq 0$.
\end{proof}

\begin{lemma}\label{L.5.13}
	Assume $f_{i,j,k}^{n-1}$ satisfies $E^{n-1}$. Then, $f_{i,j,k}^n$ satisfies $B^n$:
	\begin{align*}
	(B^n)& \quad \tilde{f}_{i,j,k}^n \geq \bigg(\frac{\kappa}{\kappa + A_{\nu,\theta} \Delta t}\bigg)^n C_0^1e^{-C_0^2(|v_j|^a +I_k^b)} \geq e^{-\frac{A_{\nu,\theta}}{\kappa} T^f} C_0^1e^{-C_0^2(|v_j|^a +I_k^b)}.
	\end{align*}
\end{lemma}
\begin{proof}
	From the non-negativity of $\mathcal{M}_{\nu,\theta}$ and \eqref{B-6}, we have
	\begin{align*}
		f_{i,j,k}^n \geq \frac{\kappa}{\kappa + A_{\nu,\theta} \Delta t}\tilde{f}_{i,j,k}^{n-1} = \frac{\kappa}{\kappa + A_{\nu,\theta} \Delta t}\big(a_jf_{s,j,k}^{n-1} + (1-a_j) f_{s+1,j,k}^{n-1}).
	\end{align*}
	We recall \eqref{B-2}, $0\leq a_j \leq 1$ and use the lower bound of $f_{i,j,k}^{n-1}$ in Lemma \ref{L.5.11} to obtain  
	\begin{align*}
		f_{i,j,k}^n &\geq \frac{\kappa}{\kappa + A_{\nu,\theta} \Delta t}\big(a_jf_{s,j,k}^{n-1} + (1-a_j) f_{s+1,j,k}^{n-1})\cr
		& \geq \frac{\kappa}{\kappa + A_{\nu,\theta} \Delta t} \bigg(\frac{\kappa}{\kappa + A_{\nu,\theta} \Delta t}\bigg)^{n-1} \min_{i} f_{i,j,k}^0\cr
		& \geq  \bigg(\frac{\kappa}{\kappa + A_{\nu,\theta} \Delta t}\bigg)^n C_0^1e^{-C_0^2(|v_j|^a +I_k^b)}.
	\end{align*}
	Using $(1+x)^{-n} \geq e^{-nx}$, we complete the proof.
\end{proof}

\begin{lemma}\label{L.5.14}
	Assume $f_{i,j,k}^n$ satisfies $A^n \wedge B^n$. Then, $f_{i,j,k}^n$ satisfies $C^n$:
	\begin{align*}
		(C^n)& \quad \tilde{\rho}_i^n \geq \frac{1}{2}\bar{C}_{a,b} C_0^1e^{-\frac{A_{\nu,\theta}}{\kappa} T^f}, \quad (\tilde{T}_\delta)_i^n \geq  \bigg( \frac{1}{2}\frac{\bar{C}_{a,b} C_0^1}{C_\delta \|f_0\|_{L_{q}^{\infty}}}e^{-\big(\frac{1}{\kappa}  +\frac{ (C_{\mathcal{M}}-1) }{\kappa + A_{\nu,\theta} \Delta t}\big)T^f}\bigg)^{\frac{2}{3+\delta}}.  
	\end{align*}

\end{lemma}
\begin{proof}
	Since Lemma \ref{L.5.13} holds, the discrete local density $\tilde{\rho}_i^n$ satisfies
	\[
	\tilde{\rho}_i^n = \sum_{j,k} \tilde{f}_{i,j,k}^{n} (\Delta v)^3 \Delta I \geq   C_0^1e^{-\frac{A_{\nu,\theta}}{\kappa} T^f}  \sum_{j,k}e^{-C_0^2(|v_j|^a +I_k^b)} (\Delta v)^3 \Delta I \geq   \frac{1}{2}\bar{C}_{a,b} C_0^1e^{-\frac{A_{\nu,\theta}}{\kappa} T^f}.
	\]
	This, together with Lemma \ref{L.5.6} and Lemma \ref{L.5.12},
	gives 
	\begin{align*}
	(\tilde{T}_{\delta})_i^n \geq \bigg(\frac{\tilde{\rho}_i^n}{C_{\delta}\|f^n\|_{L_q^{\infty}}}\bigg)^{\frac{2}{3+\delta}} \geq \bigg( \frac{1}{2}\frac{\bar{C}_{a,b} C_0^1}{C_\delta \|f_0\|_{L_{q}^{\infty}}}e^{-\big(\frac{1}{\kappa}  +\frac{ C_{\mathcal{M}} }{\kappa + A_{\nu,\theta} \Delta t}\big)T^f}\bigg)^{\frac{2}{3+\delta}} .
	\end{align*}
\end{proof}
\begin{lemma}\label{L.5.15}
	Assume $f_{i,j,k}^n$ satisfies $A^n \wedge B^n \wedge C^n$. Then, $f_{i,j,k}^n$ satisfies $D^n$:
	\begin{align*}
		(D^n)& \quad \|\tilde{\rho}^n\|_{L_x^\infty} \leq 2 \bar{C}_{\delta,q} e^{\frac{ C_{\mathcal{M}}A_{\nu,\theta}  }{\kappa + A_{\nu,\theta} \Delta t}T^f} \|f_0\|_{L_{q}^{\infty}},\cr
		& \quad \|\tilde{U}^n\|_{L_x^\infty} \leq \frac{4\bar{C}_{\delta,q-1}}{\bar{C}_{a,b} C_0^1} e^{\big(\frac{1}{\kappa}  + \frac{ C_{\mathcal{M}}  }{\kappa + A_{\nu,\theta} \Delta t}\big)A_{\nu,\theta}T^f} \|f_0\|_{L_{q}^{\infty}},\cr
		& \quad \|(\tilde{T}_\delta)^n\|_{L_x^\infty} \leq  \frac{8}{3+\delta}\frac{\bar{C}_{\delta,q-2}}{\bar{C}_{a,b} C_0^1}e^{\big(\frac{1}{\kappa}  + \frac{ C_{\mathcal{M}}  }{\kappa + A_{\nu,\theta} \Delta t}\big)A_{\nu,\theta}T^f} \|f_0\|_{L_{q}^{\infty}}.
	\end{align*}

\end{lemma}
\begin{proof}
	From the upper bound for $\|\tilde{f}^n\|_{L_q^\infty}$ in Lemma \ref{L.5.12}, we see that
	\begin{align*}
		\tilde{\rho}_i^n 
		&= \sum_{j,k} \tilde{f}_{i,j,k}^{n} \frac{(1+|v_j|^2 + I_k^\frac{2}{\delta})^\frac{q}{2}}{(1+|v_j|^2 + I_k^\frac{2}{\delta})^\frac{q}{2}} (\Delta v)^3 \Delta I\cr
		&\leq \|\tilde{f}^n\|_{L_q^\infty}\sum_{j,k} \frac{1}{(1+|v_j|^2 + I_k^\frac{2}{\delta})^{\frac{q}{2}} } (\Delta v)^3 \Delta I \cr
		&\leq 2 \bar{C}_{\delta,q} e^{\frac{ C_{\mathcal{M}}A_{\nu,\theta}  }{\kappa + A_{\nu,\theta} \Delta t}T^f} \|f_0\|_{L_{q}^{\infty}}.
	\end{align*}
	To estimate $\tilde{U}_i^n$, we use the upper bound of $\|\tilde{f}^n\|_{L_q^\infty}$ in Lemma \ref{L.5.12} and the lower bound of $\tilde{\rho}_i^n$ in Lemma \ref{L.5.14}:
	\begin{align*}
		|\tilde{U}_i^n| 
		&= \frac{1}{\tilde{\rho}_i^n}\sum_{j,k} \tilde{f}_{i,j,k}^{n} \frac{(1+|v_j|^2 + I_k^\frac{2}{\delta})^\frac{q}{2}}{(1+|v_j|^2 + I_k^\frac{2}{\delta})^\frac{q}{2}} |v_j|(\Delta v)^3 \Delta I\cr
		&\leq \frac{\|\tilde{f}^n\|_{L_q^\infty}}{\tilde{\rho}_i^n}\sum_{j,k} \frac{1}{(1+|v_j|^2 + I_k^\frac{2}{\delta})^{\frac{q-1}{2}} } (\Delta v)^3 \Delta I \cr
		&\leq 2\bar{C}_{\delta,q-1}\bigg(\frac{1}{2}\bar{C}_{a,b} C_0^1e^{-\frac{A_{\nu,\theta}}{\kappa} T^f}\bigg)^{-1} e^{\frac{ C_{\mathcal{M}}A_{\nu,\theta}  }{\kappa + A_{\nu,\theta} \Delta t}T^f} \|f_0\|_{L_{q}^{\infty}}\cr
		&= \frac{4\bar{C}_{\delta,q-1}}{\bar{C}_{a,b} C_0^1} e^{\big(\frac{1}{\kappa}  + \frac{ C_{\mathcal{M}}  }{\kappa + A_{\nu,\theta} \Delta t}\big)A_{\nu,\theta}T^f} \|f_0\|_{L_{q}^{\infty}}.
	\end{align*}
	Similarly, we compute
	\begin{align*}
		(\tilde{T}_{\delta})_i^n &= \frac{2}{3+\delta}\frac{1}{\tilde{\rho}_i^n}\sum_{j,k} \tilde{f}_{i,j,k}^{n} \left(\frac{|v_j-\tilde{U}_i^n|^2}{2} + I^{\frac{2}{\delta}}\right) (\Delta v)^3 \Delta I \cr 	&\leq\frac{2}{3+\delta}\bigg(\frac{1}{\tilde{\rho}_i^n}\sum_{j,k} \tilde{f}_{i,j,k}^{n} (|v_j|^2 + I_k^{\frac{2}{\delta}})(\Delta v)^3 \Delta I - |\tilde{U}_i^n|^2\bigg)\cr
		&\leq\frac{2}{3+\delta}\frac{1}{\tilde{\rho}_i^n}\sum_{j,k} \tilde{f}_{i,j,k}^{n}  \frac{(1+|v_j|^2 + I_k^\frac{2}{\delta})^\frac{q}{2}}{(1+|v_j|^2 + I_k^\frac{2}{\delta})^\frac{q}{2}} (|v_j|^2 + I_k^{\frac{2}{\delta}}) (\Delta v)^3 \Delta I.
	\end{align*}
	Then, from $A^n$ and $C^n$, we have
	\begin{align*}
		(\tilde{T}_{\delta})_i^n&\leq \frac{4}{3+\delta}\bar{C}_{\delta,q-2}\bigg(\frac{1}{2}\bar{C}_{a,b} C_0^1e^{-\frac{A_{\nu,\theta}}{\kappa} T^f}\bigg)^{-1}e^{\frac{ C_{\mathcal{M}}A_{\nu,\theta}  }{\kappa + A_{\nu,\theta} \Delta t}T^f} \|f_0\|_{L_{q}^{\infty}}\cr
		&= \frac{8}{3+\delta}\frac{\bar{C}_{\delta,q-2}}{\bar{C}_{a,b} C_0^1}e^{\big(\frac{1}{\kappa}  + \frac{ C_{\mathcal{M}}  }{\kappa + A_{\nu,\theta} \Delta t}\big)A_{\nu,\theta}T^f} \|f_0\|_{L_{q}^{\infty}}.
	\end{align*}
\end{proof}
Now, we have built up all ingredients to prove the Theorem \ref{T.5.5}.

\noindent{\bf{Proof of Theorem \ref{T.5.5}}.} The proof is based on the induction argument. Lemma \ref{L.5.11} implies $E^0$. For $n \geq 1$, one can easily confirm that Lemma \ref{L.5.12} - \ref{L.5.15} gives $E^n$.

\section{Consistent form}
In this section, we rewrite (\ref{A-1}) in a consistent form
to make it easily comparable with (\ref{B-4}). For convenience, we introduce the following notation: 
\begin{align*} 
	\begin{array}{lll}
		\text{$\bullet$ Distribution function on $x-v^1\Delta t$:} &\qquad\qquad\qquad  &\qquad\qquad \\
		&\tilde{f}(x,v,t,I):=f(x-v^1\Delta t, v,t,I).&  \\
		\text{$\bullet$ Mass:} &  & \\
		&\tilde{\rho}(x,t)= \int_{\mathbb{R}^3\times\mathbb{R}_+}\tilde{f}(x,v,t,I)dvdI. &\\
		\text{$\bullet$ Momentum:} &\qquad\qquad\qquad  &\qquad\qquad \\
		&\tilde{\rho}(x,t)\tilde{U}(x,t):= \int_{\mathbb{R}^3\times\mathbb{R}_+}v\tilde{f}(x,v,t,I)dvdI.&\\
		\text{$\bullet$ Stress tensor:} &  & \\
		&\tilde{\rho}(x,t)\tilde{\Theta}(x,t) = \int_{\mathbb{R}^3\times\mathbb{R}_+}(v-\tilde{U}(x,t))\otimes(v-\tilde{U}(x,t)) \tilde{f}(x,v,t,I)dvdI.&\\
		\text{$\bullet$ Polyatomic temperature:} &  & \\
		&\tilde{T}_\delta(x,t) = \frac{3}{3+\delta} \tilde{T}_{tr}(x,t) + \frac{\delta}{3+\delta}\tilde{T}_{I,\delta}(x,t),&\\
		\text{where} &  & \\
		&(\tilde{T}_{tr})(x,t):= \frac{2}{3}\frac{1}{\tilde{\rho}(x,t)}\int_{\mathbb{R}^3\times\mathbb{R}_+} \frac{|v-\tilde{U}(x,t)|^2}{2} \tilde{f}(x,v,t,I) dvdI,&\\
		&\tilde{T}_{I,\delta}(x,t):=  \frac{2}{\delta}\frac{1}{\tilde{\rho}(x,t)}\int_{\mathbb{R}^3\times\mathbb{R}_+} I^{\frac{2}{\delta}}  \tilde{f}(x,v,t,I) dvdI.&\\
		\text{$\bullet$ Relaxation temperature:} &  & \\
		&\tilde{T}_\theta(x,t):= \theta \tilde{T}_\delta(x,t) + (1-\theta)\tilde{T}_{I,\delta}(x,t),& \\
		&&\\
		\text{$\bullet$ Polyatomic temperature:} &  & \\
		&\tilde{\mathcal{T}}_{\nu,\theta}(x,t):=\theta \tilde{T}_\delta(x,t) Id + (1-\theta)(1-\nu)\tilde{T}_{tr}(x,t)Id  +(1-\theta)\nu  \tilde{\Theta}(x,t).&	
	\end{array}
\end{align*}

\begin{lemma}\label{L.6.1}
	The equation (\ref{A-1}) can be rewritten as
	\begin{align}\label{consistent form}
		\begin{split}
		f(x,v,t+\Delta t,I)&=\frac{\kappa}{\kappa+A_{\nu,\theta}\Delta t}\tilde{f}(x,v,t,I) + \frac{A_{\nu,\theta}\Delta t}{\kappa+A_{\nu,\theta}\Delta t}\mathcal{M}_{\nu,\theta}(\tilde{f})(x,v,t,I)
		\cr
		&\quad + R_1 + R_2,
		\end{split}
	\end{align}
	with
	\begin{align*}
		R_1&= - \frac{A_{\nu,\theta}}{\kappa+A_{\nu,\theta}\Delta t}\int_{t}^{t+\Delta t}  \{\mathcal{M}_{\nu,\theta}(\tilde{f})(x,v,t,I) - \mathcal{M}_{\nu,\theta}(f)(x,v,t,I)\} ds\cr
		&\quad - \frac{A_{\nu,\theta}}{\kappa+A_{\nu,\theta}\Delta t}\int_{t}^{t+\Delta t}  (t+\Delta t-s)v^1\partial_x \mathcal{M}_{\nu,\theta}(f)(x_{\theta_1},v,t_{\theta_1},I) \cr
		&\hspace{5cm}- (s-t) \partial_t \mathcal{M}_{\nu,\theta}(f)(x_{\theta_1},v,t_{\theta_1},I) ds,\cr
		R_2&=- \frac{A_{\nu,\theta}}{\kappa+A_{\nu,\theta}\Delta t}\int_{t}^{t+\Delta t} (s-t-\Delta t)A_{\nu,\theta} (\mathcal{M}_{\nu,\theta}(f) - f)(x_{\theta_2},v,t_{\theta_2},I) ds
	\end{align*}
where $x_{\theta_i}$, $i=1,2$, lies between $x$ and $x_i-v^1 \Delta t$ and $t_{\theta_i}$ between $t$ and $t+\Delta t$.	
\end{lemma}  
\begin{proof}
We start by integrating \eqref{B-3} from $t$ to $t+ \Delta t$:
\begin{align*}
f(x,v,t+\Delta t,I)&=f(x-v^1 \Delta t,v,t,I) \cr
&\quad + \frac{A_{\nu,\theta}}{\kappa}\int_{t}^{t+\Delta t}\left( \mathcal{M}_{\nu,\theta}(f) - f \right)(x-(t+\Delta t-s)v^1,v,s,I)ds.
\end{align*}
Using Taylor's theorem, we obtain
\begin{align}\label{cons 1}
\begin{split}
\mathcal{M}_{\nu,\theta}(f)&(x-(t+\Delta t-s)v^1,v,s,I)\cr
&=\mathcal{M}_{\nu,\theta}(f)(x,v,t,I) - (t+\Delta t-s)v^1\partial_x \mathcal{M}_{\nu,\theta}(f)(x_{\theta_1},v,t_{\theta_1},I) \cr
&\quad + (s-t) \partial_t \mathcal{M}_{\nu,\theta}(f)(x_{\theta_1},v,t_{\theta_1},I) \cr
&=\{\mathcal{M}_{\nu,\theta}(f)(x,v,t,I) - \mathcal{M}_{\nu,\theta}(\tilde{f})(x,v,t,I)\} \cr &\quad + \mathcal{M}_{\nu,\theta}(\tilde{f})(x,v,t,I)- (t+\Delta t-s)v^1\partial_x \mathcal{M}_{\nu,\theta}(f)(x_{\theta_1},v,t_{\theta_1},I) \cr
&\quad + (s-t) \partial_t \mathcal{M}_{\nu,\theta}(f)(x_{\theta_1},v,t_{\theta_1},I),
\end{split}
\end{align}
for some $x_{\theta_1}$ between $x$ and $x-(t+\Delta t-s)v^1$ and $t_{\theta_1}$ between $t$ and $t+\Delta t$. Similarly, 
\begin{align}\label{cons 2}
\begin{split}
f(x-&(t+\Delta t-s)v^1,v,s,I)\cr
&=f(x,v,t+\Delta t,I) - (t+\Delta t-s)v^1\partial_x f(x_{\theta_1},v,t_{\theta_1},I)  + (s-t-\Delta t) \partial_t f(x_{\theta_1},v,t_{\theta_1},I) \cr
&=f(x,v,t+\Delta t,I) + (s-t-\Delta t)\{\partial_t + v^1\partial_x\} f(x_{\theta_2},v,t_{\theta_2},I) \cr
&=f(x,v,t+\Delta t,I) + (s-t-\Delta t)\frac{A_{\nu,\theta}}{\kappa} (\mathcal{M}_{\nu,\theta}(f) - f)(x_{\theta_2},v,t_{\theta_2},I) .
\end{split}
\end{align}
Combining \eqref{cons 1} and \eqref{cons 2}, we can derive the desired representation.
\end{proof}


\begin{proposition}\label{P.6.1}
	\cite{PY} Let $f$ and $g$ satisfy ($\mathcal{A}1$) and ($\mathcal{A}2$) in Theorem \ref{T.3.1}.	
	Then $\mathcal{M}_{\nu,\theta}$ satisfies the following continuity property:
	\begin{align*}
	\|\mathcal{M}_{\nu,\theta}(f)-\mathcal{M}_{\nu,\theta}(g)\|_{L_q^{\infty}} \leq C_{Lip}\|f-g\|_{L_q^{\infty}}
	\end{align*}
	for some constant $C_{Lip}$ depending on $T^f, \delta, \theta, q$ and $f_0$.
\end{proposition}
\begin{proposition}\label{P.6.2}
	\cite{PY} Let $\delta>0$, $-1/2<\nu < 1$ and $0< \theta \leq 1$. Suppose $\rho>0$, $T_{tr}>0$ and $T_{I,\delta}>0$. Then, temperature tensor $\mathcal{T}_{\nu,\theta}$ and the relaxation temperature $T_{\theta}$ satisfy the following equivanlenec type estimates:
	\begin{align*}
	&(1)\ \theta T_{\delta}Id \leq \mathcal{T}_{\nu,\theta} \leq C_{\nu}\frac{(3+\delta-\delta\theta)}{3} T_{\delta}Id, \cr 
	&(2)\ \theta T_{\delta} \leq T_{\theta} \leq \frac{(3+\delta-3\theta)}{\delta} T_{\delta},
	\end{align*}
	where the constants $C_{\nu}=\max_{\nu}\{1-\nu,1+2\nu\}$. 
\end{proposition}
\begin{proposition}\label{P.6.3}
	\cite{PY} Let $\delta>0$, $-1/2<\nu < 1$, $0<\theta \leq 1$,  $q>5+\delta$. Suppose $f \in \Omega_{0,q}$, there exists a constant $C$ depending on $\nu, \delta,\theta$ and $q$ such that
	\begin{align*}
	\|\mathcal{M}_{\nu,\theta}(f)\|_{L_{q}^{\infty}} \leq C\|f\|_{L_{q}^{\infty}},
	\end{align*}
	where $C$ blows up as $\theta$ tends to 0. 
\end{proposition}

\begin{proposition}\label{Mt estimate}
	Let $f$ be a smooth solution to \eqref{A-1} in $\Omega_{1,q}$ corresponding to $f_0$. Then, for $q>5+\delta$, $\delta>0$, we have
	\begin{align*}
		\|\partial_t \mathcal{M}_{\nu,\theta}\|_{L^\infty},\ \|\nabla_x \mathcal{M}_{\nu,\theta}\|_{L^\infty} <C\{ \|f_0\|_{L_{1,q}^\infty} + 1 \},
	\end{align*} 
	where $C$ is a positive constant which depends on $\nu, \delta, q, \theta, f_0, T^f$.	
\end{proposition}
\begin{proof}
	We begin by estimating the time derivative of macroscopic quantities. Using the collision invariants, $1,\,v_j,\, \frac{1}{2}|v|^2+I^{\frac{2}{\delta}}$, we obtain
	\begin{align}\label{rho t estimate}
		\begin{split}
		\left|\frac{d}{dt} \int_{\mathbb{R}^3\times\mathbb{R}^+} f \begin{pmatrix}
		1\\v\\\frac{1}{2}|v|^2+I^{\frac{2}{\delta}}
		\end{pmatrix} dvdI\right| &= \left| \int_{\mathbb{R}^3\times\mathbb{R}^+} v \cdot \nabla_x f \begin{pmatrix}
		1\\v\\\frac{1}{2}|v|^2+I^{\frac{2}{\delta}}
		\end{pmatrix} dvdI\right| \cr
		&\quad\leq C\left| \int_{\mathbb{R}^3\times\mathbb{R}^+} |v|  |\nabla_x f| \big(1+|v|^2 + I^{\frac{2}{\delta}} \big)dvdI\right| \cr
		&\quad\leq C \|f(t)\|_{L_{1,q}^\infty} \left| \int_{\mathbb{R}^3\times\mathbb{R}^+}  \frac{1}{(1+|v|^2+I^\frac{2}{\delta})^{\frac{q}{2}-2}}  dvdI\right|\cr
		&\quad\leq C\{ \|f_0\|_{L_{1,q}^\infty} + 1 \},
		\end{split}
	\end{align}
	which gives $|\partial_t\rho|, |\partial_t\{\rho U\}| < C\{ \|f_0\|_{L_{1,q}^\infty} + 1 \}$. Using the lower bound for $\rho$ and the upper bound for $\rho + |U| + T_\delta$ in Theorem \ref{T.3.1}, we further obtain 
	\begin{align}\label{Ut}
		|\partial_t U| &\leq \frac{1}{\rho}\bigg(|\partial_t \rho| |U| + C\{ \|f_0\|_{L_{1,q}^\infty} + 1 \}\bigg)
		\leq C\{ \|f_0\|_{L_{1,q}^\infty} + 1 \}.
	\end{align}
	To bound $|\partial_t E_\delta|$, we start from
	\begin{align}\label{Et split}
	\begin{split}
	|\partial_t E_\delta| &= \left|\frac{d}{dt} \int_{\mathbb{R}^3\times\mathbb{R}^+} f  \bigg(\frac{1}{2}|v-U|^2+I^{\frac{2}{\delta}}\bigg) dvdI\right| \cr
	& \leq \left| \int_{\mathbb{R}^3\times\mathbb{R}^+} v \cdot \nabla_x f \bigg(\frac{1}{2}|v-U|^2+I^{\frac{2}{\delta}}\bigg) dvdI\right| + \left| \int_{\mathbb{R}^3\times\mathbb{R}^+} f \bigg(|v-U||\partial_t U|\bigg) dvdI\right|\cr
	& \equiv \mathcal{I}_{41} + \mathcal{I}_{42}.
	\end{split}
	\end{align}
	$\mathcal{I}_{41}$ satisfies
	\begin{align}\label{Et I1}
	\begin{split}
	\mathcal{I}_{41} &\leq \left| \int_{\mathbb{R}^3\times\mathbb{R}^+} v \cdot \nabla_x f \bigg(\frac{1}{2}|v|^2+I^{\frac{2}{\delta}}\bigg) dvdI\right| + \left| \int_{\mathbb{R}^3\times\mathbb{R}^+} v \cdot \nabla_x f \bigg(|v||U| + \frac{|U|^2}{2}+I^{\frac{2}{\delta}}\bigg) dvdI\right|.
	\end{split}
	\end{align}
	In \eqref{Et I1}, the first term of the upper bound can be estimated by \eqref{rho t estimate}. The second term is bounded by
	\begin{align}\label{Et I1 r.h.s}
	\begin{split}
	&\left| \int_{\mathbb{R}^3\times\mathbb{R}^+} v \cdot \nabla_xf \bigg(|v||U| + \frac{|U|^2}{2}+I^{\frac{2}{\delta}}\bigg) dvdI\right|\cr
	&\quad \leq \|\nabla_x \cdot f\|_{L_{q}^\infty}  \int_{\mathbb{R}^3\times\mathbb{R}^+}  |v|\frac{|v||U| + \frac{|U|^2}{2}+I^{\frac{2}{\delta}}}{(1+|v|^2+I^\frac{2}{\delta})^{\frac{q}{2}}} dvdI\cr
	&\quad \leq \|\nabla_x \cdot f\|_{L_{q}^\infty} \left| \int_{\mathbb{R}^3\times\mathbb{R}^+}  |v|\frac{\frac{|v|^2}{2} + |U|^2+I^{\frac{2}{\delta}}}{(1+|v|^2+I^\frac{2}{\delta})^{\frac{q}{2}}} dvdI\right|\cr
	&\quad \leq \max\{1,|U|^2\}\|\nabla_x \cdot f\|_{L_{q}^\infty} \left| \int_{\mathbb{R}^3\times\mathbb{R}^+}  \frac{1}{(1+|v|^2+I^\frac{2}{\delta})^{\frac{q}{2}-\frac{3}{2}}} dvdI\right|\cr
	&\quad \leq C\{ \|f_0\|_{L_{1,q}^\infty} + 1 \}, 
	\end{split}
	\end{align}
	where we use the boundedness of $|U|$ in Theorem \ref{T.3.1} and $q>5+\delta$. 
	
	To estimate $\mathcal{I}_{42}$, we use the boundedness of $f$ and $U$ in Theorem \ref{T.3.1} and $\partial_t U$ in \eqref{Ut}: \begin{align}\label{Et I2}
	\begin{split}
	\mathcal{I}_{42} &\leq |\partial_t U| \|f\|_{L_{q}^\infty} \left| \int_{\mathbb{R}^3\times\mathbb{R}^+}   \frac{|v|+|U|}{(1+|v|^2+I^\frac{2}{\delta})^{\frac{q}{2}}} dvdI\right|\cr
	&\leq \max\{1,|U|\}|\partial_t U| \|f\|_{L_{q}^\infty} \left| \int_{\mathbb{R}^3\times\mathbb{R}^+}   \frac{|v|+1}{(1+|v|^2+I^\frac{2}{\delta})^{\frac{q}{2}}} dvdI\right|\cr
	&\leq C\{ \|f_0\|_{L_{1,q}^\infty} + 1 \}.
	\end{split}
	\end{align}
	Combining \eqref{Et split}, \eqref{Et I1}, \eqref{Et I1 r.h.s} and \eqref{Et I2}, we obtain
	\begin{align*}
	\begin{split}
	|\partial_t E_\delta| &\leq C\{ \|f_0\|_{L_{1,q}^\infty} + 1 \}.
	\end{split}
	\end{align*}
	Now, we use the relation $\displaystyle E_\delta = \frac{3+\delta}{2} \rho T_\delta$ and the lower and upper bounds for $\rho$ and $T_\delta$ in Theorem \ref{T.3.1}, which together with $|\partial_t \rho|, |\partial_t E_\delta| \leq C\{ \|f_0\|_{L_{1,q}^\infty} + 1 \}$ give
	\begin{align*}
		|\partial_t T_\delta| &= \frac{1}{\rho}\bigg(\frac{2}{3+\delta} |\partial_t E_\delta| + |\partial_t \rho| T_\delta \bigg) \leq C\{ \|f_0\|_{L_{1,q}^\infty} + 1 \}.
	\end{align*}
	Similarly, we compute
	\begin{align}\label{TIdelta split}
	\begin{split}
	|\partial_t T_{I,\delta}| &= \left|\frac{d}{dt} \int_{\mathbb{R}^3\times\mathbb{R}^+} f  I^{\frac{2}{\delta}} dvdI\right| \cr
	& \leq \left| \int_{\mathbb{R}^3\times\mathbb{R}^+} v \cdot \nabla_x f I^{\frac{2}{\delta}} dvdI\right| + \frac{A_{\nu,\theta}}{\kappa}\left| \int_{\mathbb{R}^3\times\mathbb{R}^+}( \mathcal{M}_{\nu,\theta}-f)   I^{\frac{2}{\delta}} dvdI\right|\cr
	& = \left| \int_{\mathbb{R}^3\times\mathbb{R}^+} v \cdot \nabla_x f I^{\frac{2}{\delta}} dvdI\right| + \frac{A_{\nu,\theta}}{\kappa}\left| \frac{\delta}{2}\rho T_\theta - E_{I,\delta} \right|.
	\end{split}
	\end{align}
	In the last line, the first term can be bounded by \eqref{rho t estimate}. For the second term, we use 
	\[T_\theta = \theta T_\delta + (1-\theta) T_{I,\delta}, \quad T_\delta = \frac{3}{3+\delta} T_{tr} + \frac{\delta}{3+\delta}T_{I,\delta}, \quad E_{I,\delta}= \frac{\delta}{2} \rho T_{I,\delta},
	\]
	to obtain
	\begin{align}\label{TIdelta split r.h.s}
	\begin{split}
	\left| \frac{\delta}{2}\rho T_\theta - E_{I,\delta} \right|
	 =  \left| \frac{\delta}{2}\rho\theta (T_\delta - T_{I,\delta}) \right|&= \left| \frac{\rho\theta}{2} \frac{3\delta}{3+\delta} (T_{tr} - T_{I,\delta}) \right|
	\leq \frac{\rho \theta}{2}(3+\delta)T_\delta.
	\end{split}
	\end{align}
	Combining \eqref{TIdelta split} and \eqref{TIdelta split r.h.s}, we also derive $|\partial_t T_{I,\delta}|<C$. From $\displaystyle T_\delta = \frac{3}{3+\delta} T_{tr} + \frac{\delta}{3+\delta}T_{I,\delta}$, we further have $|\partial_t T_{tr}|<C$. It remains to estimate $|\partial_t\Theta|$. We recall the definition of stress tensor $\Theta(x,t)$:
	\begin{align*}
		\rho(x,t) \Theta(x,t) &= \int_{\mathbb{R}^3\times\mathbb{R}_+}\left(v-U(x,t)\right)\otimes\left(v-U(x,t)\right) f(x,v,t)dvdI.
	\end{align*}
	For simplicity, we only consider two cases $|\partial_t\Theta_{11}|$ and $|\partial_t\Theta_{12}|$:
	\begin{align*}
	|\partial_t\Theta_{11}| &=\left|\frac{\partial_t \rho}{\rho^2} \int_{\mathbb{R}^3\times\mathbb{R}_+} \left|v^1-U^1\right|^2 fdvdI\right| + \left|\frac{1}{\rho} \int_{\mathbb{R}^3\times\mathbb{R}_+}2\left|v^1-U^1\right||\partial_t U^1| fdvdI\right| \cr
	&\quad + \left|\frac{1}{\rho} \int_{\mathbb{R}^3\times\mathbb{R}_+}\left|v^1-U^1\right|^2 |\partial_t f| dvdI\right|\cr
	&\leq \left|\frac{\partial_t \rho}{\rho^2} + \frac{2}{\rho}\right|  (1+ |\partial_t U|)\left|\int_{\mathbb{R}^3\times\mathbb{R}_+} \left|v-U\right|^2 (|f|+|\partial_t f |)dvdI\right|
	\end{align*}
	and
	\begin{align*}
	|\partial_t\Theta_{12}| &=\left|\frac{\partial_t \rho}{\rho^2} \int_{\mathbb{R}^3\times\mathbb{R}_+} \left|v^1-U^1\right|\left|v^2-U^2\right| fdvdI\right|\cr
	&\quad + \left|\frac{1}{\rho} \int_{\mathbb{R}^3\times\mathbb{R}_+}\bigg(\left|v^1-U^1\right||\partial_t U^2| + \left|v^2-U^2\right||\partial_t U^1| \bigg)fdvdI\right| \cr
	&\quad + \left|\frac{1}{\rho} \int_{\mathbb{R}^3\times\mathbb{R}_+}\left|v^1-U^1\right|\left|v^2-U^2\right| \partial_t fdvdI\right|\cr
	&\leq \left|\frac{\partial_t \rho}{\rho^2} + \frac{2}{\rho}\right|  (1+ |\partial_t U|)\left|\int_{\mathbb{R}^3\times\mathbb{R}_+} \big(\left|v-U\right|^2 + \left|v-U\right|\big) (|f|+|\partial_t f |) dvdI\right|.
	\end{align*}
	In both cases, the last upper bounds can be bounded using \eqref{rho t estimate}, the lower bound of $\rho$ and the upper bounds of $\rho$, $U$, $|\partial_t\rho|$, $|\partial_t U|$, $|\partial_t E_\delta|$. Therefore, we have $|\partial_t\Theta_{11}|,\, |\partial_t\Theta_{12}| < C$ for a constant $C>0$. 
	  
	Until now, we show that the following time derivatives of macroscopic quantities are bounded:
	\[ |\partial_t \rho|,\ |\partial_t U|,\ |\partial_t T_\delta|,\ |\partial_t T_{I,\delta}|,\ |\partial_t T_{tr}|,\ |\partial_t T_\theta|,\ |\partial_t \Theta_{ij}| \leq C.
	\]
	From the definition of $T_{\nu,\theta}$, we further obtain that $|\partial_t (T_{\nu,\theta})_{ij}| \leq C$ for $1 \leq i,j \leq 3$.
	
	Now, we move on to the estimate of $|\partial_t \mathcal{M}_{\nu,\theta}|$. For this, we write
	\begin{align}\label{Mt compute}
		\begin{split}
		\partial_t \mathcal{M}_{\nu,\theta} &= \partial_t\left\{\frac{\rho\Lambda_\delta}{\sqrt{\det\left(2 \pi \mathcal{T}_{\nu,\theta} \right)}(T_\theta)^\frac{\delta}{2}}\exp \left({-\frac{(v-U(x,t))^{\top}\mathcal{T}_{\nu,\theta}^{-1}(v-U(x,t))}{2} -\frac{I^{\frac{2}{\delta}}}{T_\theta}}\right)\right\}\cr
		&=\bigg(\frac{\partial_t \rho}{\rho} -\frac{1}{2}\big\{\det\left(2 \pi \mathcal{T}_{\nu,\theta} \right)\big\}^{-1} \partial_t \big\{\det\left(2 \pi \mathcal{T}_{\nu,\theta} \right)\big\} -\frac{\delta}{2}\frac{1}{T_\theta} \partial_t T_\theta\bigg)\mathcal{M}_{\nu,\theta}\cr
		&\quad +\bigg( {-\frac{(\partial_t U)^{\top}\mathcal{T}_{\nu,\theta}^{-1}(v-U)}{2}-\frac{( v-U)^{\top}\mathcal{T}_{\nu,\theta}^{-1}\partial_t U}{2} +\frac{I^{\frac{2}{\delta}}}{(T_\theta)^2}}\partial_t T_\theta \bigg)\mathcal{M}_{\nu,\theta}\cr
		&\quad +\bigg( {-\frac{(v-U)^{\top}\mathcal{T}_{\nu,\theta}^{-1}\partial_t\{\mathcal{T}_{\nu,\theta}^{-1}\}\mathcal{T}_{\nu,\theta}^{-1}(v-U)}{2} } \bigg)\mathcal{M}_{\nu,\theta}.
		\end{split}
	\end{align}
	Note that each macroscopic quantity and its time derivative are bounded, and the positivity of $T_\theta$ is also guaranteed by $T_\delta>C$. Finally, we combine Proposition \ref{P.6.2}, Proposition \ref{P.6.3}, Theorem \ref{T.3.1} and \eqref{Mt compute} to derive
	\begin{align*}
	|\partial_t \mathcal{M}_{\nu,\theta}| \leq C \big(1+|v|+|v|^2\big)\mathcal{M}_{\nu,\theta} \leq C \big(1+|v|^2 + I^{\frac{2}{\delta}}\big)^{\frac{q}{2}}\mathcal{M}_{\nu,\theta} \leq C\{ \|f_0\|_{L_{1,q}^\infty} + 1 \}
	\end{align*}	
	for $q>5+\delta$.
    The estimate for spatial derivative $|\nabla_x \mathcal{M}_{\nu,\theta}|$ can be done similarly. 
    
\end{proof}
\begin{lemma}\label{L.6.2}
	Under the assumption of Theorem \ref{T.3.1}, the estimations for $R_1$ and $R_2$ satisfy
	\[
	\|R_1\|_{L_q^\infty} + \|R_2\|_{L_q^\infty}\leq C(\Delta t)^2
	\]
	for a constant $C>0$ depending on $T^f,q,\delta,\kappa,\theta,\nu, C_{2,1}, C_{2,2}$.
\end{lemma}
\begin{proof}
	We first split $R_1$ in Lemma \ref{L.6.1} into two parts: 
	\begin{align*}
	R_1&=- \frac{A_{\nu,\theta}}{\kappa+A_{\nu,\theta}\Delta t}\int_{t}^{t+\Delta t}  \{\mathcal{M}_{\nu,\theta}(\tilde{f})(x,v,t,I) - \mathcal{M}_{\nu,\theta}(f)(x,v,t,I)\} ds\cr
	&\quad - \frac{A_{\nu,\theta}}{\kappa+A_{\nu,\theta}\Delta t}\bigg(\int_{t}^{t+\Delta t}  (t+\Delta t-s)v^1\partial_x \mathcal{M}_{\nu,\theta}(f)(x_{\theta_1},v,t_{\theta_1},I)\cr
	 &\qquad \qquad \qquad \qquad \qquad \qquad \, \, \quad - (s-t) \partial_t \mathcal{M}_{\nu,\theta}(f)(x_{\theta_1},v,t_{\theta_1},I) ds\bigg)\cr
	&=\mathcal{I}_{51} + \mathcal{I}_{52}.
	\end{align*}
	For $\mathcal{I}_{51}$, we use Proposition \ref{P.6.1} to get
	\begin{align*}
		\|\mathcal{M}_{\nu,\theta}(\tilde{f}) - \mathcal{M}_{\nu,\theta}(f)\|_{L_q^\infty}\leq C_{Lip}\|\tilde{f} - f\|_{L_q^\infty}.
	\end{align*}
	Next, we use the mean value theorem to obtain
	\begin{align*}
	\|\tilde{f} - f\|_{L_q^\infty} = \|\Delta t v^1 \partial_x f\|_{L_q^\infty} \leq \| f\|_{L_{1,q+1}^\infty} \Delta t \leq C_{2,1}e^{C_{2,2} T^f} (\| f_0\|_{L_{1,q+1}^\infty} + 1) \Delta t.
	\end{align*}
	In the last line, we use Theorem \ref{T.3.1}.
	Then,
	\begin{align*}
		\begin{split}
			|\mathcal{I}_{51}| &\leq \frac{A_{\nu,\theta}}{\kappa+A_{\nu,\theta}\Delta t}C_{2,1}e^{C_{2,2} T^f} (\| f_0\|_{L_{1,q+1}^\infty} + 1) (\Delta t)^2\cr
			&\leq \frac{A_{\nu,\theta}}{\kappa}C_{2,1}e^{C_{2,2} T^f} (\| f_0\|_{L_{1,q+1}^\infty} + 1) (\Delta t)^2.
		\end{split}
	\end{align*}
	To estimate $\mathcal{I}_{52}$, we use Proposition \ref{Mt estimate}: 
	\begin{align*}
		\|v^1\partial_x \mathcal{M}_{\nu,\theta}(f)\|_{L^\infty}, \quad \|\partial_t \mathcal{M}_{\nu,\theta}(f)\|_{L^\infty} \leq C(\| f_0\|_{L_{1,q+1}^\infty} + 1),
	\end{align*}
	then
	\begin{align*}
		|\mathcal{I}_{52}| &\leq \frac{2 A_{\nu,\theta}}{\kappa+A_{\nu,\theta}\Delta t}C\left(\| f_0\|_{L_{1,q+1}^\infty} + 1\right) (\Delta t)^2\cr
		&\leq \frac{2 A_{\nu,\theta}}{\kappa}C\left(\| f_0\|_{L_{1,q+1}^\infty} + 1\right) (\Delta t)^2.
	\end{align*}
	Therefore, $R_1$ is estimated by
	\begin{align*}
		|R_1| \leq C (\Delta t)^2.
	\end{align*}
	For $R_2$, we use Proposition \ref{P.6.3} Theorem \ref{T.3.1} to obtain
	\begin{align*}
		\| (\mathcal{M}_{\nu,\theta}(f) - f) \|_{L_q^\infty}\leq \| \mathcal{M}_{\nu,\theta}(f) \|_{L_q^\infty} + \| f \|_{L_q^\infty} \leq C \{ \|f_0\|_{L_q^\infty} + 1\},
	\end{align*}
	from which we have
	\begin{align*}
	|R_2|&= \bigg|\int_{t}^{t+\Delta t} (s-t-\Delta t)A_{\nu,\theta} (\mathcal{M}_{\nu,\theta}(f) - f)(x_{\theta_2},v,t_{\theta_2},I) ds\bigg|\cr
	&\leq A_{\nu,\theta}\| (\mathcal{M}_{\nu,\theta}(f) - f) \|_{L_q^\infty} \int_{t}^{t+\Delta t} (s-t-\Delta t)A_{\nu,\theta} ds\cr
	&\leq  C \{ \|f_0\|_{L_q^\infty} + 1\} (\Delta t)^2.
	\end{align*}
	This completes the proof.
\end{proof}

\section{Estimate of $\mathcal{M}_{\nu,\theta}(\tilde{f}(t^n)) - \mathcal{M}_{\nu,\theta}(\tilde{f}^n)$} 
The goal of this section is to establish the discrepancy estimate of the continuous ellipsoidal Gaussian $\mathcal{M}_{\nu,\theta}(\tilde{f}(t^n))$ in \eqref{conti eg}
and the discrete one $\mathcal{M}_{\nu,\theta}(\tilde{f}^n)$ in \eqref{approx M}.
\begin{lemma}\label{L.7.1}
	Let $\tilde{f}(t^n)$ and $\tilde{f}^n$ denote the continuous and the discrete solutions at $t^n$. Then, 
	\begin{align*}
	\|\tilde{f}(t^n) - \tilde{f}^n\|_{L_q^\infty} \leq \|f(t^n) - f^n\|_{L_q^\infty} + \frac{C_{2,1}}{2}e^{C_{2,2} T^f}  \{ \| f_0\|_{L_{2,q}^\infty} +1 \} (\Delta x)^2,
	\end{align*}
	where $C_{2,1},C_{2,2}$ are defined in Theorem \ref{T.3.1}.
\end{lemma}
\begin{proof}
	Recalling \eqref{B-2}, we compute $\tilde{f}_{i,j,k}^n$ as
	\[
	\tilde{f}_{i,j,k}^n:= a_j f_{s,j,k}^n + (1-a_j) f_{s+1,j,k}^n,\quad a_j=(x_{s+1}-x(i,j))/\Delta x.
	\]
	Also, we use Taylor's theorem to obtain
	\begin{align}\label{G-1}
\begin{split}
		\tilde{f}(x_i,v_j,t^n,I_k)&= f(x_i-v_j^1\Delta t, v_j,t^n,I_k)\cr
	&=a_j \left(f(x_s,v_j,t^n,I_k) + \frac{(x_i-\Delta t v_j^1  -x_s)^2}{2}\partial_{xx} f(x_{\xi_1},v_j,t^n,I_k)\right) \cr
	&+ (1-a_j)  \left(f(x_{s+1},v_j,t^n,I_k) +\frac{(x_i-\Delta t v_j^1  -x_{s+1})^2}{2}\partial_{xx} f(x_{\xi_2},v_j,t^n,I_k)\right),
\end{split}
	\end{align}
	where $x_{\xi_1}$ lies between $x_s$ and $x_i-v_j^1\Delta t$, and $x_{\xi_2}$ lies between $x_{s+1}$ and $x_i-v_j^1\Delta t$. Now, we estimate the discrepancy of $\tilde{f}(x_i,v_j,t^n,I_k)$ and $\tilde{f}_{i,j,k}^n$ as
	\begin{align}\label{subtraction}
	\begin{split}
	|\tilde{f}(t^n) - \tilde{f}^n|&\leq
	a_j |f(x_s,v_j,t^n,I_k) -  f_{s,j,k}^n|  + (1-a_j)|f(x_{s+1},v_j,t^n,I_k) -  f_{s+1,j,k}^n| \cr
	&\quad +a_j \frac{(\Delta x)^2}{2}|\partial_{xx} f(x_{\xi_1},v,t^n,I)| 
	+ (1-a_j)\frac{(\Delta x)^2}{2}|\partial_{xx} f(x_{\xi_2},v_j,t^n,I_k)|.
	\end{split}
	\end{align}
	We also note that Theorem \ref{T.3.1} imposes
	\begin{align*}
	\|\partial_{xx} f(t^n)\|_{L_q^\infty} \leq  \|f(t^n)\|_{L_{2,q}^\infty} 
	\leq C_{2,1}e^{C_{2,2} T^f} \left( \|f_0\|_{L_{2,q}^\infty} +  1\right).
	\end{align*}
	This, combined with \eqref{subtraction}, gives
	\begin{align*}
	\|\tilde{f}(t^n) - \tilde{f}^n\|_{L_q^\infty}\leq&
	a_j \|f(t^n) - f^n\|_{L_q^\infty}  + (1-a_j)\|f(t^n) - f^n\|_{L_q^\infty} +\frac{(\Delta x)^2}{2} \|\partial_{xx} f\|_{L_q^\infty} \cr
	\leq&
	\|f(t^n) - f^n\|_{L_q^\infty} + 2\| f(t^n)\|_{L_{2,q}^\infty}  (\Delta x)^2\cr
	\leq&
	\|f(t^n) - f^n\|_{L_q^\infty} + \frac{C_{2,1}}{2}e^{C_{2,2} T^f}  \{ \| f_0\|_{L_{2,q}^\infty} + 1\} (\Delta x)^2,
	\end{align*}
	which completes the proof.
\end{proof}

\begin{lemma}\label{L.7.2}
	Suppose that $q>5+\delta$ and $\Delta v$, $\Delta I$ satisfies the condition \eqref{cond T.5.5}. Let $\Phi(v,I)$ denote one of $1,\,v,\,|v|^2,\,I^{\frac{2}{\delta}},\, v^m v^n $ $(1\leq m,n \leq 3)$ and $\Phi_{jk}:=\Phi(v_j,I_k)$, then we have
	\begin{align*}
	&\bigg|\sum_{j,k}\tilde{f}_{i,j,k}^n\Phi_{jk}(\Delta v)^3 \Delta I -\int_{\mathbb{R}^3\times\mathbb{R}_+} \tilde{f}(x_i,v,t^n,I)\Phi(v,I)dvdI\bigg| \cr
	&\qquad \leq \bar{C}_1\|\tilde{f}(t^n) - \tilde{f}^n\|_{L_q^\infty} + \bar{C}_2\bigg(\|f_0 \|_{L_{2,q}^\infty} + 1\bigg)( (\Delta x)^2 + \Delta v\Delta t + \Delta v + \Delta I)
	\end{align*}
	for some positive constants $\bar{C}_1$ and $\bar{C}_2$ which depend on $\delta, q, C_{2,1}, C_{2,2}, T^f$.
\end{lemma}
\begin{proof}
Let $\Delta_{j,k}$ denote an domain such that
$$(v_j,I_k) \in \Delta_{j,k}=[v_j^1,v_j^1+ \Delta v) \times [v_j^2,v_j^2+ \Delta v) \times [v_j^3,v_j^3+ \Delta v) \times [I_k,I_{k+1}). $$
With this, we have
\begin{align*}
\sum_{j,k}\tilde{f}_{i,j,k}^n\Phi_{jk}&(\Delta v)^3 \Delta I -\int_{\mathbb{R}^3\times\mathbb{R}_+} \tilde{f}(x_i,v,t^n,I)\Phi(v,I)dvdI\cr &= 
\sum_{j,k}\tilde{f}_{i,j,k}^n\Phi_{jk}(\Delta v)^3 \Delta I - \sum_{j,k}\int_{\Delta_{j,k}} \tilde{f}(x_i,v,t^n,I)\Phi(v,I)dvdI\cr
&= 
\bigg(\sum_{j,k}\tilde{f}_{i,j,k}^n \Phi_{jk}(\Delta v)^3 \Delta I - \sum_{j,k}\int_{\Delta_{j,k}} \tilde{f}(x_i,v,t^n,I)\Phi_{jk}dvdI\bigg)\cr
&\quad +\bigg(\sum_{j,k}\int_{\Delta_{j,k}} \tilde{f}(x_i,v,t^n,I)\Phi_{jk}dvdI - \sum_{j,k}\int_{\Delta_{j,k}} \tilde{f}(x_i,v,t^n,I)\Phi(v,I)dvdI\bigg)\cr
&=\mathcal{I}_{61}+\mathcal{I}_{62}.
\end{align*}
From (\ref{G-1}) and Taylor's theorem , we have
\begin{align*}
\tilde{f}(x_i,v,t^n,I)&=f(x_i-v^1 \Delta t ,v,t^n,I)\cr
&=f(x_i-v_j^1 \Delta t ,v_j,t^n,I_k) + (v_j^1-v^1)\Delta t \partial_x f(z_{\theta_1}) \cr
&\quad  + (v-v_j)\cdot \nabla_v f(z_{\theta_2}) +(I-I_k) \partial_I f(z_{\theta_3}) \cr
&=a_j f(x_s,v_j,t^n,I_k)+ (1-a_j) f(x_{s+1},v_j,t^n,I_k) +R,
\end{align*}
where $R$ is given by
\begin{align*}
R &= (v_j^1-v^1)\Delta t \partial_x f(z_{\theta_1}) + (v-v_j)\cdot \nabla_v f(z_{\theta_2}) +(I-I_k) \partial_I f(z_{\theta_3})\cr
&\quad +a_j \frac{(x_i-\Delta t v_j^1  -x_s)^2}{2}\partial_{xx} f(x_{\xi_1},v_j,t^n,I_k) \cr
&\quad + (1-a_j)\frac{(x_i-\Delta t v_j^1  -x_{s+1})^2}{2}\partial_{xx} f(x_{\xi_2},v_j,t^n,I_k), 
\end{align*}
where $x_{\xi_1}, x_{\xi_2}\in [x_s,x_{s+1})$ 
and  $z_{\theta_\ell}:=(x_s+\theta_{\ell x}\Delta x,v_j+\theta_{\ell v}\Delta v,t^n,I_k +\theta_{\ell I}\Delta I)$ for some $\theta_{\ell x},\theta_{\ell I} \in [0,1)$, $\theta_{\ell v} \in [0,1)^3$, $(\ell=1,2,3) $.
To estimate $\mathcal{I}_{61}$, we first separate it into two parts:
\begin{align*}
\mathcal{I}_{61}&= 
\sum_{j,k}\tilde{f}_{i,j,k}^n \Phi_{jk}(\Delta v)^3 \Delta I - \sum_{j,k}\int_{\Delta_{j,k}} \tilde{f}(x_i,v,t^n,I)\Phi_{jk}dvdI\cr
&= 
\sum_{j,k}\int_{\Delta_{j,k}}a_j\left( f_{s,j,k}^n-f(x_s,v_j,t^n,I_k) \right) + (1-a_j)\left(   f_{s+1,j,k}^n -  f(x_{s+1},v_j,t^n,I_k)\right)\Phi_{jk} dvdI  \cr
&\quad - \sum_{j,k}\int_{\Delta_{j,k}} R\Phi_{jk}dvdI\cr
&=\mathcal{I}_{611}+\mathcal{I}_{612}.
\end{align*}
We bound $\mathcal{I}_{611}$ as follows:
\begin{align}\label{I6111}
\begin{split}
	&|\mathcal{I}_{611}| \cr 
&\leq \sum_{j,k} \bigg(a_j\left|  f_{s,j,k}^n-f(x_s,v_j,t^n,I_k) \right| 
+ (1-a_j) \left|f_{s+1,j,k}^n -  f(x_{s+1},v_j,t^n,I_k) \right|\bigg) |\Phi_{jk}|  (\Delta v)^3 \Delta I  \cr
& \leq \|\tilde{f}(t^n) - \tilde{f}^n\|_{L_q^\infty} \sum_{j,k}  \frac{|\Phi_{jk}|}{(1+|v_j|^2+I_k^\frac{2}{\delta})^{\frac{q}{2}}}  (\Delta v)^3 \Delta I  \cr
& \leq 2\bar{C}_{\delta,q-2}\|\tilde{f}(t^n) - \tilde{f}^n\|_{L_q^\infty}.
\end{split}
\end{align}
In the last line, the inequality comes from Theorem \ref{T.5.5}.
For $\mathcal{I}_{612}$, we bound $R$ using Theorem \ref{T.3.1} and the following inequality:
\begin{align*}
|\partial_{x} f(z_{\theta_\ell})| ,|\partial_{xx} f(z_{\theta_\ell})| ,| \nabla_v f(z_{\theta_\ell})|, |\partial_I f(z_{\theta_\ell})|
& \leq  \frac{\|f \|_{L_{2,q}^\infty} }{(1+|v_j+\theta_{\ell v} \Delta v|^2+(I_k + \theta_{\ell I}\Delta I)^\frac{2}{\delta})^{\frac{q}{2}}}\cr
& \leq  \frac{\|f \|_{L_{2,q}^\infty} }{\left( 1+|v_j|^2+I_k ^\frac{2}{\delta}\right)^{\frac{q}{2}}}.
\end{align*} 
That is,
$$|R|\leq  C_{2,1}e^{C_{2,2} T^f}( (\Delta x)^2 + \Delta v\Delta t + \Delta v + \Delta I)\left(\frac{\|f_0 \|_{L_{2,q}^\infty} + 1}{(1+|v_j|^2+I_k^\frac{2}{\delta})^{\frac{q}{2}}} \right). $$
Using this, we have
\begin{align*}
|\mathcal{I}_{612}| &\leq  \sum_{j,k}\int_{\Delta_{j,k}} |R||\Phi_{jk}|dvdI\cr
&\leq C_{2,1}e^{C_{2,2} T^f}( (\Delta x)^2 + \Delta v\Delta t + \Delta v + \Delta I)\bigg(\|f_0 \|_{L_{2,q}^\infty} + 1\bigg) \sum_{j,k} \int_{\Delta_{j,k}}  \frac{|\Phi_{jk}|}{(1+|v_j|^2+I_k^\frac{2}{\delta})^{\frac{q}{2}}}dvdI\cr
&\leq 2\bar{C}_{\delta,q-2}C_{2,1}e^{C_{2,2} T^f}( (\Delta x)^2 + \Delta v\Delta t + \Delta v + \Delta I)\bigg(\|f_0 \|_{L_{2,q}^\infty} + 1\bigg),
\end{align*}
where the last inequality holds as in \eqref{I6111}.
For $\mathcal{I}_{62}$, we consider $(v,I)\equiv (v_j+\xi \Delta v, I_k + \eta \Delta I) \in \Delta_{j,k}$, for $\xi,\eta \in [0,1)$. Then, we have from $\Delta v \leq \frac{1}{2}$ that
\begin{align*}
	|v_j-v| \leq \sqrt{3}\Delta v&, \quad ||v_j|^2-|v|^2| \leq \sqrt{3}\Delta v (|v_j| + |v|) \leq \sqrt{3} \Delta v (\sqrt{3}\Delta v + 2|v|) \leq 6 \Delta v (1 + |v|^2)
\end{align*}
and, for $1\leq m, n \leq 3$,
\begin{align*}
	|v_j^m v_j^n - v^m v^n|&\leq |v_j^m v_j^n - v_j^m v^n + v_j^m v^n- v^m v^n|\cr
	&\leq |v_j^m v_j^n - v_j^m v^n| + |v_j^m v^n- v^m v^n|\cr
	&\leq \Delta v|v_j^m| + \Delta v |v^n|\cr
	&\leq 3 \Delta v (1 + |v|^2).
\end{align*}
Moreover, for $I\in [I_k,I_{k+1}) $, the mean-value theorem implies 
\[
|I_k^{\frac{2}{\delta}}-I^{\frac{2}{\delta}}| \leq |I_k-I| \frac{2}{\delta}(I+\Delta I)^{\frac{2}{\delta}-1} \leq \frac{2 \Delta I}{\delta} (I+\Delta I)^{\frac{2}{\delta}-1}, \quad 0 <\delta \leq 2.
\]
This, together with the assumption $\Delta I <\frac{1}{2}$ in \eqref{cond T.5.5}, gives
\begin{align*}
|I_k^{\frac{2}{\delta}}-I^{\frac{2}{\delta}}| 
\leq \frac{2\Delta I}{\delta} \big(I + 1 \big)^{\frac{2}{\delta}-1} \leq \frac{2\Delta I}{\delta} \big(2^{\frac{2}{\delta}-1}+ (2I)^{\frac{2}{\delta}-1}\big)\leq 2^{\frac{2}{\delta}}\frac{\Delta I}{\delta} \big(1+ I^{\frac{2}{\delta}}\big). 
\end{align*}
To sum up,
\begin{align*}
|\Phi_{jk} - \Phi(v,I)| &\leq 
6 \Delta v (1 + |v|^2) + 2^{\frac{2}{\delta}}\frac{\Delta I}{\delta} \big(1+ I^{\frac{2}{\delta}}\big).
\end{align*}
Now, $\mathcal{I}_{62}$ is estimated by
\begin{align*}
|\mathcal{I}_{62}|
&\leq\sum_{j,k}\int_{\Delta_{j,k}} \tilde{f}(x_i,v,t^n,I)|\Phi_{jk}-\Phi(v,I)|dvdI \cr
&\leq \|f(t^{n})\|_{L_q^\infty} \sum_{j,k}\bigg\{\int_{\Delta_{j,k}} \frac{6 \Delta v (1 + |v|^2) }{(1+|v|^2 + I^{\frac{2}{\delta}})^\frac{q}{2}}dvdI + \int_{\Delta_{j,k}} \frac{2^{\frac{2}{\delta}}\frac{\Delta I}{\delta} \big(1+ I^{\frac{2}{\delta}}\big)}{(1+|v|^2 + I^{\frac{2}{\delta}})^\frac{q}{2}}dvdI \bigg\} \cr
&\leq \|f(t^{n})\|_{L_q^\infty} \left(6\Delta v  + 2^{\frac{2}{\delta}}\frac{\Delta I}{\delta}\right)\sum_{j,k} \int_{\Delta_{j,k}} \frac{ 1}{(1+|v|^2 + I^{\frac{2}{\delta}})^\frac{q-2}{2}}dvdI \cr
&\leq \left(6+ 2^{\frac{2}{\delta}}\frac{1}{\delta}\right)  (\Delta v  + \Delta I ) \bar{C}_{\delta,q-2} C_{2,1}e^{C_{2,2} T^f}\{ \|f_0\|_{L_{2,q}^\infty} + 1 \} .
\end{align*}
where $\bar{C}_{\delta,q-2}$ is given in  Definition \ref{D.5.1}. Combining $\mathcal{I}_{61}$ and $\mathcal{I}_{62}$, we obtain the desired result. 
\end{proof}

\begin{lemma}\label{L.7.3} Suppose that $q>5+\delta$ and $\Delta v$, $\Delta I$ satisfies the condition \eqref{cond T.5.5}. Then, 
	\begin{align*}
		&|\tilde{\rho}_i -\tilde{\rho}(x_i,t^n)|,~|\tilde{U}_i -\tilde{U}(x_i,t^n)|,~|(\tilde{\mathcal{T}}_{\nu,\theta}^{\alpha,\beta})_i^{n} -\tilde{\mathcal{T}}_{\nu,\theta}^{\alpha,\beta}(x_i,t^n)| \cr
		&\qquad \leq C \|f(t^n) - f^n\|_{L_q^\infty} + C \{\|f_0\|_{L_{2,q}^\infty} +1 \} \{(\Delta x)^2 + \Delta v + \Delta I + \Delta v \Delta t\}.
	\end{align*}
	where $C>0$ is a constant and the $(\alpha,\beta)$ element of $\tilde{\mathcal{T}}_{\nu,\theta}$ is denoted by $\tilde{\mathcal{T}}_{\nu,\theta}^{\alpha,\beta}$ for $1 \leq \alpha, \beta \leq 3$.
\end{lemma}

\begin{proof}
	Consider the case $\Phi_{jk} \equiv 1$ in Lemma \ref{L.7.2}, then 
	\begin{align*}
		|\tilde{\rho}_i^n -\tilde{\rho}(x_i,t^n)|&= \bigg|\sum_{j,k}\tilde{f}_{i,j,k}^n (\Delta v)^3 \Delta I - \sum_{j,k}\int_{\Delta_{j,k}} \tilde{f}(x_i,v,t^n,I)dvdI\bigg|\cr
		&\leq \bar{C}_1 \|f(t^n) - f^n\|_{L_q^\infty} + \bar{C}_2 \{\| f_0\|_{L_{2,q}^{\infty}} +1 \} \{(\Delta x)^2 + \Delta v + \Delta I + \Delta v \Delta t\}.
	\end{align*}
The number $\bar{C}_1$ and $\bar{C}_2$ are constants in Lemma \ref{L.7.2}. For the second estimate, we begin with
	\begin{align*}
		|\tilde{U}_i^n -\tilde{U}(x_i,t^n)|&=\bigg|\frac{\tilde{\rho}_i^n\tilde{U}_i^n - \tilde{\rho}(x_i,t^n)\tilde{U}(x_i,t^n)}{\tilde{\rho}_i^n} +\frac{\tilde{\rho}(x_i,t^n)\tilde{U}(x_i,t^n)-\tilde{\rho}_i^n\tilde{U}(x_i,t^n)}{\tilde{\rho}_i^n}\bigg|.
	\end{align*}
	From $C^n$ in Definition \ref{D.5.2}, we have
	\begin{align*}
	\frac{1}{\tilde{\rho}_i^n} \leq \frac{2}{\bar{C}_{a,b} C_0^1}e^{\frac{A_{\nu,\theta}}{\kappa} T^f},
	\end{align*}
	which together with Lemma \ref{L.7.2} gives
	\begin{align*}
		|\tilde{\rho}_i^n\tilde{U}_i^n - \tilde{\rho}\tilde{U}(x_i,t^n)| \leq \bar{C}_1\|\tilde{f}(t^n) - \tilde{f}^n\|_{L_q^\infty} + \bar{C}_2\bigg(\|f_0 \|_{L_{2,q}^\infty} + 1\bigg)( (\Delta x)^2 + \Delta v\Delta t + \Delta v + \Delta I).
	\end{align*}
	Moreover, we have
	\begin{align*}
		|\tilde{\rho}\tilde{U}(x_i,t^n)| &= \left|\int_{\mathbb{R}^3\times\mathbb{R}_+}v\tilde{f}(x_i,v,t^n,I_k)dvdI\right|\\
		&= \int_{\mathbb{R}^3\times\mathbb{R}_+}|v|f(x_i-v^1\Delta t,v,t^n,I)\frac{(1+|v|^2 + I^\frac{2}{\delta})^\frac{q}{2}}{(1+|v|^2 + I^\frac{2}{\delta})^\frac{q}{2}}dvdI\\
		&\leq \|f(t^n)\|_{L_q^\infty}\int_{\mathbb{R}^3\times\mathbb{R}_+}\frac{1}{(1+|v|^2 + I^\frac{2}{\delta})^\frac{q-1}{2}}dvdI\\
		&= \bar{C}_{\delta,q-1}C_{2,1}e^{C_{2,2} T^f}\{ \|f_0\|_{L_{2,q}^\infty} + 1 \}.
	\end{align*}
	Therefore,
	\begin{align*}
	|\tilde{U}_i^n -\tilde{U}(x_i,t^n)|&=\bigg|\frac{\tilde{\rho}_i^n\tilde{U}_i^n - \tilde{\rho}(x_i,t^n)\tilde{U}(x_i,t^n)}{\tilde{\rho}_i^n} +\frac{\tilde{\rho}(x_i,t^n)\tilde{U}(x_i,t^n)-\tilde{\rho}_i^n\tilde{U}(x_i,t^n)}{\tilde{\rho}_i^n}\bigg|\cr
	&= \frac{1}{\tilde{\rho}_i^n}\bigg|\sum_{j,k}\tilde{f}_{i,j,k}^n v_j (\Delta v)^3 \Delta I - \sum_{j,k}\int_{\Delta_{j,k}} \tilde{f}(x_i,v,t^n,I)vdvdI\bigg|\cr
	&\quad + \frac{\tilde{U}(x_i,t^n)}{\tilde{\rho}_i^n}\bigg|\sum_{j,k}\tilde{f}_{i,j,k}^n  (\Delta v)^3 \Delta I - \sum_{j,k}\int_{\Delta_{j,k}} \tilde{f}(x_i,v,t^n,I)dvdI\bigg|\cr
	&\leq C \|f(t^n) - f^n\|_{L_q^\infty} + C\{\| f_0\|_{L_{2,q}^\infty}+1\} \{(\Delta x)^2 + \Delta v + \Delta I + \Delta v \Delta t\},
	\end{align*}
	for a constant $C>0$. 
	
	For the estimate of $\tilde{\mathcal{T}}_{\nu,\theta}$, we recall its definition in to get	
	\begin{align*}
	\tilde{\rho}_i^n(&\tilde{\mathcal{T}}_{\nu,\theta})_i^n - \tilde{\rho}\tilde{\mathcal{T}}_{\nu,\theta}\\
	&=(1-\theta)\tilde{\rho}_i^n\{(1-\nu)(\tilde{T}_{tr})_i^n Id+\nu\tilde{\Theta}_i^n\}+\theta \tilde{\rho}_i^n (\tilde{T}_{\delta})_i^nId\\
	&\quad + \tilde{\rho}_i^n\bigg[\frac{(1-\theta)\nu \Delta tA_{\nu,\theta}\theta}{\Delta t + \kappa} (\tilde{T}_\delta)_i^n Id + \frac{(1-\theta)\nu\Delta t(1- A_{\nu,\theta} \theta)}{\Delta t + \kappa}(\tilde{T}_{tr})_i^nId - (1-\theta)\nu \frac{ \Delta t }{\Delta t + \kappa}\tilde{\Theta}_i^n \bigg]\\
	&\quad -(1-\theta)\tilde{\rho}\{(1-\nu)\tilde{T}_{tr}Id+\nu\tilde{\Theta}\}-\theta\tilde{\rho}\tilde{T}_{\delta}Id\\
	&=(1-\theta)
	\sum_{j,k} \tilde{f}_{i,j,k}^n\Big\{\frac{(1-\nu)}{3}|v_j-\tilde{U}_i^n|^2Id+\nu(v_j-\tilde{U}_i^n)\otimes(v-\tilde{U}_i^n)\Big\} (\Delta v)^3 \Delta I\\
	&\quad +\frac{\theta}{3+\delta}
	\sum_{j,k}\tilde{f}_{i,j,k}^n\Big\{|v_j-\tilde{U}_i^n|^2+2I^{\frac{2}{\delta}}\Big\}Id (\Delta v)^3 \Delta I\\
	&\quad -(1-\theta)
	\int_{\mathbb{R}^3\times\mathbb{R}^+}\tilde{f}\Big\{\frac{(1-\nu)}{3}|v-\tilde{U}|^2Id+\nu(v-\tilde{U})\otimes(v-\tilde{U})\Big\}dvdI \\
	&\quad -\frac{\theta}{3+\delta}
	\int_{\mathbb{R}^3\times\mathbb{R}^+}\tilde{f}\big\{|v-\tilde{U}|^2+2I^{\frac{2}{\delta}}\big\}IddvdI	\\
	&\quad + \tilde{\rho}_i^n\bigg[\frac{(1-\theta)\nu \Delta tA_{\nu,\theta}\theta}{\Delta t + \kappa} (\tilde{T}_\delta)_i^n Id + \frac{(1-\theta)\nu\Delta t(1- A_{\nu,\theta} \theta)}{\Delta t + \kappa}(\tilde{T}_{tr})_i^nId - (1-\theta)\nu \frac{ \Delta t }{\Delta t + \kappa}\tilde{\Theta}_i^n \bigg],
	\end{align*}
	which can be rewritten as
	\begin{align*}
	&\tilde{\rho}_i^n(\tilde{\mathcal{T}}_{\nu,\theta})_i^n - \tilde{\rho}\tilde{\mathcal{T}}_{\nu,\theta}\\
	&=\Big\{(1-\theta)\frac{1-\nu}{3}+\frac{\theta}{3+\delta}\Big\}\bigg( \sum_{j,k}\tilde{f}_{i,j,k}^n|v_j-\tilde{U}_i^n|^2 (\Delta v)^3 \Delta I - \int_{\mathbb{R}^3\times\mathbb{R}^+}\tilde{f}|v-\tilde{U}|^2 dvdI \bigg)Id \\
	&\quad+(1-\theta)\nu
	\bigg(\sum_{j,k}\tilde{f}_{i,j,k}^n (v_j-\tilde{U}_i^n)\otimes(v_j-\tilde{U}_i^n) (\Delta v)^3 \Delta I-\int_{\mathbb{R}^3\times\mathbb{R}^+}\tilde{f}(v-\tilde{U})\otimes(v-\tilde{U})dvdI \bigg) \\
	&\quad+\frac{2\theta}{3+\delta}\bigg( \sum_{j,k}\tilde{f}_{i,j,k}^nI_k^{\frac{2}{\delta}} (\Delta v)^3 \Delta I-\int_{\mathbb{R}^3\times\mathbb{R}^+}\tilde{f}I^{\frac{2}{\delta}}dvdI  \bigg)Id \\
	&\quad+ \tilde{\rho}_i^n\bigg[\frac{(1-\theta)\nu \Delta tA_{\nu,\theta}\theta}{\Delta t + \kappa} (\tilde{T}_\delta)_i^n Id + \frac{(1-\theta)\nu\Delta t(1- A_{\nu,\theta} \theta)}{\Delta t + \kappa}(\tilde{T}_{tr})_i^nId - (1-\theta)\nu \frac{ \Delta t }{\Delta t + \kappa}\tilde{\Theta}_i^n \bigg]\\
	&\equiv \mathcal{I}_{71}+\mathcal{I}_{72}+\mathcal{I}_{73}+\mathcal{I}_{74}.  
	\end{align*}
	For $\mathcal{I}_{71}$, we use Lemma \ref{L.7.2} to obtain
	\begin{align*}
	&\sum_{j,k}\tilde{f}_{i,j,k}^n|v_j-\tilde{U}_i^n|^2 (\Delta v)^3 \Delta I-\int_{\mathbb{R}^3\times\mathbb{R}^+}\tilde{f}|v_j-\tilde{U}|^2 dvdI \\ &\quad \leq C \|f(t^n) - f^n\|_{L_q^\infty} + C \{\| f_0\|_{L_{2,q}^\infty} + 1\} \{(\Delta x)^2 + \Delta v + \Delta I + \Delta v \Delta t\}.
	\end{align*}
	Similar estimates hold for $\mathcal{I}_{72}$ and $\mathcal{I}_{73}$. Together with 
	\[
	\bigg|\frac{(1-\theta)\nu \Delta tA_{\nu,\theta}\theta}{\Delta t + \kappa}\bigg|,\quad  \bigg|\frac{(1-\theta)\nu\Delta t(1- A_{\nu,\theta} \theta)}{\Delta t + \kappa}\bigg|,\quad \bigg|(1-\theta)\nu \frac{ \Delta t }{\Delta t + \kappa}\bigg| \leq C \frac{\Delta t}{\Delta t + \kappa},
	\]
	the macroscopic quantities in $\mathcal{I}_{74}$ are also bounded by $D^n$ in Definition \ref{D.5.2}. Therefore, for $1\leq \alpha, \beta \leq 3$, we have
	\begin{align*}
	|(\tilde{\mathcal{T}}_{\nu,\theta}^{\alpha,\beta})_i^{n} &-\tilde{\mathcal{T}}_{\nu,\theta}^{\alpha,\beta}(x_i,t^n)|\\&=\bigg|\frac{\tilde{\rho}_i^n(\tilde{\mathcal{T}}_{\nu,\theta}^{\alpha,\beta})_i^n - \tilde{\rho}(x_i,t^n)\tilde{\mathcal{T}}_{\nu,\theta}^{\alpha,\beta}(x_i,t^n)}{\tilde{\rho}_i^n} +\frac{\tilde{\rho}(x_i,t^n)\tilde{\mathcal{T}}_{\nu,\theta}^{\alpha,\beta}(x_i,t^n)-\tilde{\rho}_i^n\tilde{\mathcal{T}}_{\nu,\theta}^{\alpha,\beta}(x_i,t^n)}{\tilde{\rho}_i^n}\bigg|\cr
	&\leq \frac{1}{\tilde{\rho}_i^n}\bigg|\tilde{\rho}_i^n(\tilde{\mathcal{T}}_{\nu,\theta}^{\alpha,\beta})_i^n - \tilde{\rho}(x_i,t^n)\tilde{\mathcal{T}}_{\nu,\theta}^{\alpha,\beta}(x_i,t^n)  \bigg|+\frac{|\tilde{\mathcal{T}}_{\nu,\theta}^{\alpha,\beta}(x_i,t^n)|}{\tilde{\rho}_i^n}\bigg|\tilde{\rho}(x_i,t^n)-\tilde{\rho}_i^n\bigg|\cr
	&\leq
	C \|f(t^n) - f^n\|_{L_q^\infty} + C \{\| f_0\|_{L_{2,q}^\infty} + 1\} \{(\Delta x)^2 + \Delta v + \Delta I + \Delta v \Delta t\},
	\end{align*}
	for a constant $C>0$. This completes the proof.

\end{proof}

The following is the main result of this section.
\begin{proposition}\label{P.7.1}
	Suppose that $q>5+\delta$ and $\Delta v$, $\Delta I$ satisfies the condition \eqref{cond T.5.5}. Then, 
	\begin{align*}
	\|\mathcal{M}_{\nu,\theta}&(\tilde{f}(t^n)) - \mathcal{M}_{\nu,\theta}(\tilde{f}^n)\|_{L_q^\infty}\cr
	&\leq C \|f(t^n) - f^n\|_{L_q^\infty} + C \{\| f_0\|_{L_{2,q}^\infty} +1\} \{(\Delta x)^2 + \Delta v + \Delta I + \Delta v \Delta t\}.
	 \end{align*}
\end{proposition}
\begin{proof}
	We begin by writting
	\begin{align*}
	\mathcal{M}_{\nu,\theta}&(\tilde{f}(x_i,v_j,I_k,t^n)) - \mathcal{M}_{\nu,\theta}(\tilde{f}_{i,j,k}^n) \cr &= 
	\mathcal{M}_{\nu,\theta}(\tilde{\rho}(x_i,t^n),\tilde{U}(x_i,t^n),\tilde{\mathcal{T}}_{\nu,\theta}(x_i,t^n))(v_j,I_k) - \mathcal{M}_{\nu,\theta}(\tilde{\rho}_i^n,\tilde{U}_i^n,(\tilde{\mathcal{T}}_{\nu,\theta})_i^n)(v_j,I_k).
	\end{align*}
	Then,
	\begin{align*} 
	\mathcal{M}_{\nu,\theta}&(\tilde{\rho}(x_i,t^n),\tilde{U}(x_i,t^n),\tilde{\mathcal{T}}_{\nu,\theta}(x_i,t^n))(v_j,I_k) - \mathcal{M}_{\nu,\theta}(\tilde{\rho}_i^n,\tilde{U}_i^n,(\tilde{\mathcal{T}}_{\nu,\theta})_i^n)(v_j,I_k) \cr
	&=(\tilde{\rho}(x_i,t^n)- \tilde{\rho}_i^n)\int_0^1 \frac{\partial \mathcal{M}_{\nu,\theta}}{\partial \rho}(\eta)d\eta +(\tilde{U}(x_i,t^n)-\tilde{U}_i^n)\int_0^1 \frac{\partial \mathcal{M}_{\nu,\theta}}{\partial U}(\eta)d\eta\cr
	&\quad +\sum_{1\leq \alpha, \beta \leq 3 }(\tilde{\mathcal{T}}_{\nu,\theta}^{\alpha,\beta}(x_i,t^n)- (\tilde{\mathcal{T}}_{\nu,\theta}^{\alpha,\beta})_i^n)\int_0^1 \frac{\partial \mathcal{M}_{\nu,\theta}}{\partial \mathcal{T}_{\nu,\theta}^{\alpha,\beta}}(\eta)d\eta \cr&\quad +(\tilde{T}_{\theta}(x_i,t^n)-(\tilde{T}_{\theta})_i^n)\int_{0}^{1}\frac{\partial\mathcal{M}_{\nu,\theta} }{\partial T_{\theta}}({\eta}) d{\eta} \cr
	&\equiv J_1 + J_2 + J_3 + J_4
	\end{align*}
	where
	\[
	\frac{\partial \mathcal{M}_{\nu,\theta}}{\partial X}(\eta):=\frac{\partial \mathcal{M}_{\nu,\theta}}{\partial X}\bigg|_{X\equiv\big(\rho,U,\mathcal{T}_{\nu,\theta},T_{\theta}\big)=\big(\tilde{\rho}_i^n(\eta),\tilde{U}_i^n(\eta),(\tilde{\mathcal{T}}_{\nu,\theta})_i^n(\eta), (\tilde{T}_{\theta})_i^n(\eta)\big)}
	\]
	and
	\begin{align*}
		&\bigg(\tilde{\rho}_i^n(\eta),\tilde{U}_i^n(\eta),(\tilde{\mathcal{T}}_{\nu,\theta})_i^n(\eta),(\tilde{T}_{I,\delta})_i^n(\eta),(\tilde{T}_{\theta})_i^n(\eta),(\tilde{T}_{\delta})_i^n(\eta)\bigg)\cr&\qquad:= (1-\eta)\bigg(\tilde{\rho}(x_i,t^n),\tilde{U}(x_i,t^n),\tilde{\mathcal{T}}_{\nu,\theta}(x_i,t^n),(\tilde{T}_{I,\delta})(x_i,t^n),(\tilde{T}_{\theta})(x_i,t^n),(\tilde{T}_{\delta})(x_i,t^n)\bigg) \cr
		&\qquad\quad + \eta \bigg(\tilde{\rho}_i^n,\tilde{U}_i^n,(\tilde{\mathcal{T}}_{\nu,\theta})_i^n,(\tilde{T}_{I,\delta})_i^n ,(\tilde{T}_{\theta})_i^n,(\tilde{T}_{\delta})_i^n\bigg)
	\end{align*}
	for $\eta \in [0,1]$.
Since each macroscopic quantity is given by the convex combination of continous and discrete macroscopic fields, its estimate can be directly obtained by combining the estimates of continous solution in Theorem \ref{T.3.1} and those of discrete solution in Theorem \ref{T.5.5} as follows:
	\begin{align}\label{G-2}
		\begin{split}
&\tilde{\rho}_i^n(\eta),~\tilde{U}_i^n(\eta),~(\tilde{T}_{\delta})_i^n(\eta),~(\tilde{T}_{\theta})_i^n(\eta)  \leq C_{T^f}\cr
			&\tilde{\rho}_i^n(\eta),~(\tilde{T}_{\delta})_i^n(\eta),~ (\tilde{T}_{\theta})_i^n(\eta) \geq C_{T^f} e^{-C_{T^f}}\cr
			& k^{\top}\{ (\tilde{\mathcal{T}}_{\nu,\theta})_i^n(\eta)\}k \geq C_{T^f} e^{-C_{T^f}}|k|^2, \quad k\in \mathbb{R}^3.
		\end{split}
	\end{align} 
	On the other hand, Brum-Minkowski inequality implies that
	\begin{align*}
		\det \{(\tilde{\mathcal{T}}_{\nu,\theta})_i^n(\eta)\} &= \det \{(1-\eta)\tilde{\mathcal{T}}_{\nu,\theta}(x_i,t^n) + \eta (\tilde{\mathcal{T}}_{\nu,\theta})_i^n\}\cr
		&\geq \det \{\tilde{\mathcal{T}}_{\nu,\theta}(x_i,t^n)\}^{1-\eta} \det \{(\tilde{\mathcal{T}}_{\nu,\theta})_i^n\}^{\eta}\cr
		&\geq \{C_{T^f} e^{-C_{T^f}}\}^{1-\eta} \{C_{T^f} e^{-C_{T^f}}\}^{\eta}\cr
		&\geq C_{T^f} e^{-C_{T^f}},
	\end{align*}
	from which we have
	\begin{align*}
		\mathcal{M}_{\nu,\theta}(\eta)(v_j,I_k)&=\frac{\tilde{\rho}_i^n(\eta)\Lambda_\delta}{\sqrt{\det\left(2 \pi (\tilde{\mathcal{T}}_{\nu,\theta})_i^n(\eta) \right)}((\tilde{T}_\theta)_i^n(\eta))^\frac{\delta}{2}}\cr
		&\quad \times \exp \left({-\frac{(v_j-\tilde{U}_i^n(\eta))^{\top} ((\tilde{\mathcal{T}}_{\nu,\theta})_i^n(\eta))^{-1}(v_j-\tilde{U}_i^n(\eta))}{2} -\frac{I_k^{\frac{2}{\delta}}}{(\tilde{T}_\theta)_i^n(\eta)}}\right)\cr
		&\leq C_{T^f}\exp\left(-C_{T^f}\left(|v_j-\tilde{U}_i^n(\eta)|^2+ I_k^{\frac{2}{\delta}}\right)\right).
	\end{align*}
	Now, we return to the estimate of $J_i$ for $i=1,2,3,4$.
	We bound $J_1$ with
	\begin{align*}
		\bigg| \int_0^1 \frac{\partial \mathcal{M}_{\nu,\theta}}{\partial \rho}(\eta)d\eta \bigg| &=  \int_0^1 \frac{\mathcal{M}_{\nu,\theta}}{\tilde{\rho}_i^n(\eta)} d\eta \cr
		&\leq   \int_0^1 C_{T^f}\exp\left(-C_{T^f}(|v_j-\tilde{U}_i^n(\eta)|^2+ I_k^{\frac{2}{\delta}})\right) d\eta.
	\end{align*}
	For $J_2$, we recall from Lemma \ref{L.4.2} that
	\begin{align*}
		&\lambda\theta (\tilde{T}_{\delta})_i^n Id \leq (\tilde{\mathcal{T}}_{\nu,\theta})_i^n.
	\end{align*}
	This, combined with Proposition \ref{P.6.2}, gives
	\begin{align*}
		C\theta (\tilde{T}_{\delta})_i^n(\eta) Id \leq (\tilde{\mathcal{T}}_{\nu,\theta})_i^n(\eta).
	\end{align*}
	Now, we consider the following inequality:
	\begin{align*}
	\bigg|\frac{\partial \mathcal{M}_{\nu,\theta}}{\partial U}(\eta) \bigg| \leq \left( \left|((\tilde{\mathcal{T}}_{\nu,\theta})_i^n(\eta))^{-1}(v_j-\tilde{U}_i^n(\eta))\right| + \left|(v_j-\tilde{U}_i^n(\eta))^{\top} ((\tilde{\mathcal{T}}_{\nu,\theta})_i^n(\eta))^{-1}\right| \right) \mathcal{M}_{\nu,\theta}(\eta).
	\end{align*}
	To estimate the upper bound, we introduce $X=v_j-\tilde{U}_i^n(\eta)$ and obtain
	\begin{align}\label{G-3}
	\begin{split}
		|X^{\top}((&\tilde{\mathcal{T}}_{\nu,\theta})_i^n(\eta))^{-1}|\cr
		&\leq \sup_{|Y|\leq 1}|X^{\top}((\tilde{\mathcal{T}}_{\nu,\theta})_i^n(\eta))^{-1} Y|\cr
		&\leq \sup_{|Y|\leq 1}\left|(X+Y)^{\top}((\tilde{\mathcal{T}}_{\nu,\theta})_i^n(\eta))^{-1} (X+Y) - X^{\top}((\tilde{\mathcal{T}}_{\nu,\theta})_i^n(\eta))^{-1} X - Y^{\top}((\tilde{\mathcal{T}}_{\nu,\theta})_i^n(\eta))^{-1} Y\right|\cr
		&\leq \frac{C}{\theta} \sup_{|Y|\leq 1}\bigg|\frac{|X+Y|^2 - |X|^2 - |Y|^2}{((\tilde{T}_{\delta})_i^n(\eta))^{-1}}\bigg|\cr
		&\leq \frac{C}{\theta }(1+|v_j-\tilde{U}(\eta)|^2).
	\end{split}
	\end{align}
	In the last line, we use Lemma \ref{L.4.2}. Similarly, we compute
	\begin{align}\label{G-4}
	|((\tilde{\mathcal{T}}_{\nu,\theta})_i^n(\eta))^{-1}X|\leq \frac{C}{\theta }(1+|v_j-\tilde{U}(\eta)|^2).
	\end{align}
	Consequently,
	\begin{align*}
	\bigg| \int_0^1 \frac{\partial \mathcal{M}_{\nu,\theta}}{\partial U}(\eta) d\eta \bigg|
	&\leq   \int_0^1 \frac{C_{T^f}}{\theta }(1+|v_j-\tilde{U}(\eta)|^2) \exp\left(-C_{T^f}(|v_j-\tilde{U}_i^n(\eta)|^2+ I_k^{\frac{2}{\delta}})\right) d\eta.
	\end{align*}
	To estimate $J_3$, for $1\leq \alpha,\beta \leq 3$, we compute 
	\begin{align*}
	\frac{\partial\mathcal{M}_{\nu,\theta}}{\partial \mathcal{T}_{\nu,\theta}^{\alpha,\beta}}({\eta})=
	\frac{1}{2}
	\bigg[&-\frac{1}{\det(\tilde{\mathcal{T}}_{\nu,\theta})_i^n(\eta)}\frac{\partial\det(\mathcal{T}_{\nu,\theta})}{\partial\mathcal{T}_{{\nu,\theta}}^{\alpha,\beta}}(\eta)\cr
	&+(v_j-\tilde{U}_i^n(\eta))^{\top}\tilde{\mathcal{T}}_{\nu,\theta}^{-1}(\eta)\left(\frac{\partial(\mathcal{T}_{\nu,\theta})}{\partial \mathcal{T}_{\nu,\theta}^{\alpha,\beta}}(\eta)\right)\tilde{\mathcal{T}}_{\nu,\theta}^{-1}(\eta)(v_j-\tilde{U}(\eta))\bigg]\mathcal{M}_{\nu,\theta}({\eta}),
	\end{align*}
	where 
	\begin{align*}
		\frac{\partial\det(\mathcal{T}_{\nu,\theta})}{\partial\mathcal{T}_{{\nu,\theta}}^{\alpha,\beta}}(\eta):=\frac{\partial\det(\mathcal{T}_{\nu,\theta})}{\partial\mathcal{T}_{{\nu,\theta}}^{\alpha,\beta}}\bigg|_{\mathcal{T}_{{\nu,\theta}}=(\tilde{\mathcal{T}}_{{\nu,\theta}})_i^n(\eta)},\quad \frac{\partial(\mathcal{T}_{\nu,\theta})}{\partial \mathcal{T}_{\nu,\theta}^{\alpha,\beta}}(\eta):=\frac{\partial(\mathcal{T}_{\nu,\theta})}{\partial \mathcal{T}_{\nu,\theta}^{\alpha,\beta}}\bigg|_{\mathcal{T}_{\nu,\theta}=(\tilde{\mathcal{T}}_{\nu,\theta})_i^n(\eta)}.
	\end{align*}
	Now, we prove the following estimates: 
	\begin{align*}
	(\mathcal{F}_1) &:~ \left|(v_j-\tilde{U}(\eta))^{\top}\tilde{\mathcal{T}}_{\nu,\theta}^{-1}(\eta)\left(\frac{\partial \mathcal{T}_{\nu,\theta}}{\partial\mathcal{T}_{\nu,\theta}^{\alpha,\beta}}(\eta)\right)\tilde{\mathcal{T}}_{\nu,\theta}^{-1}(\eta)(v_j-\tilde{U}(\eta))\right|\leq \bigg(\frac{C}{\theta }(1+|v_j-\tilde{U}(\eta)|^2)\bigg)^2  \cr
	(\mathcal{F}_2) &:~\det(\tilde{\mathcal{T}}_{\nu,\theta})_i^n(\eta)\geq \theta^3 C.\cr
	(\mathcal{F}_3) &:~\left|\frac{\partial \det(\mathcal{T}_{\nu,\theta})_i^n}{\partial\mathcal{T}^{\alpha,\beta}_{\nu,\theta}}(\eta)\right|
	\leq C.\cr
	\end{align*}
	\noindent$\bullet~(\mathcal{F}_1)$: We use that $\tilde{\mathcal{T}}_{\nu,\theta}^{-1}$ is symmetric matrix and $\tilde{\mathcal{T}}^{\alpha, \beta}_{\nu,\theta } = \tilde{\mathcal{T}}^{\beta,\alpha}_{\nu,\theta }$ to obtain
	\[
	\left|X^{\top}\left(\frac{\partial\mathcal{T}_{\nu,\theta}}{\partial\mathcal{T}^{\alpha, \beta}_{\nu,\theta }}(\eta)\right)Y\right| = |X^\alpha Y^\beta+X^\beta Y^\alpha|\leq|X||Y|.
	\]
	This gives
	\begin{align*}
	&\left|(v_j-\tilde{U}(\eta))^{\top}\tilde{\mathcal{T}}_{\nu,\theta}^{-1}(\eta)\left(\frac{\partial\mathcal{T}_{\nu,\theta}}{\partial\mathcal{T}_{\nu,\theta}^{\alpha,\beta}}(\eta)\right)\tilde{\mathcal{T}}_{\nu,\theta}^{-1}(\eta)(v_j-\tilde{U}(\eta))\right|\cr
	&\qquad\qquad\qquad \leq
	|(v_j-\tilde{U}(\eta))^{\top}((\tilde{\mathcal{T}}_{\nu,\theta})_i^n(\eta))^{-1}||((\tilde{\mathcal{T}}_{\nu,\theta})_i^n(\eta))^{-1}(v_j-\tilde{U}(\eta))|\cr
	&\qquad\qquad\qquad \leq \bigg(\frac{C}{\theta }(1+|v_j-\tilde{U}(\eta)|^2)\bigg)^2,
	\end{align*}
	where we use \eqref{G-3} and \eqref{G-4}.\newline
	$\bullet~(\mathcal{F}_2)$:
	By (\ref{G-2}), we have
	\[
	\det(\tilde{\mathcal{T}}_{\nu,\theta})_i^n(\eta)\geq \big(C\theta (\tilde{T}_{\delta})_i^n(\eta) ^3\big)=\theta^3 C.
	\]\newline
	$\bullet~(\mathcal{F}_3)$:	
	Recalling the definition of $\tilde{\Theta}$, $\tilde{T}_{tr}$, $\tilde{T}_{\delta}$ and $\tilde{\mathcal{T}}_{\nu,\theta}$, for $1\leq \alpha, \beta \leq 3$, we have 
	\begin{align*}
		&|\tilde{\Theta}^{\alpha,\beta}(x_i,t^n)| \leq 3 \tilde{T}_{tr}(x_i,t^n)\cr
		&\tilde{T}_{\delta}(x_i,t^n)= \frac{3}{3+\delta}\tilde{T}_{tr}(x_i,t^n)+\frac{\delta}{3+\delta}\tilde{T}_{I,\delta}(x_i,t^n)\geq \frac{3}{3+\delta}\tilde{T}_{tr}(x_i,t^n).
	\end{align*}
	and
	\begin{align*}
	|\tilde{\mathcal{T}}^{\alpha,\beta}_{\nu,\theta}(x_i,t^n)|&\leq \theta \tilde{T}_\delta(x_i,t^n) + (1-\theta)\big\{(1-\nu)\tilde{T}_{tr}(x_i,t^n) + \nu |\tilde{\Theta}^{\alpha,\beta}(x_i,t^n)| \big\}\cr
	&\leq \theta \tilde{T}_\delta(x_i,t^n) + (1-\theta)(1+2\nu)\tilde{T}_{tr}(x_i,t^n) \cr
	&\leq \theta \tilde{T}_\delta(x_i,t^n) + (1-\theta)(1+2\nu)\frac{3+\delta}{3}\tilde{T}_{\delta}(x_i,t^n) \cr
	&\leq (1+2\nu)\frac{3+\delta}{3}\tilde{T}_{\delta}(x_i,t^n).		
	\end{align*}
That is,
	\begin{align*}
	&\left|(\tilde{\mathcal{T}}^{\alpha,\beta}_{\nu,\theta})_i^n\right| \leq C(\tilde{T}_{\delta})_i^n,
	\end{align*}
	which implies
	\begin{align*}
	&|(\tilde{\mathcal{T}}^{\alpha,\beta}_{\nu,\theta})_i^n(\eta)| \leq C(\tilde{T}_{\delta})_i^n (\eta)\leq C.
	\end{align*}
	For simplicity, we only cover the case:  $(\alpha,\beta)=(1,2)$. A direct calculation gives
	\begin{align*}
	\frac{\partial\det\mathcal{T}_{\nu,\theta}}{\partial\mathcal{T}^{1,2}_{\nu,\theta}}(\eta)
	=\tilde{\mathcal{T}}^{2,3}_{\nu,\theta}(\eta)\tilde{\mathcal{T}}^{3,1}_{\nu,\theta}(\eta)-\tilde{\mathcal{T}}^{3,3}_{\nu,\theta}(\eta)
	\tilde{\mathcal{T}}^{2,1}_{\nu,\theta}(\eta),
	\end{align*}
	which is a second order polynomial of $(\tilde{\mathcal{T}}^{\alpha,\beta}_{\nu,\theta})_i^n(\eta)$ for $1\leq \alpha,\beta \leq 3$. Therefore,
	\begin{align*}
	\bigg|\frac{\partial\det\mathcal{T}_{\nu,\theta}}{\partial\mathcal{T}^{1,2}_{\nu,\theta}}(\eta)\bigg|\leq C,
	\end{align*}
	for some constant $C$. This completes the proof for claims.
	
	Using $(\mathcal{F}_1)$, $(\mathcal{F}_2)$ and $(\mathcal{F}_3)$, we can bound the integral $J_3$ as 
	\begin{align*}
	\left|\int_{0}^{1}\frac{\partial\mathcal{M}_{\nu,\theta}}{\partial \mathcal{T}_{\nu,\theta}^{\alpha,\beta}}({\eta})d{\eta}\right|
	\leq \int_{0}^{1} \bigg[\frac{C}{\theta^3} + \bigg(\frac{C}{\theta }(1+|v_j-\tilde{U}(\eta)|^2)\bigg)^2\bigg]\mathcal{M}_{\nu,\theta}({\eta})d{\eta}.
	\end{align*}
	Then, 
	\begin{align*}
	\left|\int_{0}^{1}\frac{\partial\mathcal{M}_{\nu,\theta} }{\partial \mathcal{T}_{\nu,\theta}^{\alpha,\beta}}({\eta})d{\eta}\right|
	&\leq
	C\left(\frac{1}{\theta^2} + \frac{1}{\theta^3}\right)\int_{0}^{1} \big(1+|v_j-\tilde{U}(\eta)|^2\big)^2\exp{\left({-C_{T^f}(|v_j-\tilde{U}(\eta)|^2+I_k^{2/\delta})}\right)}d{\eta}.
	\end{align*}
	For $J_4$, we begin with
	\begin{align*}
		\frac{\partial \mathcal{M}_{\nu,\theta}}{\partial T_{\theta}}(\eta) &= \left(\frac{2I_k^{\frac{2}{\delta}}-\delta \tilde{T}_{\theta}(\eta)}{2(\tilde{T}_{\theta}(\eta))^2}\right)\mathcal{M}_{\nu,\theta}\leq \left( \frac{1}{(\tilde{T}_{\theta}(\eta))^2} + \frac{\delta}{\tilde{T}_{\theta}(\eta)}\right)\left(1+I_k^{\frac{2}{\delta}} \right)\mathcal{M}_{\nu,\theta}.
	\end{align*}
	Since there exist a lower bound for $\tilde{T}_{\theta}(\eta)$, we have
	\begin{align*}
	\bigg|\int_{0}^{1}\frac{\partial\mathcal{M}_{\nu,\theta}}{\partial T_{\theta}}(\eta)d{\eta}\bigg|
	&\leq
	C(1+ I_k^{\frac{2}{\delta}})\exp\left(-C_{T^f}(|v_j-\tilde{U}_i^n(\eta)|^2+ I_k^{\frac{2}{\delta}})\right).
	\end{align*}
	Combining all the estimates for $J_i$, $i=1,2,3,4$, we finally obtain
	\begin{align}\label{G-8}
	\begin{split}
	&|\mathcal{M}_{\nu,\theta}(\tilde{f}(x_i,v_j,I_k,t^n)) - \mathcal{M}_{\nu,\theta}(\tilde{f}_{i,j,k}^n)|\cr
	&\leq C\big(1+\frac{1}{\theta} + \frac{1}{\theta^2} + \frac{1}{\theta^3}\big)\Big\{|\tilde{\rho}-\tilde{\rho}_i^n|+|\tilde{U}-\tilde{U}_i^n|+\sum_{1\leq \alpha, \beta \leq 3}|\tilde{\mathcal{T}}_{\nu,\theta}^{\alpha,\beta}-(\tilde{\mathcal{T}}_{\nu,\theta}^{\alpha,\beta})_i^n|+|\tilde{T}_{\theta}-(\tilde{T}_{\theta})_i^n|\Big\} \cr
	&\quad \times \big(1+|v_j-\tilde{U}(\eta)|^2 +|v_j-\tilde{U}(\eta)|^4 +  I_k^{\frac{2}{\delta}}\big) e^{-C(|v_j-\tilde{U}_i^n(\eta)|^2+I_k^{2/\delta})}.
	\end{split}
	\end{align}
	Now, recall that $\tilde{U}_i^n(\eta)\leq C_{T^f}$ to derive
	\begin{align*}
	(1+|v_j|^2+I_k^{\frac{2}{\delta}})^\frac{q}{2} &= (1+|v_j-\tilde{U}_i^n(\eta)+\tilde{U}_i^n(\eta)|^2+I_k^{\frac{2}{\delta}})^\frac{q}{2}\leq C(1+|v_j-\tilde{U}_i^n(\eta)|^2+I_k^{\frac{2}{\delta}})^\frac{q}{2},
	\end{align*}
	which further gives
	\begin{align}\label{G-9}
	\begin{split}
	&(1+|v_j|^2+I_k^{\frac{2}{\delta}})^\frac{q}{2} \big(1+|v_j-\tilde{U}(\eta)|^2 +|v_j-\tilde{U}(\eta)|^4 +  I_k^{\frac{2}{\delta}}\big) e^{-C(|v_j-\tilde{U}_i^n(\eta)|^2+I_k^{2/\delta})}\cr
	&\leq C(1+|v_j-\tilde{U}_i^n(\eta)|^2+I_k^{\frac{2}{\delta}})^\frac{q}{2}\big(1+|v_j-\tilde{U}(\eta)|^2 +|v_j-\tilde{U}(\eta)|^4 +  I_k^{\frac{2}{\delta}}\big) e^{-C(|v_j-\tilde{U}_i^n(\eta)|^2+I_k^{2/\delta})}\cr
	&\leq C(1+|v_j-\tilde{U}_i^n(\eta)|^2+I_k^{\frac{2}{\delta}})^{\frac{q}{2}+3} e^{-C(|v_j-\tilde{U}_i^n(\eta)|^2+I_k^{2/\delta})}.
	\end{split}
	\end{align}
	Note that the last upper bound can be understood as the form of $C(1+x)^{\frac{q}{2}+3}e^{-Cx}$, hence it is uniformly bounded for $x\geq 0$. To obtain desired estimate, we multiply $(1+|v_j|^2+I_k^{\frac{2}{\delta}})^\frac{q}{2}$ on both sides of \eqref{G-8} and take supremum, then we have from \eqref{G-9} that
	\begin{align*}
	&\|\mathcal{M}_{\nu,\theta}(\tilde{f}(x_i,v_j,I_k,t^n)) - \mathcal{M}_{\nu,\theta}(\tilde{f}_{i,j,k}^n)\|_{L_q^\infty}\cr
	&\leq C\big(1+\frac{1}{\theta} + \frac{1}{\theta^2} + \frac{1}{\theta^3}\big)\Big\{|\tilde{\rho}-\tilde{\rho}_i^n|+|\tilde{U}-\tilde{U}_i^n|+\sum_{1\leq \alpha, \beta \leq 3}|\tilde{\mathcal{T}}_{\nu,\theta}^{\alpha,\beta}-(\tilde{\mathcal{T}}_{\nu,\theta}^{\alpha,\beta})_i^n|+|\tilde{T}_{\theta}-(\tilde{T}_{\theta})_i^n|\Big\}.
	\end{align*}
	This, together with Lemma \ref{L.7.3}, gives the desired estimate.
\end{proof}

\section{Proof of Theorem \ref{T.3.2}}
Here, we prove our main theorem. We first subtract \eqref{consistent form} from \eqref{B-6} and take $L_q^\infty$-norm: 
\begin{align*}
&\|f^{n+1} - f(t^{n+1})\|_{L_q^\infty}\cr &\qquad=\frac{\kappa}{\kappa+A_{\nu,\theta}\Delta t} \|\tilde{f}^n- \tilde{f}(t^n)\|_{L_q^\infty}  + \frac{A_{\nu,\theta}\Delta t}{\kappa+A_{\nu,\theta}\Delta t} \|\mathcal{M}_{\nu,\theta}(\tilde{f}^n) - \mathcal{M}_{\nu,\theta}(\tilde{f})(t^n)\|_{L_q^\infty}\cr
&\quad\qquad+ \frac{A_{\nu,\theta}}{\kappa+A_{\nu,\theta}\Delta t} \|R_1\|_{L_q^\infty} + \frac{A_{\nu,\theta}}{\kappa+A_{\nu,\theta}\Delta t} \|R_2\|_{L_q^\infty}.
\end{align*}
Next, we recall Lemma \ref{L.6.2}, Lemma \ref{L.7.1} and Proposition \ref{P.7.1}:
\begin{align*}
&\|R_1\|_{L_q^\infty} + \|R_2\|_{L_q^\infty}\leq C(\Delta t)^2,\cr
&\|\tilde{f}(t^n) - \tilde{f}^n\|_{L_q^\infty} \leq \|f(t^n) - f^n\|_{L_q^\infty} + C (\Delta x)^2,\cr
&\|\mathcal{M}_{\nu,\theta}(\tilde{f}(t^n)) - \mathcal{M}_{\nu,\theta}(\tilde{f}^n)\|_{L_q^\infty} \leq C \left(\|f(t^n) - f^n\|_{L_q^\infty} + \{(\Delta x)^2 + \Delta v + \Delta I + \Delta v \Delta t\}\right),
\end{align*}
where $C$ is a constant which can bounded regardless of the values of $\Delta t$. 
From these estimates, we obtain
\begin{align}\label{H-1}
\begin{split}
\|f^{n+1} - f(t^{n+1})\|_{L_q^\infty}
&\leq \frac{\kappa + C A_{\nu,\theta}\Delta t}{\kappa+A_{\nu,\theta}\Delta t} \|f(t^n) - f^n\|_{L_q^\infty} \cr
&+ \frac{C}{\kappa+A_{\nu,\theta}\Delta t} \bigg( \kappa(\Delta x)^2 + A_{\nu,\theta}\Delta t\big( (\Delta x)^2 +\Delta v + \Delta I + \Delta v \Delta t + \Delta t\big)\bigg).
\end{split}
\end{align}
For the sake of simplicity, we introduce 
$$\Gamma_n := \|f^n - f(t^n)\|_{L_q^\infty}$$
and
$$P(\Delta x, \Delta v, \Delta I, \Delta t):= \frac{C}{\kappa+A_{\nu,\theta}\Delta t} \bigg( \kappa(\Delta x)^2 + A_{\nu,\theta}\Delta t\big( (\Delta x)^2 + \Delta v + \Delta I + \Delta v \Delta t + \Delta t\big)\bigg).$$
Then, we write \eqref{H-1} in a recurrence form as follows:
\begin{align*}
	\Gamma_{n+1} \leq ( 1 + Q\Delta t) \Gamma_n + P(\Delta x, \Delta v, \Delta I, \Delta t)
\end{align*} 
where $Q:=\frac{C A_{\nu,\theta}}{\kappa+A_{\nu,\theta}\Delta t}$.
Since it is assumed that there is no error in the initial step:
\[
\Gamma_0 = \|f^0 - f(t^0)\|_{L_q^\infty}=0,
\]
we have from $n \Delta t \leq N_t \Delta t = T^f$ that 
\begin{align*}
\Gamma_{n+1} &\leq ( 1+Q\Delta t )^{n+1}\Gamma_0 + \sum_{k=0}^{n} ( 1+Q\Delta t )^k P(\Delta x, \Delta v, \Delta I, \Delta t)\cr
&\leq \frac{(1+Q\Delta t) ^{N_t}-1}{(1+Q\Delta t)-1}P(\Delta x, \Delta v, \Delta I, \Delta t)\cr
&\leq \frac{1}{Q\Delta t}e^{Q T^f}P(\Delta x, \Delta v, \Delta I, \Delta t).
\end{align*}
In the last line, we use $(1+x)^n \leq e^{nx}$. Using $\Delta v <\frac{1}{2}$ and
\begin{align*}
	\frac{P(\Delta x, \Delta v, \Delta I, \Delta t)}{\Delta t}\leq  \frac{C(\kappa+A_{\nu,\theta})}{\kappa+A_{\nu,\theta}\Delta t} \bigg( \frac{(\Delta x)^2}{\Delta t} + (\Delta x)^2 + \Delta v + \Delta I + \Delta v \Delta t + \Delta t\bigg),
\end{align*}
we derive
\begin{align*}
	\Gamma_{n+1} 
	\leq  \frac{2}{Q}e^{Q T^f}\frac{C(\kappa+A_{\nu,\theta})}{\kappa+A_{\nu,\theta}\Delta t} \bigg( \frac{(\Delta x)^2}{\Delta t} + (\Delta x)^2 + \Delta v + \Delta I + \Delta t\bigg).
\end{align*}
This completes the proof.

\section*{Conclusion}
In this paper, we present an implicit semi-Lagrangian scheme for the ES-BGK model for polyatomic gases. The main result is the convergence estimate of the scheme using argument previously adopted in \cite{RSY} for BGK model and \cite{RY} for ES-BGK model for monatomic gas. For the proof of convergence estimate, the lower bound estimate for polyatomic temperature is crucially used to prevent the discrete polyatomic ellipsoidal Gaussian from degenerating into Dirac delta. The restriction of our result is that is that convergence estimate holds for fixed value of Knudsen number and relaxation parameter $\theta$. Our proof covers the 
biatomic molecules with no vibrational degree of freedom. In future work we shall try to remove some of these restrictions, in particular we plan to make use of the asymptotic preserving property of the method to obtain a convergence estimate which is
uniform in the Knudsen number.

\section*{Acknowledgments}
S.-Y. Cho has been supported by ITN-ETN Horizon 2020 Project ModCompShock, Modeling and Computation on Shocks and Interfaces, Project Reference 642768. 
S.-B. Yun has been supported by Samsung Science and Technology Foundation under Project Number SSTF-BA1801-02. 
All the authors would like to thank the Italian Ministry of Instruction, University and Research (MIUR) to support this research with funds coming from PRIN Project 2017 (No. 2017KKJP4X entitled “Innovative numerical methods for evolutionary partial differential equations and applications”). 
S. Boscarino has been supported by the University of Catania (“Piano della Ricerca 2016/2018, Linea di intervento 2”). S. Boscarino and G. Russo are members of the INdAM Research group GNCS. 

\bibliographystyle{amsplain}

\begin{thebibliography}{10}


\bibitem{ABLP} Andries, P., Bourgat, J.-F., Le Tallec, P.; Perthame, B.: Numerical comparison between the Boltzmann and ES-BGK models for rarefied gases. Comput. Methods Appl. Mech. Engrg. 191 (2002), no. 31, 3369–3390.

\bibitem{ALPP} Andries, P., Le Tallec, P., Perlat, J.-P., Perthame, B.: The Gaussian-BGK model of Boltzmann equation
with small Prandtl number. Eur. J. Mech. B Fluids {\bf 19} (2000), no. 6, 813-830.

\bibitem{BDMMMM} Baranger, C., Dauvois, Y., Marois, G., Mathé, J., Mathiaud, J.,  Mieussens, L.: A BGK model for high temperature rarefied gas flows. European Journal of Mechanics-B/Fluids, 80 (2020), 1-12.


\bibitem{BIP} Bernard, F., Iollo, A., Puppo, G.: BGK polyatomic model for rarefied flows. J. Sci. Comput. 78 (2019), no. 3, 1893–1916.

\bibitem{BGK} Bhatnagar, P. L., Gross, E. P. and Krook, M.: A model for collision processes in gases. Small amplitude
process in charged and neutral one-component systems, Physical Revies, 94 (1954), 511-525.


\bibitem{BCRY} Boscarino, S., Cho, S.-Y., Russo, G. and Yun, S.-B.: High order conservative Semi-Lagrangian scheme for the BGK model of the Boltzmann equation, arXiv preprint 	arXiv:1905.03660 (2019).





\bibitem{BS2000} Brull, S., Schneider, J.: A new approach for the ellipsoidal statistical model. Contin. Mech. Thermodyn. 20 (2008), no. 2, 63–74. 


\bibitem{BS} Brull, S., Schneider, J. : On the ellipsoidal statistical model for polyatomic gases.
Contin. Mech. Thermodyn. {\bf20} (2009), no. 8, 489–508.



\bibitem{CJR} Caflisch, R. E., Jin, S., and Russo, G.: Uniformly accurate schemes for hyperbolic systems with relaxation. SIAM Journal on Numerical Analysis, 34(1) (1997), 246-281.

\bibitem{CL} Cai, Z., Li, R.: The NRxx method for polyatomic gases. Journal of Computational Physics, 267, 63-91 (2014).



\bibitem{CMS} Crouseilles, N., Mehrenberger, M., Sonnendrücker, E.: Conservative semi-Lagrangian schemes for Vlasov equations. Journal of Computational Physics, 229(6) (2010), 1927-1953.


\bibitem{DP} Dimarco, G., Pareschi, L.: Numerical methods for kinetic equations. Acta Numerica, 23 (2014), 369-520.


\bibitem{CRG} Elisabetta, C., Ferretti, R., Russo, G.: "A Weighted Essentially Nonoscillatory, Large Time-Step Scheme for Hamilton--Jacobi Equations." SIAM Journal on Scientific Computing 27.3 (2005): 1071-1091.


\bibitem{FSB} Filbet, F., Sonnendrücker, E., and Bertrand, P.: Conservative numerical schemes for the Vlasov equation. Journal of Computational Physics, 172(1) (2001), 166-187.

\bibitem{GRS2} Groppi, M., Russo, G.,  Stracquadanio, G.: Boundary conditions for semi-Lagrangian methods for the BGK model. Communications in Applied and Industrial Mathematics, 7(3) (2016), 138-164.


\bibitem{GRS} Groppi, M., Russo, G., Stracquadanio, G.: High order semi-Lagrangian methods for the BGK equation. Commun. Math. Sci. 14 (2016), no. 2, 389–414.


\bibitem{GRS3} Groppi, M., Russo, G.,  Stracquadanio, G.: Semi-Lagrangian Approximation of BGK Models for Inert and Reactive Gas Mixtures. In Meeting on Particle Systems and PDE's (pp. 53-80). Springer, Cham. (2016, November)



\bibitem{H} Holway, L.H.: Kinetic theory of shock structure using and ellipsoidal distribution function. Rarefied Gas Dynamics, Vol. I (Proc. Fourth Internat. Sympos., Univ. Toronto, 1964), Academic Press, New York, (1966), pp. 193-215.



\bibitem{J} Jin, S.: Runge-Kutta methods for hyperbolic conservation laws with stiff relaxation terms. Journal of Computational Physics, 122(1) (1995), 51-67.

\bibitem{KPP} Klingenberg, C., Pirner, M., Puppo, G.: A consistent kinetic model for a two-component mixture of polyatomic molecules. arXiv preprint (2018) arXiv:1806.11486.


\bibitem{KA} Kosuge, S., Aoki, K.: Shock-wave structure for a polyatomic gas with large bulk viscosity. Physical Review Fluids, 3(2) (2018), 023401.


\bibitem{KAG} Kosuge, S., Aoki, K., Goto, T.: Shock wave structure in polyatomic gases: Numerical analysis using a model Boltzmann equation. In AIP Conference Proceedings (Vol. 1786, No. 1, p. 180004)(2016, November). AIP Publishing LLC.

\bibitem{KKA} Kosuge, S., Kuo, H. W., Aoki, K.: A kinetic model for a polyatomic gas with temperature-dependent specific heats and its application to shock-wave structure. Journal of Statistical Physics, 177(2) (2019), 209-251.


\bibitem{PR} Pareschi, L., Russo, G.: Implicit-explicit Runge-Kutta schemes for stiff systems of differential equations. Recent trends in numerical analysis, 3 (2000), 269-289.


\bibitem{Park sa jun PHD thesis} Park, S.: Mathematical studies on the ellipsoidal BGK model of the Boltzmann equation for polyatomic particles, Thesis (Ph.D.)-- Sungkyunkwan university : Department of Mathematics 2018. 8



\bibitem{PY} Park, S., Yun, S.-B.: "Cauchy problem for the ellipsoidal BGK model for polyatomic particles." Journal of Differential Equations, 266(11) (2019), 7678-7708.

\bibitem{PY1} Park, S., Yun, S.-B.: Entropy production estimates for the polyatomic ellipsoidal BGK model. Appl. Math. Lett. 58 (2016), 26–33.




\bibitem{PP} Pieraccini, S. and Puppo, G.: Implicit-explicit schemes for BGK kinetic equations,
J. Sci. Comput. 32 (2007) 1-28.


\bibitem{P} Pirner, M.: A BGK model for gas mixtures of polyatomic molecules allowing for slow and fast relaxation of the temperatures. Journal of Statistical Physics, 173(6) (2018), 1660-1687.



\bibitem{QC} Qiu, J. M., and Christlieb, A.: A conservative high order semi-Lagrangian WENO method for the Vlasov equation. Journal of Computational Physics, 229(4) (2010), 1130-1149.

\bibitem{QS} Qiu, J. M., and  Shu, C. W.: Conservative high order semi-Lagrangian finite difference WENO methods for advection in incompressible flow. Journal of Computational Physics, 230(4) (2011), 863-889.

\bibitem{RF} Russo, G. and Filbet, F.: Semilagrangian schemes applied to moving boundary problems for the BGK model of rarefied gas dynamics.

\bibitem{RS} Russo, G. and Santagati, P.: A new class of large time step methods for the BGK models of the
Boltzmann equation, arXiv:1103.5247v1, 2011.

\bibitem{RSY} Russo, G. and Santagati, P. and Yun, S.-B.: Convergence of a semi-Lagrangian scheme for the BGK model of the Boltzmann equation. SIAM J. Numer. Anal. 50 (2012), no. 3, 1111–1135.

\bibitem{RY} Russo, G. and Yun, S. B.: Convergence of a semi-Lagrangian scheme for the ellipsoidal BGK model of the Boltzmann equation. SIAM J. Numer. Anal., 56(6) (2018), 3580-3610.


\bibitem{SP} Santagati, P.: High order semi-Lagrangian schemes for the BGK model of the Boltzmann equation. Department of Mathematics and Computer Science, University of Catania. Diss. PhD. thesis, 2007.



\bibitem{SRB} Sonnendrücker, E., Roche, J., Bertrand, P., Ghizzo, A.: The semi-Lagrangian method for the numerical resolution of the Vlasov equation. Journal of computational physics, 149(2) (1999), 201-220.


\bibitem{TRQ} Xiong, T. , Russo, G. and Qiu, J. M.: Conservative Multi-Dimensional Semi-Lagrangian Finite Difference Scheme: Stability and Applications to the Kinetic and Fluid Simulations. arXiv:1607.07409.

\bibitem{Y} Yun, S.-B.: Entropy production for ellipsoidal BGK model of the Boltzmann equation. Kinet. Relat. Models 9 (2016), no. 3, 605–619.

\bibitem{Y1} Yun, S. B.: Ellipsoidal BGK model for polyatomic molecules near Maxwellians: A dichotomy in the dissipation estimate. Journal of Differential Equations, 266(9) (2019), 5566-5614.

\end{thebibliography}

\end{document}